\documentclass[10pt,reqno]{amsart}

\usepackage{pdfpages}
\usepackage{pdflscape} 

\usepackage{booktabs}
\usepackage{amssymb}
\usepackage{graphicx}
\usepackage{epstopdf}
\usepackage[hyphens]{url}
\usepackage[square,numbers,sort&compress]{natbib}

\usepackage{multirow}

\usepackage{tikz}
\definecolor{mediumspringgreen}{rgb}{0.0, 0.98039215, 0.60392156}

\usepackage{pgfplots} 

\usetikzlibrary{positioning}

\def\visible<#1>{}  

\usepackage[utf8]{inputenc}

\usepackage[english]{babel}
\usepackage{amsfonts}
\usepackage{amsmath}
\usepackage{latexsym}
\usepackage{color}
\usepackage{subfigure}
\usepackage{enumerate}

\usepackage{hyperref}  

\usepackage{ifpdf}
\usepackage{listings}
\DeclareMathOperator    \conv           {conv}

\DeclareMathOperator    \intr                   {int}

\DeclareMathOperator    \relint         {rel\,int}

\DeclareMathOperator    \verts          {vert}

\newcommand{\old}[1]{{}}

\newcommand{\bb}{\mathbb}

\newcommand{\R}{\bb R}
\newcommand{\Q}{\bb Q}
\newcommand{\Z}{\bb Z}
\newcommand{\N}{\bb N}

\newcommand\st{\mid}

\def\ve#1{\mathchoice{\mbox{\boldmath$\displaystyle\bf#1$}}
{\mbox{\boldmath$\textstyle\bf#1$}}
{\mbox{\boldmath$\scriptstyle\bf#1$}}
{\mbox{\boldmath$\scriptscriptstyle\bf#1$}}}

\newcommand{\setcond}[2]{\left\{ #1 \,\st\, #2 \right\}}

\renewcommand{\P}{\mathcal{P}}

\newcommand{\x}{{\ve x}}
\newcommand{\y}{{\ve y}}
\newcommand{\z}{{\ve z}}

\newcommand{\B}{B}

\def\st{\mid}

\newenvironment{psmallmatrixbig}{\bigl(\smallmatrix}{\endsmallmatrix\bigr)}
\newcommand\InlineFrac[2]{#1/#2}  
\newcommand\ColVec[3][\relax]
{
  \ifx#1\relax
  \bgroup\let\frac=\InlineFrac\begin{psmallmatrixbig}#2\vphantom{/}\\#3\vphantom{/}\end{psmallmatrixbig}\egroup
  \else
  \bgroup\let\frac=\InlineFrac\begin{psmallmatrixbig}\ifx#200\else#2/#1\fi\\\ifx#300\else#3/#1\fi\end{psmallmatrixbig}\egroup
  \fi
}

\newtheorem{theorem}{Theorem}[section]

\makeatletter
\newcommand\MkNewTheorem[2]{%
  \newtheorem{#1}{#2}
  \expandafter\def\csname c@#1\endcsname{\c@theorem}
  \expandafter\def\csname p@#1\endcsname{\p@theorem}
  \expandafter\def\csname the#1\endcsname{\thetheorem}
  \expandafter\def\csname #1name\endcsname{#2}
}

\MkNewTheorem{corollary}{Corollary}
\MkNewTheorem{lemma}{Lemma}
\MkNewTheorem{proposition}{Proposition}
\MkNewTheorem{prop}{Proposition}
\MkNewTheorem{claim}{Claim}
\MkNewTheorem{observation}{Observation}
\MkNewTheorem{obs}{Observation}
\MkNewTheorem{conjecture}{Conjecture}
\MkNewTheorem{openquestion}{Open question}
\MkNewTheorem{question}{Question}

\theoremstyle{definition}
\MkNewTheorem{example}{Example}
\MkNewTheorem{exercise}{Exercise}
\MkNewTheorem{notation}{Notation}
\MkNewTheorem{assumption}{Assumption}
\MkNewTheorem{definition}{Definition}
\MkNewTheorem{remark}{Remark}
\MkNewTheorem{goal}{Goal}
\MkNewTheorem{problem}{Problem}

\makeatletter
\let\OurMathBbAux=\mathbb
\DeclareRobustCommand\OurMathBb{\OurMathBbAux}
\let\mathbb=\OurMathBb
\let\bfseries=\undefined
\DeclareRobustCommand\bfseries
{\not@math@alphabet\bfseries\mathbf
  \boldmath\fontseries\bfdefault\selectfont\let\OurMathBbAux=\mathbf}
\def\@thm#1#2#3{%
  \ifhmode\unskip\unskip\par\fi
  \normalfont
  \trivlist
  \let\thmheadnl\relax
  \let\thm@swap\@gobble
  \thm@notefont{\fontseries\mddefault\upshape\unboldmath}
  \thm@headpunct{.}
  \thm@headsep 5\p@ plus\p@ minus\p@\relax
  \thm@space@setup
  #1
  \@topsep \thm@preskip               
  \@topsepadd \thm@postskip           
  \def\@tempa{#2}\ifx\@empty\@tempa
    \def\@tempa{\@oparg{\@begintheorem{#3}{}}[]}%
  \else
    \refstepcounter{#2}%
    \def\@tempa{\@oparg{\@begintheorem{#3}{\csname the#2\endcsname}}[]}%
  \fi
  \@tempa
}
\makeatother

\makeatletter
\renewcommand{\pod}[1]
{\allowbreak\mathchoice{\mkern18mu}{\mkern8mu}{\mkern8mu}{\mkern8mu}(#1)}
\makeatother

\let\epsilon=\varepsilon

\let\Myunderscore=\textunderscore   
\newcommand\underscore{\Myunderscore\allowbreak}
\let\_=\underscore

\newcommand\githubsearchurl{https://github.com/mkoeppe/infinite-group-relaxation-code/search}
\pgfkeyssetvalue{/sagefunc/bccz_counterexample}{\href{\githubsearchurl?q=\%22def+bccz_counterexample(\%22}{\sage{bccz\underscore{}counterexample}}}
\pgfkeyssetvalue{/sagefunc/bcdsp_arbitrary_slope}{\href{\githubsearchurl?q=\%22def+bcdsp_arbitrary_slope(\%22}{\sage{bcdsp\underscore{}arbitrary\underscore{}slope}}}
\pgfkeyssetvalue{/sagefunc/bhk_gmi_irrational}{\href{\githubsearchurl?q=\%22def+bhk_gmi_irrational(\%22}{\sage{bhk\underscore{}gmi\underscore{}irrational}}}
\pgfkeyssetvalue{/sagefunc/bhk_irrational}{\href{\githubsearchurl?q=\%22def+bhk_irrational(\%22}{\sage{bhk\underscore{}irrational}}}
\pgfkeyssetvalue{/sagefunc/bhk_slant_irrational}{\href{\githubsearchurl?q=\%22def+bhk_slant_irrational(\%22}{\sage{bhk\underscore{}slant\underscore{}irrational}}}
\pgfkeyssetvalue{/sagefunc/chen_4_slope}{\href{\githubsearchurl?q=\%22def+chen_4_slope(\%22}{\sage{chen\underscore{}4\underscore{}slope}}}
\pgfkeyssetvalue{/sagefunc/dg_2_step_mir}{\href{\githubsearchurl?q=\%22def+dg_2_step_mir(\%22}{\sage{dg\underscore{}2\underscore{}step\underscore{}mir}}}
\pgfkeyssetvalue{/sagefunc/dg_2_step_mir_limit}{\href{\githubsearchurl?q=\%22def+dg_2_step_mir_limit(\%22}{\sage{dg\underscore{}2\underscore{}step\underscore{}mir\underscore{}limit}}}
\pgfkeyssetvalue{/sagefunc/drlm_2_slope_limit}{\href{\githubsearchurl?q=\%22def+drlm_2_slope_limit(\%22}{\sage{drlm\underscore{}2\underscore{}slope\underscore{}limit}}}
\pgfkeyssetvalue{/sagefunc/drlm_3_slope_limit}{\href{\githubsearchurl?q=\%22def+drlm_3_slope_limit(\%22}{\sage{drlm\underscore{}3\underscore{}slope\underscore{}limit}}}
\pgfkeyssetvalue{/sagefunc/drlm_backward_3_slope}{\href{\githubsearchurl?q=\%22def+drlm_backward_3_slope(\%22}{\sage{drlm\underscore{}backward\underscore{}3\underscore{}slope}}}
\pgfkeyssetvalue{/sagefunc/gj_2_slope}{\href{\githubsearchurl?q=\%22def+gj_2_slope(\%22}{\sage{gj\underscore{}2\underscore{}slope}}}
\pgfkeyssetvalue{/sagefunc/gj_2_slope_repeat}{\href{\githubsearchurl?q=\%22def+gj_2_slope_repeat(\%22}{\sage{gj\underscore{}2\underscore{}slope\underscore{}repeat}}}
\pgfkeyssetvalue{/sagefunc/gj_forward_3_slope}{\href{\githubsearchurl?q=\%22def+gj_forward_3_slope(\%22}{\sage{gj\underscore{}forward\underscore{}3\underscore{}slope}}}
\pgfkeyssetvalue{/sagefunc/gmic}{\href{\githubsearchurl?q=\%22def+gmic(\%22}{\sage{gmic}}}
\pgfkeyssetvalue{/sagefunc/hildebrand_2_sided_discont_1_slope_1}{\href{\githubsearchurl?q=\%22def+hildebrand_2_sided_discont_1_slope_1(\%22}{\sage{hildebrand\underscore{}2\underscore{}sided\underscore{}discont\underscore{}1\underscore{}slope\underscore{}1}}}
\pgfkeyssetvalue{/sagefunc/hildebrand_2_sided_discont_2_slope_1}{\href{\githubsearchurl?q=\%22def+hildebrand_2_sided_discont_2_slope_1(\%22}{\sage{hildebrand\underscore{}2\underscore{}sided\underscore{}discont\underscore{}2\underscore{}slope\underscore{}1}}}
\pgfkeyssetvalue{/sagefunc/hildebrand_5_slope_22_1}{\href{\githubsearchurl?q=\%22def+hildebrand_5_slope_22_1(\%22}{\sage{hildebrand\underscore{}5\underscore{}slope\underscore{}22\underscore{}1}}}
\pgfkeyssetvalue{/sagefunc/hildebrand_5_slope_24_1}{\href{\githubsearchurl?q=\%22def+hildebrand_5_slope_24_1(\%22}{\sage{hildebrand\underscore{}5\underscore{}slope\underscore{}24\underscore{}1}}}
\pgfkeyssetvalue{/sagefunc/hildebrand_5_slope_28_1}{\href{\githubsearchurl?q=\%22def+hildebrand_5_slope_28_1(\%22}{\sage{hildebrand\underscore{}5\underscore{}slope\underscore{}28\underscore{}1}}}
\pgfkeyssetvalue{/sagefunc/hildebrand_discont_3_slope_1}{\href{\githubsearchurl?q=\%22def+hildebrand_discont_3_slope_1(\%22}{\sage{hildebrand\underscore{}discont\underscore{}3\underscore{}slope\underscore{}1}}}
\pgfkeyssetvalue{/sagefunc/kf_n_step_mir}{\href{\githubsearchurl?q=\%22def+kf_n_step_mir(\%22}{\sage{kf\underscore{}n\underscore{}step\underscore{}mir}}}
\pgfkeyssetvalue{/sagefunc/kzh_10_slope_1}{\href{\githubsearchurl?q=\%22def+kzh_10_slope_1(\%22}{\sage{kzh\underscore{}10\underscore{}slope\underscore{}1}}}
\pgfkeyssetvalue{/sagefunc/kzh_28_slope_1}{\href{\githubsearchurl?q=\%22def+kzh_28_slope_1(\%22}{\sage{kzh\underscore{}28\underscore{}slope\underscore{}1}}}
\pgfkeyssetvalue{/sagefunc/kzh_28_slope_2}{\href{\githubsearchurl?q=\%22def+kzh_28_slope_2(\%22}{\sage{kzh\underscore{}28\underscore{}slope\underscore{}2}}}
\pgfkeyssetvalue{/sagefunc/kzh_3_slope_param_extreme_1}{\href{\githubsearchurl?q=\%22def+kzh_3_slope_param_extreme_1(\%22}{\sage{kzh\underscore{}3\underscore{}slope\underscore{}param\underscore{}extreme\underscore{}1}}}
\pgfkeyssetvalue{/sagefunc/kzh_3_slope_param_extreme_2}{\href{\githubsearchurl?q=\%22def+kzh_3_slope_param_extreme_2(\%22}{\sage{kzh\underscore{}3\underscore{}slope\underscore{}param\underscore{}extreme\underscore{}2}}}
\pgfkeyssetvalue{/sagefunc/kzh_4_slope_param_extreme_1}{\href{\githubsearchurl?q=\%22def+kzh_4_slope_param_extreme_1(\%22}{\sage{kzh\underscore{}4\underscore{}slope\underscore{}param\underscore{}extreme\underscore{}1}}}
\pgfkeyssetvalue{/sagefunc/kzh_5_slope_fulldim_1}{\href{\githubsearchurl?q=\%22def+kzh_5_slope_fulldim_1(\%22}{\sage{kzh\underscore{}5\underscore{}slope\underscore{}fulldim\underscore{}1}}}
\pgfkeyssetvalue{/sagefunc/kzh_5_slope_fulldim_2}{\href{\githubsearchurl?q=\%22def+kzh_5_slope_fulldim_2(\%22}{\sage{kzh\underscore{}5\underscore{}slope\underscore{}fulldim\underscore{}2}}}
\pgfkeyssetvalue{/sagefunc/kzh_5_slope_fulldim_3}{\href{\githubsearchurl?q=\%22def+kzh_5_slope_fulldim_3(\%22}{\sage{kzh\underscore{}5\underscore{}slope\underscore{}fulldim\underscore{}3}}}
\pgfkeyssetvalue{/sagefunc/kzh_5_slope_fulldim_4}{\href{\githubsearchurl?q=\%22def+kzh_5_slope_fulldim_4(\%22}{\sage{kzh\underscore{}5\underscore{}slope\underscore{}fulldim\underscore{}4}}}
\pgfkeyssetvalue{/sagefunc/kzh_5_slope_fulldim_5}{\href{\githubsearchurl?q=\%22def+kzh_5_slope_fulldim_5(\%22}{\sage{kzh\underscore{}5\underscore{}slope\underscore{}fulldim\underscore{}5}}}
\pgfkeyssetvalue{/sagefunc/kzh_5_slope_fulldim_covers_1}{\href{\githubsearchurl?q=\%22def+kzh_5_slope_fulldim_covers_1(\%22}{\sage{kzh\underscore{}5\underscore{}slope\underscore{}fulldim\underscore{}covers\underscore{}1}}}
\pgfkeyssetvalue{/sagefunc/kzh_5_slope_fulldim_covers_2}{\href{\githubsearchurl?q=\%22def+kzh_5_slope_fulldim_covers_2(\%22}{\sage{kzh\underscore{}5\underscore{}slope\underscore{}fulldim\underscore{}covers\underscore{}2}}}
\pgfkeyssetvalue{/sagefunc/kzh_5_slope_fulldim_covers_3}{\href{\githubsearchurl?q=\%22def+kzh_5_slope_fulldim_covers_3(\%22}{\sage{kzh\underscore{}5\underscore{}slope\underscore{}fulldim\underscore{}covers\underscore{}3}}}
\pgfkeyssetvalue{/sagefunc/kzh_5_slope_fulldim_covers_4}{\href{\githubsearchurl?q=\%22def+kzh_5_slope_fulldim_covers_4(\%22}{\sage{kzh\underscore{}5\underscore{}slope\underscore{}fulldim\underscore{}covers\underscore{}4}}}
\pgfkeyssetvalue{/sagefunc/kzh_5_slope_fulldim_covers_5}{\href{\githubsearchurl?q=\%22def+kzh_5_slope_fulldim_covers_5(\%22}{\sage{kzh\underscore{}5\underscore{}slope\underscore{}fulldim\underscore{}covers\underscore{}5}}}
\pgfkeyssetvalue{/sagefunc/kzh_5_slope_fulldim_covers_6}{\href{\githubsearchurl?q=\%22def+kzh_5_slope_fulldim_covers_6(\%22}{\sage{kzh\underscore{}5\underscore{}slope\underscore{}fulldim\underscore{}covers\underscore{}6}}}
\pgfkeyssetvalue{/sagefunc/kzh_5_slope_q22_f10_1}{\href{\githubsearchurl?q=\%22def+kzh_5_slope_q22_f10_1(\%22}{\sage{kzh\underscore{}5\underscore{}slope\underscore{}q22\underscore{}f10\underscore{}1}}}
\pgfkeyssetvalue{/sagefunc/kzh_5_slope_q22_f10_2}{\href{\githubsearchurl?q=\%22def+kzh_5_slope_q22_f10_2(\%22}{\sage{kzh\underscore{}5\underscore{}slope\underscore{}q22\underscore{}f10\underscore{}2}}}
\pgfkeyssetvalue{/sagefunc/kzh_5_slope_q22_f10_3}{\href{\githubsearchurl?q=\%22def+kzh_5_slope_q22_f10_3(\%22}{\sage{kzh\underscore{}5\underscore{}slope\underscore{}q22\underscore{}f10\underscore{}3}}}
\pgfkeyssetvalue{/sagefunc/kzh_5_slope_q22_f10_4}{\href{\githubsearchurl?q=\%22def+kzh_5_slope_q22_f10_4(\%22}{\sage{kzh\underscore{}5\underscore{}slope\underscore{}q22\underscore{}f10\underscore{}4}}}
\pgfkeyssetvalue{/sagefunc/kzh_5_slope_q22_f2_1}{\href{\githubsearchurl?q=\%22def+kzh_5_slope_q22_f2_1(\%22}{\sage{kzh\underscore{}5\underscore{}slope\underscore{}q22\underscore{}f2\underscore{}1}}}
\pgfkeyssetvalue{/sagefunc/kzh_6_slope_1}{\href{\githubsearchurl?q=\%22def+kzh_6_slope_1(\%22}{\sage{kzh\underscore{}6\underscore{}slope\underscore{}1}}}
\pgfkeyssetvalue{/sagefunc/kzh_6_slope_fulldim_covers_1}{\href{\githubsearchurl?q=\%22def+kzh_6_slope_fulldim_covers_1(\%22}{\sage{kzh\underscore{}6\underscore{}slope\underscore{}fulldim\underscore{}covers\underscore{}1}}}
\pgfkeyssetvalue{/sagefunc/kzh_6_slope_fulldim_covers_2}{\href{\githubsearchurl?q=\%22def+kzh_6_slope_fulldim_covers_2(\%22}{\sage{kzh\underscore{}6\underscore{}slope\underscore{}fulldim\underscore{}covers\underscore{}2}}}
\pgfkeyssetvalue{/sagefunc/kzh_6_slope_fulldim_covers_3}{\href{\githubsearchurl?q=\%22def+kzh_6_slope_fulldim_covers_3(\%22}{\sage{kzh\underscore{}6\underscore{}slope\underscore{}fulldim\underscore{}covers\underscore{}3}}}
\pgfkeyssetvalue{/sagefunc/kzh_6_slope_fulldim_covers_4}{\href{\githubsearchurl?q=\%22def+kzh_6_slope_fulldim_covers_4(\%22}{\sage{kzh\underscore{}6\underscore{}slope\underscore{}fulldim\underscore{}covers\underscore{}4}}}
\pgfkeyssetvalue{/sagefunc/kzh_6_slope_fulldim_covers_5}{\href{\githubsearchurl?q=\%22def+kzh_6_slope_fulldim_covers_5(\%22}{\sage{kzh\underscore{}6\underscore{}slope\underscore{}fulldim\underscore{}covers\underscore{}5}}}
\pgfkeyssetvalue{/sagefunc/kzh_7_slope_1}{\href{\githubsearchurl?q=\%22def+kzh_7_slope_1(\%22}{\sage{kzh\underscore{}7\underscore{}slope\underscore{}1}}}
\pgfkeyssetvalue{/sagefunc/kzh_7_slope_2}{\href{\githubsearchurl?q=\%22def+kzh_7_slope_2(\%22}{\sage{kzh\underscore{}7\underscore{}slope\underscore{}2}}}
\pgfkeyssetvalue{/sagefunc/kzh_7_slope_3}{\href{\githubsearchurl?q=\%22def+kzh_7_slope_3(\%22}{\sage{kzh\underscore{}7\underscore{}slope\underscore{}3}}}
\pgfkeyssetvalue{/sagefunc/kzh_7_slope_4}{\href{\githubsearchurl?q=\%22def+kzh_7_slope_4(\%22}{\sage{kzh\underscore{}7\underscore{}slope\underscore{}4}}}
\pgfkeyssetvalue{/sagefunc/ll_strong_fractional}{\href{\githubsearchurl?q=\%22def+ll_strong_fractional(\%22}{\sage{ll\underscore{}strong\underscore{}fractional}}}
\pgfkeyssetvalue{/sagefunc/mlr_cpl3_a_2_slope}{\href{\githubsearchurl?q=\%22def+mlr_cpl3_a_2_slope(\%22}{\sage{mlr\underscore{}cpl3\underscore{}a\underscore{}2\underscore{}slope}}}
\pgfkeyssetvalue{/sagefunc/mlr_cpl3_b_3_slope}{\href{\githubsearchurl?q=\%22def+mlr_cpl3_b_3_slope(\%22}{\sage{mlr\underscore{}cpl3\underscore{}b\underscore{}3\underscore{}slope}}}
\pgfkeyssetvalue{/sagefunc/mlr_cpl3_c_3_slope}{\href{\githubsearchurl?q=\%22def+mlr_cpl3_c_3_slope(\%22}{\sage{mlr\underscore{}cpl3\underscore{}c\underscore{}3\underscore{}slope}}}
\pgfkeyssetvalue{/sagefunc/mlr_cpl3_d_3_slope}{\href{\githubsearchurl?q=\%22def+mlr_cpl3_d_3_slope(\%22}{\sage{mlr\underscore{}cpl3\underscore{}d\underscore{}3\underscore{}slope}}}
\pgfkeyssetvalue{/sagefunc/mlr_cpl3_f_2_or_3_slope}{\href{\githubsearchurl?q=\%22def+mlr_cpl3_f_2_or_3_slope(\%22}{\sage{mlr\underscore{}cpl3\underscore{}f\underscore{}2\underscore{}or\underscore{}3\underscore{}slope}}}
\pgfkeyssetvalue{/sagefunc/mlr_cpl3_g_3_slope}{\href{\githubsearchurl?q=\%22def+mlr_cpl3_g_3_slope(\%22}{\sage{mlr\underscore{}cpl3\underscore{}g\underscore{}3\underscore{}slope}}}
\pgfkeyssetvalue{/sagefunc/mlr_cpl3_h_2_slope}{\href{\githubsearchurl?q=\%22def+mlr_cpl3_h_2_slope(\%22}{\sage{mlr\underscore{}cpl3\underscore{}h\underscore{}2\underscore{}slope}}}
\pgfkeyssetvalue{/sagefunc/mlr_cpl3_k_2_slope}{\href{\githubsearchurl?q=\%22def+mlr_cpl3_k_2_slope(\%22}{\sage{mlr\underscore{}cpl3\underscore{}k\underscore{}2\underscore{}slope}}}
\pgfkeyssetvalue{/sagefunc/mlr_cpl3_l_2_slope}{\href{\githubsearchurl?q=\%22def+mlr_cpl3_l_2_slope(\%22}{\sage{mlr\underscore{}cpl3\underscore{}l\underscore{}2\underscore{}slope}}}
\pgfkeyssetvalue{/sagefunc/mlr_cpl3_n_3_slope}{\href{\githubsearchurl?q=\%22def+mlr_cpl3_n_3_slope(\%22}{\sage{mlr\underscore{}cpl3\underscore{}n\underscore{}3\underscore{}slope}}}
\pgfkeyssetvalue{/sagefunc/mlr_cpl3_o_2_slope}{\href{\githubsearchurl?q=\%22def+mlr_cpl3_o_2_slope(\%22}{\sage{mlr\underscore{}cpl3\underscore{}o\underscore{}2\underscore{}slope}}}
\pgfkeyssetvalue{/sagefunc/mlr_cpl3_p_2_slope}{\href{\githubsearchurl?q=\%22def+mlr_cpl3_p_2_slope(\%22}{\sage{mlr\underscore{}cpl3\underscore{}p\underscore{}2\underscore{}slope}}}
\pgfkeyssetvalue{/sagefunc/mlr_cpl3_q_2_slope}{\href{\githubsearchurl?q=\%22def+mlr_cpl3_q_2_slope(\%22}{\sage{mlr\underscore{}cpl3\underscore{}q\underscore{}2\underscore{}slope}}}
\pgfkeyssetvalue{/sagefunc/mlr_cpl3_r_2_slope}{\href{\githubsearchurl?q=\%22def+mlr_cpl3_r_2_slope(\%22}{\sage{mlr\underscore{}cpl3\underscore{}r\underscore{}2\underscore{}slope}}}
\pgfkeyssetvalue{/sagefunc/rlm_dpl1_extreme_3a}{\href{\githubsearchurl?q=\%22def+rlm_dpl1_extreme_3a(\%22}{\sage{rlm\underscore{}dpl1\underscore{}extreme\underscore{}3a}}}
\pgfkeyssetvalue{/sagefunc/automorphism}{\href{\githubsearchurl?q=\%22def+automorphism(\%22}{\sage{automorphism}}}
\pgfkeyssetvalue{/sagefunc/multiplicative_homomorphism}{\href{\githubsearchurl?q=\%22def+multiplicative_homomorphism(\%22}{\sage{multiplicative\underscore{}homomorphism}}}
\pgfkeyssetvalue{/sagefunc/projected_sequential_merge}{\href{\githubsearchurl?q=\%22def+projected_sequential_merge(\%22}{\sage{projected\underscore{}sequential\underscore{}merge}}}
\pgfkeyssetvalue{/sagefunc/restrict_to_finite_group}{\href{\githubsearchurl?q=\%22def+restrict_to_finite_group(\%22}{\sage{restrict\underscore{}to\underscore{}finite\underscore{}group}}}
\pgfkeyssetvalue{/sagefunc/interpolate_to_infinite_group}{\href{\githubsearchurl?q=\%22def+interpolate_to_infinite_group(\%22}{\sage{interpolate\underscore{}to\underscore{}infinite\underscore{}group}}}
\pgfkeyssetvalue{/sagefunc/two_slope_fill_in}{\href{\githubsearchurl?q=\%22def+two_slope_fill_in(\%22}{\sage{two\underscore{}slope\underscore{}fill\underscore{}in}}}
\pgfkeyssetvalue{/sagefunc/generate_example_e_for_psi_n}{\href{\githubsearchurl?q=\%22def+generate_example_e_for_psi_n(\%22}{\sage{generate\underscore{}example\underscore{}e\underscore{}for\underscore{}psi\underscore{}n}}}
\pgfkeyssetvalue{/sagefunc/chen_3_slope_not_extreme}{\href{\githubsearchurl?q=\%22def+chen_3_slope_not_extreme(\%22}{\sage{chen\underscore{}3\underscore{}slope\underscore{}not\underscore{}extreme}}}
\pgfkeyssetvalue{/sagefunc/psi_n_in_bccz_counterexample_construction}{\href{\githubsearchurl?q=\%22def+psi_n_in_bccz_counterexample_construction(\%22}{\sage{psi\underscore{}n\underscore{}in\underscore{}bccz\underscore{}counterexample\underscore{}construction}}}
\pgfkeyssetvalue{/sagefunc/gomory_fractional}{\href{\githubsearchurl?q=\%22def+gomory_fractional(\%22}{\sage{gomory\underscore{}fractional}}}
\pgfkeyssetvalue{/sagefunc/not_minimal_2}{\href{\githubsearchurl?q=\%22def+not_minimal_2(\%22}{\sage{not\underscore{}minimal\underscore{}2}}}
\pgfkeyssetvalue{/sagefunc/not_extreme_1}{\href{\githubsearchurl?q=\%22def+not_extreme_1(\%22}{\sage{not\underscore{}extreme\underscore{}1}}}
\pgfkeyssetvalue{/sagefunc/kzh_2q_example_1}{\href{\githubsearchurl?q=\%22def+kzh_2q_example_1(\%22}{\sage{kzh\underscore{}2q\underscore{}example\underscore{}1}}}
\pgfkeyssetvalue{/sagefunc/zhou_two_sided_discontinuous_cannot_assume_any_continuity}{\href{\githubsearchurl?q=\%22def+zhou_two_sided_discontinuous_cannot_assume_any_continuity(\%22}{\sage{zhou\underscore{}two\underscore{}sided\underscore{}discontinuous\underscore{}cannot\underscore{}assume\underscore{}any\underscore{}continuity}}}
\pgfkeyssetvalue{/sagefunc/extremality_test}{\href{\githubsearchurl?q=\%22def+extremality_test(\%22}{\sage{extremality\underscore{}test}}}
\pgfkeyssetvalue{/sagefunc/plot_2d_diagram}{\href{\githubsearchurl?q=\%22def+plot_2d_diagram(\%22}{\sage{plot\underscore{}2d\underscore{}diagram}}}
\pgfkeyssetvalue{/sagefunc/nice_field_values}{\href{\githubsearchurl?q=\%22def+nice_field_values(\%22}{\sage{nice\underscore{}field\underscore{}values}}}
\pgfkeyssetvalue{/sagefunc/ParametricRealFieldElement}{\href{\githubsearchurl?q=\%22def+ParametricRealFieldElement(\%22}{\sage{ParametricRealFieldElement}}}
\pgfkeyssetvalue{/sagefunc/ParametricRealField}{\href{\githubsearchurl?q=\%22def+ParametricRealField(\%22}{\sage{ParametricRealField}}}
\pgfkeyssetvalue{/sagefunc/kzh_minimal_has_only_crazy_perturbation_1}{\href{\githubsearchurl?q=\%22def+kzh_minimal_has_only_crazy_perturbation_1(\%22}{\sage{kzh\underscore{}minimal\underscore{}has\underscore{}only\underscore{}crazy\underscore{}perturbation\underscore{}1}}}

\DeclareRobustCommand\sage[1]{\textsf{\upshape #1}}
\DeclareRobustCommand\sagefunc[1]{\pgfkeys{/sagefunc/#1}}
\DeclareRobustCommand\sagefuncgraph[1]{\raisebox{-0.08ex}{\includegraphics[height=2ex,width=2.5em]{funcgraphs/#1}}}
\DeclareRobustCommand\sagefuncwithgraph[1]{\sagefunc{#1} \sagefuncgraph{#1}}

\DeclareRobustCommand\sagefuncwithgraphgomoryfractional{\sagefunc{gomory_fractional}\ \smash{\raisebox{-0.08ex}{\includegraphics[height=2.65ex,width=2.5em]{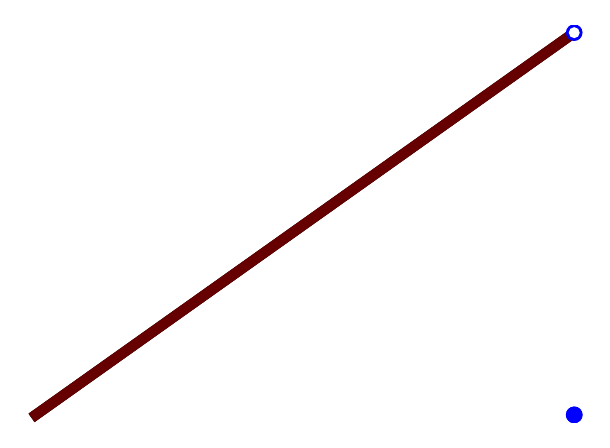}}}}

\title[Equivariant Perturbation VI. Two-Sided Discontinuous
Functions]{Equivariant Perturbation in \\Gomory and Johnson's Infinite Group 
  Problem.\\[1ex] VI. The Curious Case of Two-Sided Discontinuous Minimal Valid
  Functions}

\thanks{The authors gratefully acknowledge partial support from the National Science
  Foundation through grant DMS-1320051, awarded to M.~K\"oppe.}

\author{Matthias K\"oppe}
\address{Matthias K\"oppe: Dept.\ of Mathematics, University of California, Davis}
\email{mkoeppe@math.ucdavis.edu}

\author{Yuan Zhou} 
\address{Yuan Zhou: Dept.\ of Mathematics, University of Kentucky}
\email{yuan.zhou@uky.edu}

\date{$\relax$Revision: 2348 $ - \ $Date: 2018-01-28 20:35:05 -0800 (Sun, 28 Jan 2018) $ $\!\!\!}

\usepackage[normalem]{ulem}

\newcommand\Figure[2][\relax]{%
  \begin{figure}[h!]
    \includegraphics[width=.8\textwidth]{#2}
    \caption{\ifx#1\relax#2\else#1\fi}
  \end{figure}
}

\begin{document}
 \newcommand{\tgreen}[1]{\textsf{\textcolor {ForestGreen} {#1}}}
 \newcommand{\tred}[1]{\texttt{\textcolor {red} {#1}}}
 \newcommand{\tblue}[1]{\textcolor {blue} {#1}}

\begin{abstract}
  We construct a two-sided discontinuous piecewise linear minimal valid
  cut-generating function for the 1-row Gomory--Johnson model which is not extreme, but which
  is not a convex combination of other piecewise linear minimal valid
  functions.  
  The new function only admits piecewise microperiodic perturbations.  We
  present an algorithm for verifying certificates of non-extremality in the
  form of such perturbations.
\end{abstract}
\maketitle


\section{Introduction}

\subsection{Continuous and discontinuous functions related to corner polyhedra}

Gomory and Johnson, in their seminal papers \cite{infinite,infinite2} titled
\emph{Some \textbf{continuous functions} related to corner polyhedra I, II},
introduced piecewise linear functions that are related to Gomory's group
relaxation \cite{gom} of integer linear optimization problems.  Let
$f\in (0,1)$ be a fixed number and consider the solutions in
non-negative integer variables $y_j$ to an equation (the \emph{group
  relaxation}) of the form
\begin{equation}
  \sum_{j = 1}^m r_j y_j \equiv f \pmod1,\label{eq:source-row-relaxation}
\end{equation}
where $r_j\in\R$ are given coefficients.
A \emph{cutting plane}, or \emph{valid inequality}, is a linear inequality
that is satisfied by all non-negative integer solutions~$y$.  A classical
method of deriving cutting planes, the \emph{Gomory mixed-integer cut}, can be
expressed as follows.  Consider the function $\pi = \sagefuncwithgraph{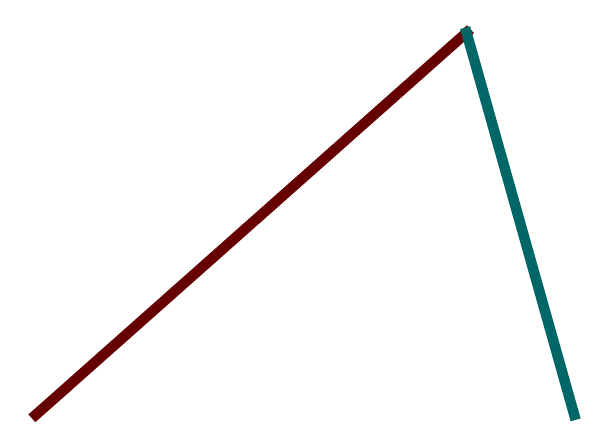}\footnote{A function name shown 
  in sans serif font is the name of the constructor of this function in the
  Electronic Compendium of Extreme Functions, part of the SageMath
  program~\cite{infinite-group-relaxation-code}.
  In an online copy of this paper, there are hyperlinks that lead to a search
  for this function in the GitHub repository.
  After the name of a function, we show a
    sparkline (inline graph) of the function on the fundamental domain $[0,1]$ as a quick
    reference.}$
as the piecewise linear function that is the $\Z$-periodic extension of the
linear interpolation of $\pi(0) = 0$, $\pi(f) = 1$, and $\pi(1)=0$.  Then 
\begin{equation}
  \sum_{j=1}^m \pi(r_j) y_j \geq 1
\end{equation}
is a valid inequality.  What is remarkable is that the function $\pi$ only
depends on a single parameter (the right-hand side~$f$) and can be applied to
any equation of the form~\eqref{eq:source-row-relaxation}, no matter what the
given coefficients $r_j$ are; moreover, it provides the coefficients
$\pi(r_j)$ of the cutting plane independently of each other.  Nowadays,
following Conforti et al.~\cite{conforti2013cut}, we call functions~$\pi$ of
this type \emph{cut-generating functions}.  

Another classical cutting plane, the \emph{Gomory fractional cut}, can be
described using this pattern.  Its cut-generating function
$\pi = \sagefuncwithgraphgomoryfractional{}$, considered as a function from
the reals to the reals, is a sawtooth function,
discontinuous at all integers.  However, Gomory and Johnson considered these
functions as going from the interval $[0, 1)$
, essentially removing the discontinuities from any further analysis.
Besides, in the hierarchy of cut-generating functions, arranged by increasing strength
from \emph{valid} over \emph{subadditive} and \emph{minimal} to \emph{extreme}
and \emph{facet}, the
\sagefunc{gomory_fractional} function belongs to the category of 
merely \emph{subadditive} functions.  As is well-known, it is dominated by the Gomory
mixed-integer cut, whose cut-generating function \sagefunc{gmic} is a
continuous extreme function.  This may explain Gomory and Johnson's focus
on the \emph{continuous functions} of their papers' titles.  In their papers,
they gave a full characterization of the minimal functions (they are the
$\Z$-periodic, subadditive functions $\pi\colon \R\to\R_+$ with $\pi(0)=0$ that
satisfy the \emph{symmetry condition} $\pi(x) + \pi(f - x) = 1$ for all
$x\in\R$).  Among the minimal functions, a function is extreme if it 
cannot be written as a convex combination of two other minimal functions.
Gomory and Johnson initiated a classification program for (continuous) extreme
functions, an early success of which was the two-slope theorem,
asserting that every continuous piecewise linear minimal function whose
derivatives take only two values is already extreme. This was a vast
generalization of the extremality of the Gomory mixed-integer cut. In parts of the later
literature, the related notion of facets, instead of extreme functions, was
considered; see \cite{koeppe-zhou:discontinuous-facets}.

Undisputably \textbf{discontinuous functions} came into play 30 years later, when
Letch\-ford--Lodi \cite{Letchford-Lodi-2002} introduced their \emph{strong
  fractional cut} (\sagefuncwithgraph{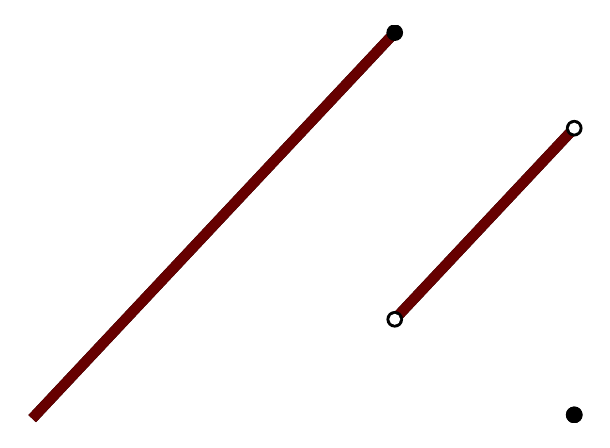}) as a strengthening
of the \sagefunc{gomory_fractional} cut.  (It neither dominates nor is
dominated by the \sagefunc{gmic} function.)
Letchford and Lodi first prove, by elementary means, that their new cutting plane gives a 
valid inequality; then they remark:
\begin{quote}
  \emph{[The] function mapping the coefficients of [the source row] onto the
    coefficients in the strong fractional cut [\dots]\@ can be shown to be
    subadditive. It also meets the other conditions of [Gomory--Johnson's
    characterization of minimal valid functions]. However, it differs from the
    subadditive functions given in [Gomory--Johnson's papers] in that it is
    discontinuous.}
\end{quote}

Further study of discontinuous functions took place in the context of
pointwise limits of continuous functions.  Dash and G\"unl\"uk~\cite{twoStepMIR}
introduced \emph{extended two-step MIR (mixed integer rounding) inequalities}
(\sagefuncwithgraph{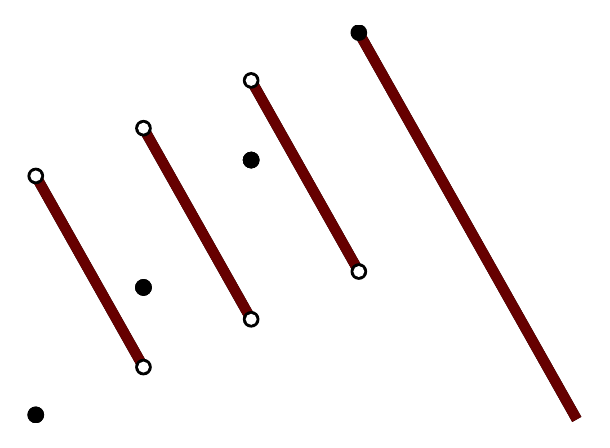}), whose corresponding cut-generating functions 
arise as limits of sequences of two-step MIR functions
(\sagefuncwithgraph{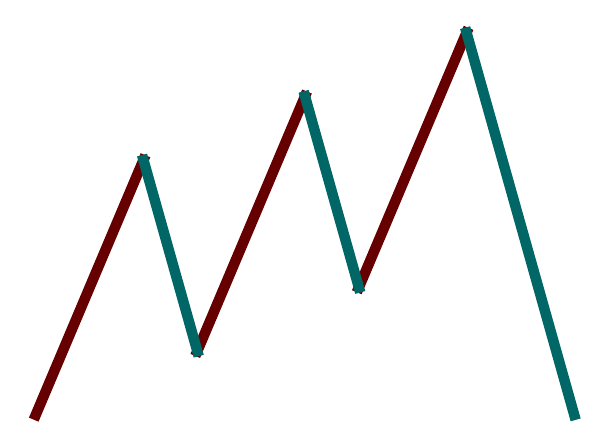}) defined in the same
paper.  They showed that these functions, up to \sagefunc{automorphism}, give
cutting planes that dominate the Letchford--Lodi strong fractional cuts.
Dey, Richard, Li, and Miller \cite{dey1} were the first to consider discontinuous
functions as first-class members of the Gomory--Johnson hierarchy of valid
functions, and introduced important tools for their study.  
They identified Dash and G\"unl\"uk's family
\sagefunc{dg_2_step_mir_limit} as extreme functions, 
introduced the enclosing family  \sagefuncwithgraph{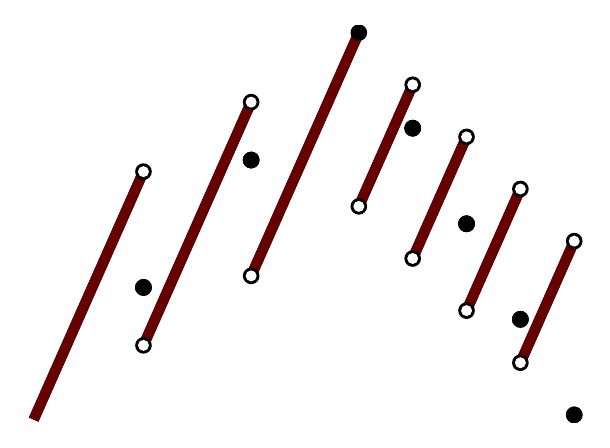}, 
and defined another family of discontinuous functions,
\sagefuncwithgraph{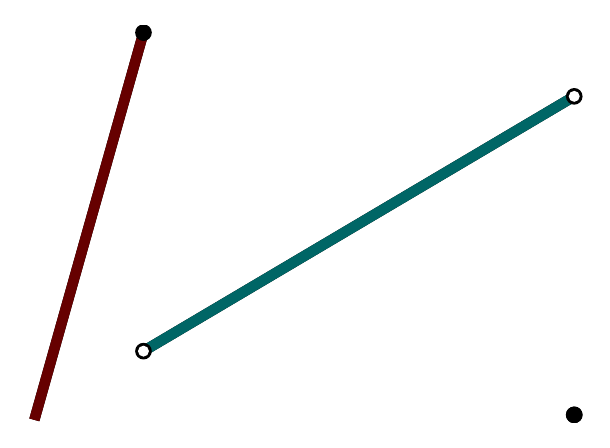}.   
Later Richard, Li, and Miller
\cite{Richard-Li-Miller-2009:Approximate-Liftings} conducted a systematic
study of discontinuous functions via the connection to superadditive lifting
functions, which brought examples such as
\sagefuncwithgraph{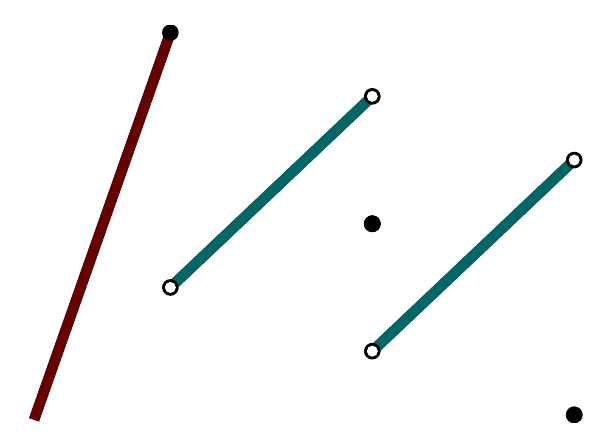}.

\subsection{The r\^ole of one-sided continuity in testing extremality}
Note that all of the above-mentioned families of extreme functions are
one-sided continuous at the origin, either from the left or from the right.
To explain the significance of this observation, we will outline the structure
of an extremality proof, using the notion of effective perturbations
introduced
in~\cite{hong-koeppe-zhou:software-abstract,hong-koeppe-zhou:software-paper}.
Let $\pi\colon \R \to \R_+$ be a minimal valid function.  Suppose
$\pi = \frac12(\pi^1+\pi^2)$ where $\pi^1$ and $\pi^2$ are valid functions;
then $\pi^1$ and $\pi^2$ are minimal.  Write
$\pi^1 = \pi + \epsilon \tilde\pi$ and $\pi^2 = \pi - \epsilon\tilde\pi$, where
$\tilde\pi\colon \R\to\R$ and $\epsilon>0$; then we call $\tilde\pi$ an
\emph{effective perturbation function}.  Towards the goal of showing
$\tilde\pi = 0$, one asks the following.
\begin{question}\label{q:properties-eff-perturb}
  Given a minimal function $\pi$, what properties does an effective
  perturbation $\tilde\pi$ necessarily have?
\end{question}
Later in this paper we will review answers to this question, which include the
boundedness of~$\tilde\pi$, the inheritance of additivity (see
\autoref{lemma:tight-implies-tight} below) and linearity properties (see
\autoref{thm:directly_covered} and \autoref{thm:indirectly_covered} below).
For now, we will focus on the following important regularity lemma by Dey,
Richard, Li, and Miller \cite{dey1}, which makes an assumption of one-sided
continuity at the origin.  It is a consequence of subadditivity of minimal
functions and serves as a crucial ingredient of many extremality proofs.
\begin{lemma}[{\cite[Theorem 2]{dey1}; see \cite[Lemma 2.11 (v)]{igp_survey}}]
  \label{survey-lemma-2.11(v)}
  Let $\pi \colon \R \to \R_+$ be a minimal valid function and $\tilde\pi$ an
  effective perturbation. 
  If $\pi$ is piecewise linear and continuous from
  the right at $0$ or 
  from the left at~$0$, 
  then $\tilde\pi$ is continuous at all points at which $\pi$ is continuous.
\end{lemma}

Hildebrand (2013, unpublished; see \cite{igp_survey}),
constructed the first examples of extreme functions that
are \textbf{two-sided discontinuous} at the origin, \sagefuncwithgraph{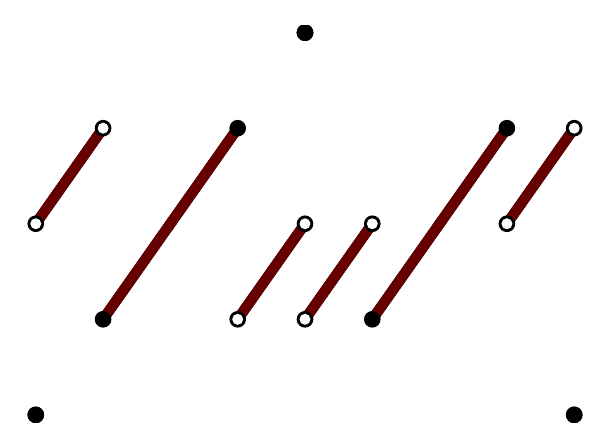},
\sagefuncwithgraph{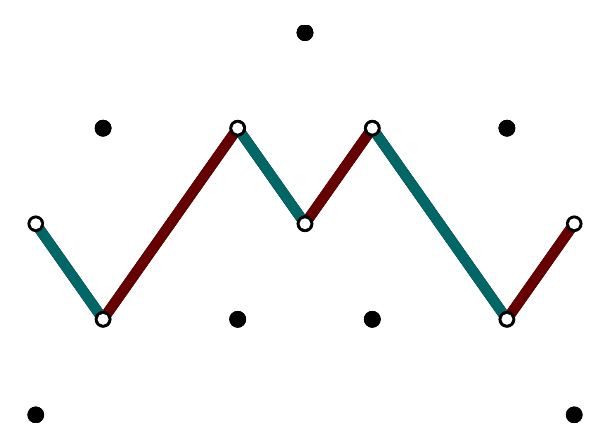}. Their extremality
proofs do not depend on \autoref{survey-lemma-2.11(v)}.  
Later, Zhou \cite{zhou:dissertation} constructed the first example,
\sagefuncwithgraph{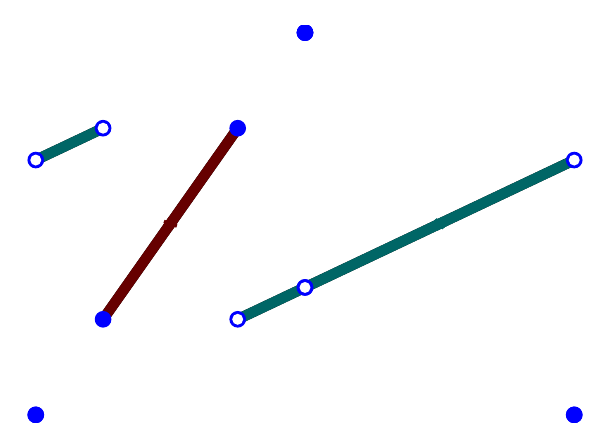}
(\autoref{fig:two_sided_discontinuous_cannot_assume_any_continuity}), that 
demonstrates that the 
hypothesis of one-sided continuity at the origin cannot be
removed from \autoref{survey-lemma-2.11(v)}.  

\begin{figure}[t]
\begin{center}
\includegraphics[width=.44\linewidth]{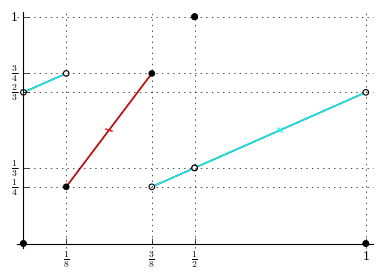}\quad
\includegraphics[width=.44\linewidth]{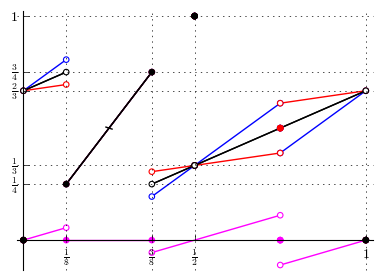}
\end{center}
\caption{This function, \sage{$\pi$ = \sagefunc{zhou_two_sided_discontinuous_cannot_assume_any_continuity}}, is
  minimal, but not extreme
  , as proved by
  \sage{\sagefunc{extremality_test}($\pi$, show\underscore{}plots=True)}.
  The procedure first shows that
  for any distinct minimal $\pi^1 = \pi + \bar\pi$ (\emph{blue}), $\pi^2 = \pi
  - \bar\pi$ (\emph{red}) such that $\pi = \tfrac{1}{2}\pi^1
  + \tfrac{1}{2} \pi^2$, the functions $\pi^1$ and $\pi^2$ are
  piecewise linear with the same breakpoints as $\pi$ and possible additional
  breakpoints at~$\frac14$ and $\frac34$.   The open intervals between these
  breakpoints are covered (see \autoref{s:preliminaries}, after
  \autoref{thm:directly_covered}, for this notion). 
  A finite-dimensional extremality
  test then finds exactly one linearly independent perturbation $\bar\pi$
  (\emph{magenta}), as shown. Thus all nontrivial perturbations are discontinuous at
  $\frac{3}{4}$, a point where $\pi$ is continuous.
}
\label{fig:two_sided_discontinuous_cannot_assume_any_continuity}
\end{figure}
\smallbreak

The breakdown of the regularity lemma in the two-sided discontinuous case
poses a challenge for the algorithmic theory of extreme functions.  Consider
the following variant of \autoref{q:properties-eff-perturb}:
\begin{question}\label{q:certify-nonextremality}
  \textup{(a)}~What class of effective perturbations $\tilde\pi$ is sufficient to
  certify the non-extremality of all non-extreme piecewise linear functions?
  \textup{(b)}~In particular, do piecewise linear effective perturbations $\tilde\pi$
  suffice?
\end{question}
For the case of piecewise linear functions $\pi$ with rational breakpoints
from $\frac1q\Z$, Basu et al.\@ \cite{basu-hildebrand-koeppe:equivariant},
in the first paper in
the present series then showed that if a nonzero effective perturbation exists, there also exists
one that is piecewise linear with rational breakpoints in $\frac1{4q}\Z$.
This 
implied the first algorithm for testing
the extremality of a (possibly discontinuous) piecewise linear minimal valid
function with rational breakpoints.  
In the same paper, however, Basu et al.\
\cite{basu-hildebrand-koeppe:equivariant} also introduced
a family \sagefuncwithgraph{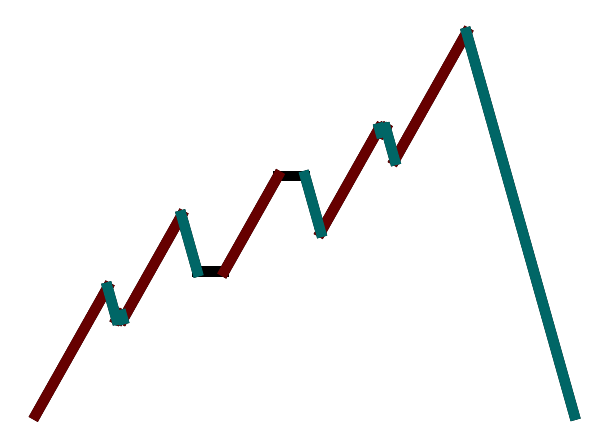} of continuous piecewise linear
minimal valid functions with irrational breakpoints, 
to which the algorithm does not apply.  
Its extremality proof uses an arithmetic argument that depends on the
$\Q$-linear independence of certain parameters of the function, 
and also relies on \autoref{survey-lemma-2.11(v)}. 

\begin{figure}[t]
\centering
\includegraphics[width=\linewidth]{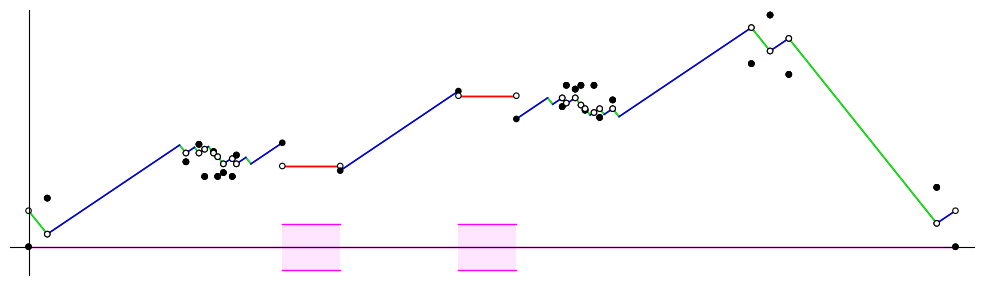}
\caption{This function,
  \sage{$\pi$ = kzh\_minimal\_has\_only\_crazy\_perturbation\_1}, 
  has three slopes (\emph{blue}, \emph{green}, \emph{red}) and is
  discontinuous on both sides of the origin. 
  It is a non-extreme minimal valid function, 
  but in order to demonstrate non-extremality, one needs to use a highly
  discontinuous (locally microperiodic) perturbation.
  We construct a simple explicit example perturbation $\epsilon\bar\pi$
  (\emph{magenta}); see \autoref{th:kzh_minimal_has_only_crazy_perturbation_1}.
  It takes three values, $\epsilon$, $0$, and $-\epsilon$ 
  (\emph{horizontal magenta line segments}) where $\epsilon=0.0003$; 
  in the figure it has been rescaled to amplitude~$\frac1{10}$.
}
\label{fig:has_crazy_perturbation}
\end{figure}
\smallbreak 

\clearpage
\subsection{Contributions of this paper}
In the present paper, we show that the breakdown of the regularity lemma
(\autoref{survey-lemma-2.11(v)}) does not just pose a technical difficulty;
rather, two-sided discontinuous functions take an exceptional place in
the theory of the Gomory--Johnson functions.
We construct a two-sided discontinuous piecewise linear
minimal valid function~$\pi$ with remarkable properties;
cf.~\autoref{fig:has_crazy_perturbation}.  It follows the basic
blueprint of the \sagefuncwithgraph{bhk_irrational} function that we mentioned
above, but introduces a number of
discontinuities and new breakpoints. 
Our function $\pi =
\sagefuncwithgraph{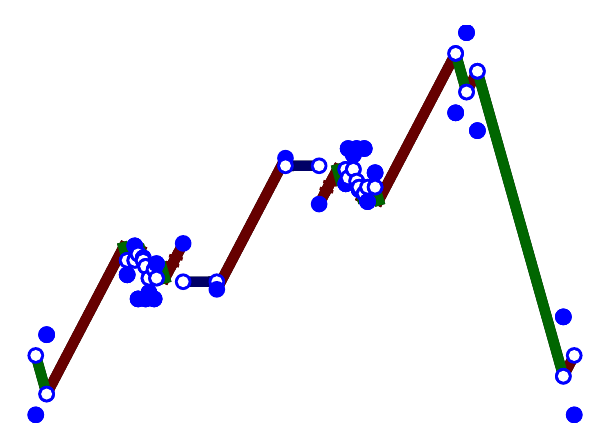}$
is not extreme, but it is impossible to write it as the
convex combination of \emph{piecewise linear} (or, more generally,
\emph{piecewise continuous}) minimal valid functions.
All effective perturbations~$\tilde\pi$ of~$\pi$ are non--piecewise linear,
highly discontinuous, ``locally microperiodic'' functions. 
Thus, we give a negative answer to \autoref{q:certify-nonextremality}\,(b):
\textbf{Piecewise linear effective perturbations do not suffice to certify
  nonextremality of piecewise linear functions.}
(Moreover, in the authors' IPCO 2017 paper
\cite{koeppe-zhou:discontinuous-facets}, building upon the present paper, the
function~$\pi$ and a certain perturbation of it are instrumental in separating
the classes of extreme functions, facets, and so-called \emph{weak facets}, thereby solving the
long-standing open question \cite[Open question 2.9]{igp_survey} regarding the
relations of these notions.) 

In this way, our paper contributes to the foundations of the cutting-plane theory of
integer programming by investigating the fine structure of a space of
cut-generating functions.  In this regard, the paper is in a line of
recent papers on the Gomory--Johnson model: the MPA 2012 paper \cite{bccz08222222}, in which the
first non--piecewise linear, measurable extreme function with 2 slopes was
discovered; and the IPCO 2016 paper \cite{bcdsp:arbitrary-slopes}, in which a
measurable extreme function with an infinite number of slopes was 
constructed using techniques similar to those~in~\cite{bccz08222222}.  

However, our paper not only constructs an
example, but also develops the 
theory of effective perturbations of
minimal valid functions further.  Local continuity of perturbations 
has been observed and used in the original
Gomory--Johnson papers \cite{infinite,infinite2}. The extension to the one-sided
discontinuous case (\autoref{survey-lemma-2.11(v)}) 
was found by Dey, Richard, Li, and Miller \cite{dey1}.  
We prove a 
new analytical tool. Our \autoref{rk:perturbation-lim-exist-from-2d-face} establishes
the \textbf{existence of one-sided or two-sided limits of effective perturbation functions} $\tilde\pi$
at certain points, providing a weak counterpart of
\autoref{survey-lemma-2.11(v)} without the assumption of one-sided continuity of~$\pi$.
We consider our \autoref{rk:perturbation-lim-exist-from-2d-face} to be a major advance, 
requiring a much more subtle proof.  
(For comparison, see the short proof of a stronger version of \autoref{survey-lemma-2.11(v)}, establishing Lipschitz
continuity of $\tilde\pi$ on the intervals of continuity of~$\pi$, in
\cite[Lemma 6.4]{hong-koeppe-zhou:software-paper}.)
We leave as an open question how to generalize
\autoref{rk:perturbation-lim-exist-from-2d-face} to the multi-row case~\cite{bhk-IPCOext}, i.e., functions
$\pi\colon \R^k\to\R$.


\medbreak

Our paper is structured as follows. In \autoref{s:preliminaries}, we introduce 
notation, definitions, and basic results from the literature. 
In \autoref{s:existence-of-limits}, we prove our result on the existence
of limits.
In \autoref{s:algorithm_for_restricted_class_of_crazy_perturbtions}, in order
to demonstrate the nonextremality of the new
function~$\pi$, we develop an \textbf{algorithm to 
verify a certificate of nonextremality} in the form of a given locally quasi\-microperiodic
function $\bar\pi$ of a restricted class.  
In \autoref{s:proof-th:kzh_minimal_has_only_crazy_perturbation_1}, 
we define the example functions $\pi$ and~$\bar\pi$ and
prove that they have the claimed properties. 
The complexity of this proof is much greater than that of functions from
the extreme functions literature; because of this, 
our proof is computer-assisted. 


\section{Preliminaries}
\label{s:preliminaries}

We begin by giving a definition of $\Z$-periodic piecewise linear functions~$\pi\colon \R\to\R$ that are allowed to be discontinuous, following \cite[section 2.1]{basu-hildebrand-koeppe:equivariant} and the recent survey~\cite{igp_survey,igp_survey_part_2}.
  
Let $0 =x_0 < x_1< \dots < x_{n-1} < x_n = 1$.
Denote by
\begin{math}
    \B = \{\, x_0 + t, x_1 + t, \dots, x_{n-1}+t\st
    t\in\Z\,\}
  \end{math}
the set of all breakpoints.
The 0-dimensional faces are defined to be the 
singletons, $\{ x \}$, $x\in B$,
and the 1-dimensional faces are the closed intervals,
$ [x_i+t, x_{i+1}+t]$, $i=0, \dots, {n-1}$, $t\in\Z$. 
The empty face, the 0-dimensional and the 1-dimensional faces form $\P =
\P_{\B}$, a locally finite 
polyhedral
complex, 
periodic modulo~$\Z$.
%
We call a function $\pi\colon \R \to \R$ 
\emph{piecewise linear} over $\mathcal{P}_B$ if for each face $I \in \mathcal{P}_B$, there is an affine linear function $\pi_I \colon \R \to \R$, $\pi_I(x) = c_I x + d_I$  such that $\pi(x) = \pi_I(x)$ for all $x \in\relint(I)$.
 Under this definition, piecewise linear functions can be discontinuous
 .
Let $I = [a, b] \in \mathcal{P}_B$ be a 1-dimensional face. The function $\pi$ can be determined on 
$\intr(I) = (a, b)$ by linear
  interpolation of the limits $\pi(a^+)=\lim_{x\to a, x>a} \pi(x)
  = \pi_I(a)$ and $\pi(b^-)=\lim_{x\to b, x<b} \pi(x) = \pi_I(b)$. 
%
Likewise, we call a function $\pi\colon \R \to \R$ \emph{piecewise continuous} over $\mathcal{P}_B$ if it is continuous over 
$\relint(I)$ for each face $I \in \mathcal{P}_B$. 

The \emph{minimal valid functions} in the classic 1-row Gomory--Johnson
\cite{infinite,infinite2} model are classified. They are the $\Z$-periodic,
subadditive functions $\pi\colon \R\to[0,1]$ with $\pi(0)=0$, $\pi(f)=1$, that
satisfy the \emph{symmetry condition} $\pi(x) + \pi(f - x) = 1$ for all
$x\in\R$.  Here $f$ is the fixed number from \eqref{eq:source-row-relaxation}.
Following
\cite{basu-hildebrand-koeppe:equivariant,igp_survey,igp_survey_part_2}, we
introduce the function 
\[\Delta\pi \colon \R \times \R \to \R,\quad \Delta\pi(x,y) =
  \pi(x)+\pi(y)-\pi(x+y),\] which measures the
slack in the subadditivity condition.
If $\pi(x)$ is piecewise linear, then this 
induces the piecewise linearity of $\Delta\pi(x,y)$.  To express the domains of
linearity of $\Delta\pi(x,y)$, and thus domains of additivity and strict
subadditivity, we introduce the two-dimensional polyhedral complex
$\Delta\P = \Delta\P_\B$. 
The faces $F$ of the complex are defined as follows. Let $I, J, K \in
\P_{\B}$, so each of $I, J, K$ is either the empty set, a breakpoint of $\pi$, or a closed
interval delimited by two consecutive breakpoints. Then 
$$ F = F(I,J,K) = \setcond{\,(x,y) \in \R \times \R}{x \in I,\, y \in J,\, x + y \in
  K\,}.$$ 
The projections  $p_1, p_2, p_3 \colon \R \times \R \to \R$ are defined as 
$p_1(x,y)=x$,  $p_2(x,y)=y$, $p_3(x,y) = x+y$.

Let $F \in \Delta\P$ and let $(u,v) \in F$. Observing that $\Delta\pi|_{\relint(F)}$ is affine, we define 
\begin{equation}\Delta\pi_F(u,v) = \lim_{\substack{(x,y) \to (u,v)\\ (x,y) \in \relint(F)}}
  \Delta\pi(x, y),\label{eq:Delta-pi-F}
\end{equation}
which allows us to conveniently express limits to boundary points of $F$, in particular to vertices of $F$, along paths within $\relint(F)$.
It is clear that $\Delta\pi_F(u,v) $ is affine over $F$, and $\Delta\pi(u,v)=\Delta\pi_F(u,v)$ for all $(u,v) \in \relint(F)$.
We will use $\verts(F)$ to denote the set of vertices of the face~$F$.  

Let $\pi$ be a piecewise linear minimal valid function. We now define the
\emph{additive faces} of the two-dimensional polyhedral complex $\Delta\P$ of
$\pi$. When $\pi$ is continuous, we say that a face $F \in \Delta\P$
is additive if $\Delta\pi =0$ over all $F$. Note that $\Delta\pi$ is affine
over $F$, so the condition is equivalent to $\Delta\pi(u, v) = 0$ for any $(u, v)
\in \verts(F)$. When $\pi$ is discontinuous, following
\cite{hong-koeppe-zhou:software-abstract,hong-koeppe-zhou:software-paper}, we say that a face $F \in
\Delta\P$  is additive if $F$ is contained in a face $F' \in \Delta\P$ such
that $\Delta\pi_{F'}(x,y) =0$ for any $(x,y) \in 
F$.
Since $\Delta\pi$ is affine in the relative interiors of each face of $\Delta\P$, the last condition is equivalent to $\Delta\pi_{F'}(u,v) =0$ for any $(u,v) \in \verts(F)$.

\smallbreak

A minimal valid function $\pi$ is said to be \emph{extreme} if it cannot be written as a
convex combination of two other minimal valid functions.
We say that a function~$\tilde\pi$ is an \emph{effective perturbation function} for
the minimal valid function~$\pi$, denoted $\tilde\pi \in \tilde\Pi^{\pi}(\R,\Z)$, 
if there exists $\epsilon>0$ such that $\pi\pm\epsilon\tilde\pi$ are minimal
valid functions.  
Thus, a minimal valid function $\pi$ is extreme if and only if no non-zero effective
perturbation $\tilde{\pi} \in \tilde{\Pi}^{\pi}(\R,\Z)$ exists. 

The key technique towards answering \autoref{q:properties-eff-perturb} is to
analyze the additivity relations.  The starting point is the
following lemma, which shows that all subadditivity conditions that are tight
(satisfied with equality) for $\pi$ are also tight for an effective
perturbation $\tilde{\pi}$.  This includes additivity in the limit, which we
express using the notation $\Delta\tilde\pi_F$, defined as a limit as in
\eqref{eq:Delta-pi-F}, though $\tilde\pi$ is not assumed to be piecewise
linear. 
\begin{lemma}[{\cite[Lemma 2.7]{basu-hildebrand-koeppe:equivariant}; see \cite[Lemma 6.1]{hong-koeppe-zhou:software-paper}}]
\label{lemma:tight-implies-tight}
Let $\pi$ be a minimal valid function that is piecewise linear over $\P$. Let $F$ be a face of $\Delta\P$ and let $(u, v) \in F$. If $\Delta\pi_F(u,v)=0$, then $\Delta\tilde{\pi}_F(u,v)=0$ for any effective perturbation function $\tilde\pi \in \tilde\Pi^{\pi}(\R,\Z)$.
\end{lemma}


We first make use of the additivity relations that are captured by the
two-dimensional additive faces $F$ of $\Delta\P$.
The following, a corollary of the convex additivity domain lemma \cite[Theorem
4.3]{igp_survey}, appears as \cite[Theorem 6.2]{hong-koeppe-zhou:software-paper}.
\begin{theorem}
\label{thm:directly_covered}
Let $\pi$ be a minimal valid
function that is piecewise linear over $\P$. Let $F$ be a two-dimensional additive face of $\Delta\P$.  Let $\theta=\pi$ or $\theta =\tilde{\pi} \in \tilde{\Pi}^{\pi}(\R,\Z)$.
Then $\theta$ is affine with the same slope over $\intr(p_1(F))$, $\intr(p_2(F))$, and $\intr(p_3(F))$.\footnote{If the function $\pi$ is continuous, then 
$\theta$ is affine with the same slope over the closed intervals $p_1(F)$, $p_2(F)$, and $p_3(F)$, 
by \cite[Corollary 4.9]{igp_survey}.}
\end{theorem}
In the situation of this result, we say that the intervals $\intr(p_1(F))$,
$\intr(p_2(F))$, and $\intr(p_3(F))$ are \emph{(directly) covered}\footnote{In
  the terminology of \cite{basu-hildebrand-koeppe:equivariant}, these
  intervals are said to be \emph{affine imposing}.} 
and are in the same \emph{connected covered component}\footnote{Connected covered components, extending
  the terminology of \cite{basu-hildebrand-koeppe:equivariant}, are simply
  collections of intervals on which an effective perturbation function is affine with the same
  slope.  This notion of connectivity is unrelated to that in the topology of
  the real line.}.  Subintervals of covered intervals are covered.
%


In addition to directly covered intervals, we also have \emph{indirectly
  covered intervals}, which arise from vertical, horizontal, or diagonal
additive edges $F$ of $\Delta\mathcal{P}$.  To handle these additive edges in
our setting of two-sided discontinuous functions with irrational breakpoints,
we need a new technical tool, which we develop in the following section.

\section{Existence of limits of effective perturbations at certain points and a general
  additive edge theorem}
\label{s:existence-of-limits}

The following main theorem is our weak counterpart of the regularity lemma, 
\autoref{survey-lemma-2.11(v)}, which does not require the assumption of
one-sided continuity.  

\begin{theorem}
\label{rk:perturbation-lim-exist-from-2d-face}  
Let $\pi$ be a minimal valid function that is piecewise linear over a complex $\P$.
Let $\tilde{\pi} \in \tilde{\Pi}^{\pi}(\R,\Z)$ be an effective perturbation function for $\pi$. 
If a point $(u, v)$ in a two-dimensional face $F \in \Delta\P$ satisfies that
$\Delta\pi_F(u,v)=0$, then the limits 
\[\lim_{\substack{x\to u\\x \in \intr(p_1(F))}}\tilde{\pi}(x), \quad \lim_{\substack{y\to v\\y \in \intr(p_2(F))}}\tilde{\pi}(y)\quad\text{ and }\quad\lim_{\substack{z\to u+v\\z \in \intr(p_3(F))}}\tilde{\pi}(z)\] 
exist.
\end{theorem}

It is convenient to first prove the following ``pexiderized'' \cite{bhk-IPCOext} proposition.
\begin{proposition}
  \label{prop:near-pexider}
  Let $F$ be a two-dimensional face of $\Delta\P$, where $\P$ is the
  one-dimensional polyhedral complex
  of a piecewise linear function.  Let $(u, v) \in F$. 
  For $i = 1, 2, 3$, let $\tilde\pi_i\colon \R\to\R$ be a function that is 
  bounded near $p_i(u,v)$. 
  If 
  \[\lim_{\substack{(x,y) \to (u,v)\\ (x,y) \in \intr(F)}}
    {\tilde\pi_1(x)+\tilde\pi_2(y)-\tilde\pi_3(x+y)} =0,\] 
  then for $i=1,2,3$, the limit $\lim_{t\to
  p_i(u,v),\; t\in \intr(p_i(F))} \tilde\pi_i(t)$  exists. 
\end{proposition}

\begin{proof}
We will prove the following claim.
\smallskip

\emph{Claim.}  
For $i =1, 2, 3$ and every $\epsilon>0$, there exists a neighborhood $N_i = N_i(\epsilon)$ of $p_i(u,v)$ so
  that for all $t, t' \in N_i \cap \intr(p_i(F))$, we have
  $|\tilde\pi_i(t) - \tilde\pi_i(t')| < 2\epsilon$.
\smallskip\par

When the claim is proved, the proposition follows.
Indeed, let $i \in \{1, 2, 3\}$ and take any sequence $\{ t_j \}_{j\in\N} \subset \intr(p_i(F))$ that converges to
$p_i(u,v)$.  Then by the claim, $\{ \tilde\pi_i(t_j) \}_{j\in\N}$ is a Cauchy sequence and
hence converges to a limit $L \in \R$.  Take any other sequence $\{ t'_j
\}_{j\in\N} \subset \intr(p_i(F))$ that converges to $p_i(u,v)$; then again
$\{ \tilde\pi_i(t'_j) \}_{j\in\N}$ converges to some limit $L' \in \R$.  Let
$\epsilon>0$.  Let $J$ be an index such that 
$t_j, t_j' \in N_i(\epsilon)$ and also $|L - \tilde\pi_i(t_j)| < \epsilon$ and $|L' - \tilde\pi_i(t_j')|
< \epsilon$ for all $j \geq J$.  Then $$|L - L'| = \bigl|\bigl(L -
\tilde\pi_i(t_j)\bigr) - \bigl(L' - \tilde\pi_i(t'_j)\bigr) + 
\bigl(\tilde\pi_i(t_j) - \tilde\pi_i(t'_j)\bigr)\bigr| < 4\epsilon$$ for $j\geq J$.  
Hence, $L = L'$.  Thus, $\lim_{t\to
  p_i(u,v),\; t\in \intr(p_i(F))} \tilde\pi_i(t)$ exists.
\smallbreak

\def\COMMONSCALE{0.35}
\begin{figure}[tp]
\begin{minipage}{.48\textwidth}
\includegraphics[scale=\COMMONSCALE]{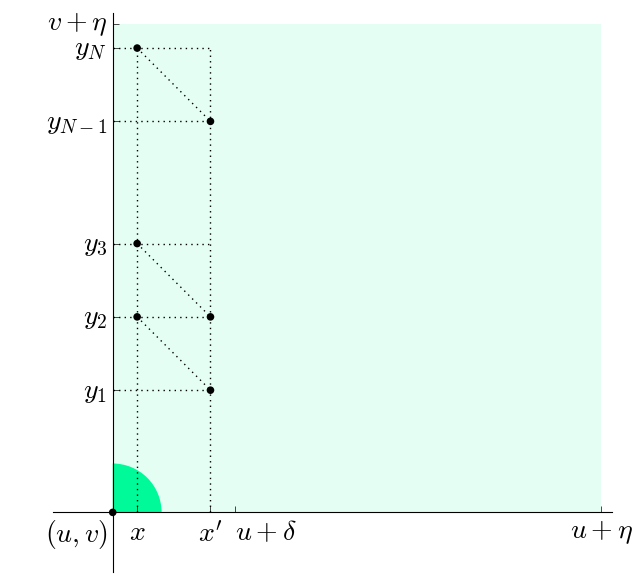}
\par\centering
  Case 1
\end{minipage}
\hfill
\begin{minipage}{.46\textwidth}
\includegraphics[scale=\COMMONSCALE]{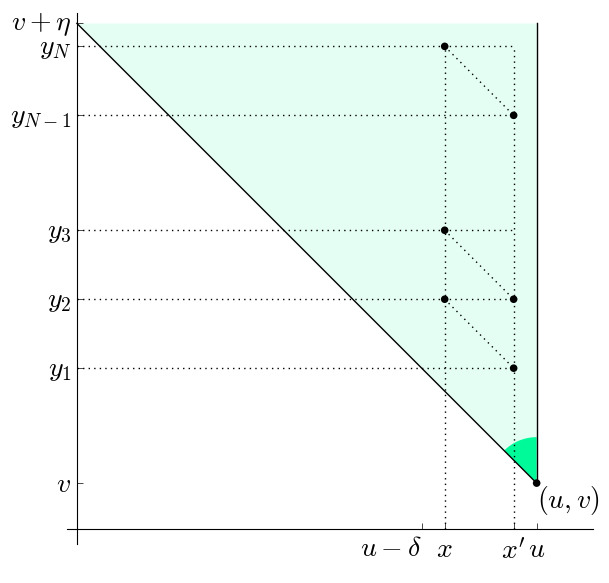}
\par\centering
  Case 3
\end{minipage}
\vspace{3ex}\par
\centering
\begin{minipage}{.58\textwidth}
\includegraphics[scale=\COMMONSCALE]{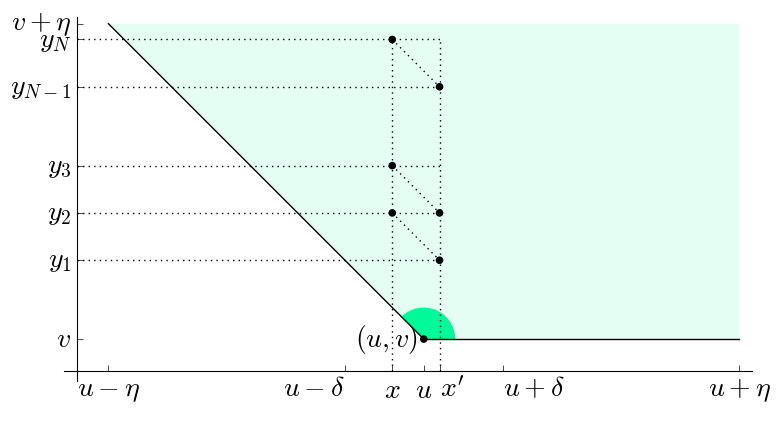}
\par\centering
  Case 2
\end{minipage}
\caption{Illustration of the proof of
  \autoref{prop:near-pexider} for $U$. Three partial diagrams of 
  $\Delta\P$, where the tangent cone $C$ of a two-dimensional face $F\in
  \Delta\P$ at vertex $(u,v)$ is a (\textit{left}, Case~1): right-angle cone (first
  quadrant); 
  (\textit{bottom}, Case~2): obtuse-angle cone; (\textit{right}, Case~3):
  sharp-angle cone (contained in a second quadrant). The \emph{light green
    area} $C_\eta$ is contained in the face~$F$. The \emph{green sector} at
  $(u,v)$ indicates that 
  $\Delta\pi_F(u,v)=0$. The \emph{black points} inside the \emph{light green
    area} show the sequences used in the proof of equation~\eqref{eq:case123-U}.}
\label{fig:proof_uniform_continuous}
\end{figure}
\smallbreak

We now prove the claim.  Let $\varepsilon >0$.  Denote $\Delta\tilde\pi(x,y)
:= \tilde\pi_1(x) + \tilde\pi_2(y) - \tilde\pi_3(x+y)$. 
There exists $\eta >0$ such that for any $(x,y)\in \intr(F)$ satisfying
$\|(x,y)-(u,v)\|_{\infty}<\eta$, we have \[\left|\Delta\tilde\pi(x,y)\right| <
  \varepsilon / 4.\]
Consider the tangent cone $C$ of $F$ at the point $(u, v)$.
One can assume that $\eta$ is small enough, so that 
\[ C_\eta := \bigl\{\, (x, y) \in C : \|(x,y)-(u,v)\|_{\infty} \leq \eta \,\bigr\} \subseteq F \]
and $|\tilde\pi_i(t)| \leq M$ for $t \in p_i(C_\eta)$ for $i = 1, 2, 3$, for some $M > 0$.
Let $N$ be a positive integer such that $N > 4M/\varepsilon + 1$. Define $\delta = \eta/(2N) > 0$. 
\smallbreak

We first consider the situation when $(u,v)$ is a vertex of $F$.  
There are 12 different possible tangent cones that can be formed from bounding
hyperplanes $x=u$, $y=v$ and $x+y=u+v$. 
By transformation using the mappings $(x,y) \mapsto (y,x)$ and $(x,y) \mapsto
(-x, -y)$, under which the statement of the proposition is covariant, and
their composition, $(x,y)\mapsto (-y, -x)$,
one can assume that the tangent cone~$C$ belongs to one of the cases below.

\emph{Case 1} (right-angle corner, first quadrant,
  \autoref{fig:proof_uniform_continuous} left):
  Then $C_\eta = [u, u+\eta] \times [v,v+\eta] \subseteq F$.
  Define $U := (u, u+\delta)$, $V := (v, v+\delta)$, and $W := (u+v, \allowbreak u+v+\delta)$.

\emph{Case 2} (obtuse-angle corner, \autoref{fig:proof_uniform_continuous}
  bottom): 
  Then the quadrilateral $C_\eta = \conv\Bigl(\binom{u}{v}, \binom{u-\eta}{v+\eta}, \allowbreak
    \binom{u+\eta}{v+\eta}, \binom{u+\eta}{ v}\Bigr)$ is contained in $F$.
  Define 
  $U := (u-\delta, u+\delta) $,
  $V := (v, v+\delta) $, and $W := (u+v, \allowbreak u+v+\delta) $. 

  \emph{Case 3a} (sharp-angle corner, \autoref{fig:proof_uniform_continuous} right):
  Then $C_\eta = \conv\left(\binom{u}{v}, \binom{u-\eta}{v+\eta},
    \binom{u}{v+\eta}\right)$ is contained in $F$.
  Define $U := (u-\delta, u)$, $V := (v, v+\delta)$, and $W := (u+v, u+v+\delta)$.

  \emph{Case 3b} (right-angle corner, second quadrant):
  The sharp-angle corner of Case~3a appears as a subcone.  Define $U$ and $V$
  as in Case~3a and $W := (u+v-\delta/2,\allowbreak u+v+\delta/2)$.  

Note that in all cases, $U$, $V$, and $W$ are sets of the form $N_i\cap
\intr(p_i(F))$ for $i = 1, 2, 3$, respectively, as in the claim.

We now show that
\begin{equation}
  \label{eq:case123-U}
  \begin{aligned}
    \qquad&\text{for all $x, x' \in U = N_1 \cap
      \intr(p_1(F))$,
      we have
      $\left|\tilde\pi_1(x)-\tilde\pi_1(x')\right| \leq \varepsilon < 2 \varepsilon$.}
  \end{aligned}
\end{equation}
Let $x, x' \in U$, without loss of generality $x < x'$.  Define a sequence
$y_n = v + \delta + (n-1)(x' -x)$ for $1 \leq n \leq N$. Note that
$\left| x - u\right| \leq \delta < \eta$, $\left| x' - u\right| \leq \delta < \eta$ and $\left| y_n - v \right| =  \left|\delta + (n-1)(x'-x) \right| \leq (2n-1)\delta < \eta$ for $1 \leq n \leq N$. For each $n \in \{1,2, \dots, N-1\}$, since $(x, y_{n+1}), (x', y_n) \in \intr(F)$, $\|(x, y_{n+1})-(u,v)\|_\infty < \eta$ and $\|(x', y_n)-(u,v)\|_\infty < \eta$, we have
 $\left| \tilde\pi_1(x)+\tilde\pi_2(y_{n+1}) -\tilde\pi_3(x+y_{n+1})\right|
 \leq \varepsilon /4$ and  $\left| \tilde\pi_1(x')+\tilde\pi_2(y_n)
   -\tilde\pi_3(x'+y_n)\right| \leq \varepsilon /4$. Note that
 $x+y_{n+1}=x'+y_n$ for $n \in \{1,2, \dots, N-1\}$, so by the triangle inequality,
\[\left| \tilde\pi_1(x)-\tilde\pi_1(x')+\tilde\pi_2(y_{n+1})-\tilde\pi_2(y_n)\right| \leq \varepsilon/2.\]
It follows from summing over $n =1, 2, \dots, N-1$ and the triangle inequality that 
\[\bigl|(N-1)\bigl(\tilde{\pi}_1(x)-\tilde{\pi}_1(x')\bigr)+\tilde{\pi}_2(y_N)-\tilde{\pi}_2(y_1)\bigr|
  \leq (N-1)\varepsilon/2.\]
Therefore,
\begin{equation*}
  \begin{aligned}
    \left|\tilde\pi_1(x)-\tilde\pi_1(x')\right| &\leq
    \left|\tilde\pi_2(y_N)-\tilde\pi_2(y_1)\right| / (N-1) +\varepsilon /2 \\ 
    &\leq 2M /(N-1)+\varepsilon/2 \leq \varepsilon < 2 \varepsilon.
  \end{aligned}
\end{equation*}

\begin{figure}[tp]
\begin{minipage}{.5\textwidth}
\includegraphics[scale=\COMMONSCALE]{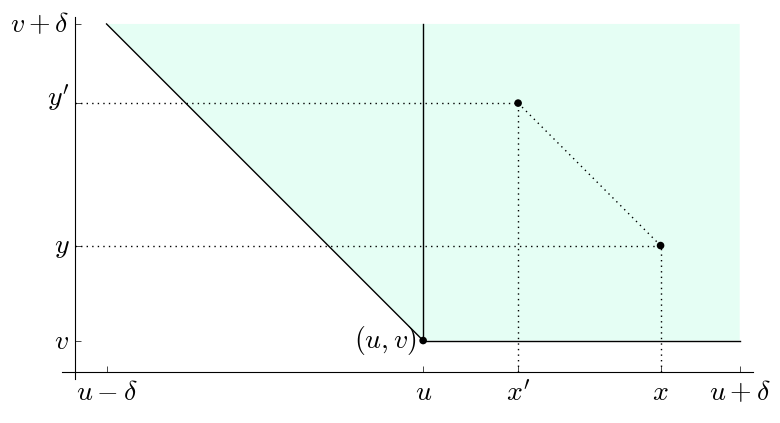}
\par\centering
  Case 1 and Case 2
\end{minipage}
\hfill
\begin{minipage}{.4\textwidth}
\includegraphics[scale=\COMMONSCALE]{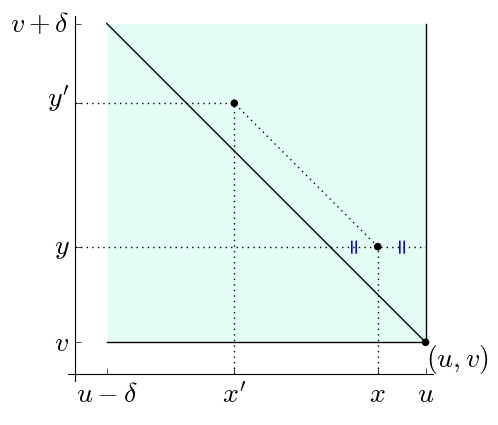}
\par\centering
  Case 3
\end{minipage}
\caption{Illustration of the proof of
  \autoref{prop:near-pexider} for~$V$,
  equation~\eqref{eq:V-from-U}.}
\label{fig:V-from-U}
\end{figure}
Next we show that 
\begin{equation}\label{eq:V-from-U}
\begin{aligned}
  &\text{for all $y, y' \in V = N_2 \cap \intr(p_2(F))$, we have
    $\left|\tilde\pi_2(y)-\tilde\pi_2(y')\right|  < 2 \varepsilon$.}
\end{aligned}
\end{equation}
Let $y, y' \in V$, without loss of generality $y < y'$.  In Case 1 and 2,
define $x = u + (y' - v)$; in Case 3a and 3b, define $x = u - (y - v)/2$. 
Let $x' = x + y - y'$. See \autoref{fig:V-from-U}. Then $x, x' \in U$, and
hence $|\tilde\pi_1(x) - 
\tilde\pi_1(x')| \leq \varepsilon$.  We have $|\Delta\tilde\pi(x,y)| < \varepsilon/4$
and $|\Delta\tilde\pi(x',y')| < \varepsilon/4$. Then
\begin{equation*}
  \begin{aligned}
    \left|\tilde\pi_2(y) - \tilde\pi_2(y')\right| &= \left|\Delta\tilde\pi(x,y) -
      \Delta\tilde\pi(x',y') - (\tilde\pi_1(x) - \tilde\pi_1(x'))\right| \\ & \leq
    \varepsilon/4 + \varepsilon/4 + \varepsilon < 2\varepsilon.
  \end{aligned}
\end{equation*}

\begin{figure}[tp]
\begin{minipage}{.4\textwidth}
\includegraphics[scale=\COMMONSCALE]{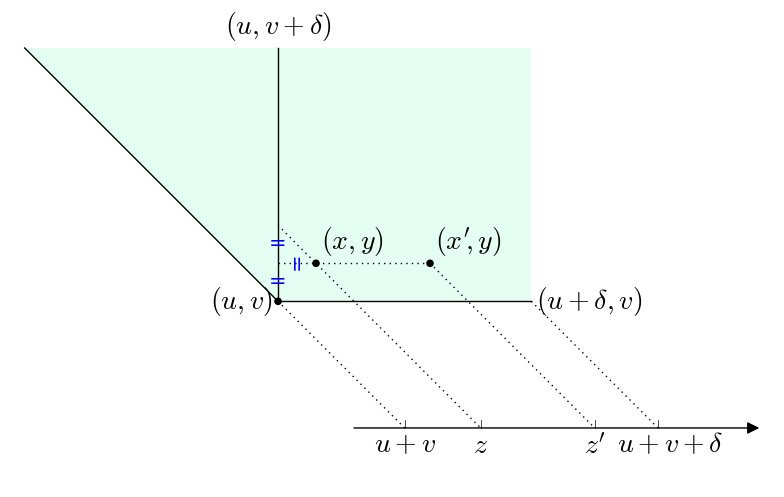}
\par\centering
  Case 1 and Case 2
\par\vspace{20ex}
\end{minipage}
\hfill
\begin{minipage}{.55\textwidth}
\includegraphics[scale=\COMMONSCALE]{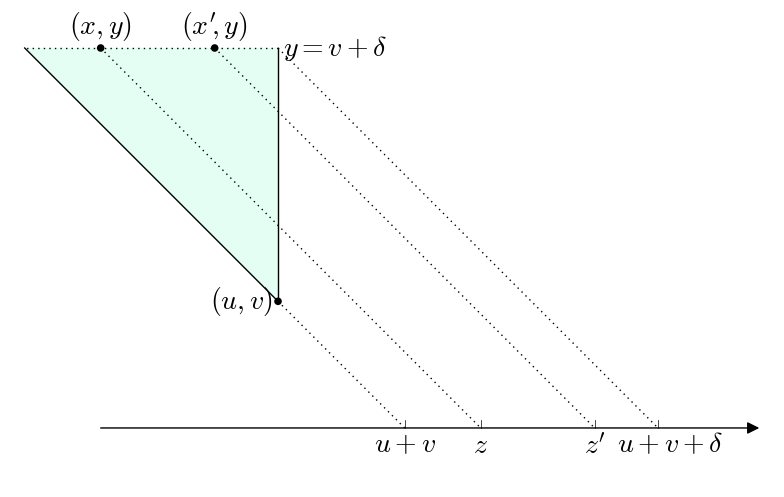}
\par{\centering
  Case 3a
\par\vspace{8ex}}
\includegraphics[scale=\COMMONSCALE]{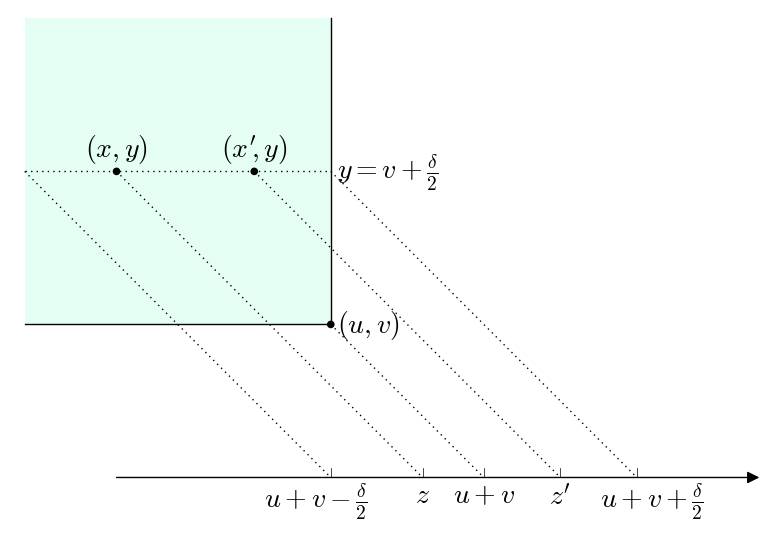}
\par\centering
  Case 3b
\end{minipage}
\caption{Illustration of the proof of
  \autoref{prop:near-pexider} for~$W$,
  equation~\eqref{eq:W-from-U}.}
\label{fig:W-from-U}
\end{figure}
Finally, we show that 
\begin{equation}\label{eq:W-from-U}
\begin{aligned}
  &\text{for all $z, z' \in W = N_3 \cap \intr(p_3(F))$, we have
    $\left|\tilde\pi_3(z)-\tilde\pi_3(z')\right| <  2\varepsilon$.}  
  \end{aligned}
\end{equation}
Let $z, z' \in W$, without loss of generality $z < z'$.  In Case 1 and 2,
define $y = v+(z-(u+v))/2$; in Case 3a, $y = v+\delta$; and in Case 3b, $y =
v+\delta/2$.  Let $y'=y$.  Define $x = z - y$ and $x' = z' - y'$. 
See \autoref{fig:W-from-U}.
Then $x, x'
\in U$, and hence $|\tilde\pi_1(x) -
\tilde\pi_1(x')| \leq \varepsilon$. 
Since $(x, y), (x', y') \in \intr(F)$ and $y - v  = y' - v < \eta$, we have
$|\Delta\tilde\pi(x,y)| < \varepsilon/4$ 
and $|\Delta\tilde\pi(x',y')| < \varepsilon/4$.
Then
\begin{equation*}
  \begin{aligned}
    \left|\tilde\pi_3(z)-\tilde\pi_3(z')\right| &= \left|
      -\Delta\tilde{\pi}(x, y) + \Delta\tilde{\pi}(x', y') +
      (\tilde{\pi}_1(x) - \tilde{\pi}_1(x')) \right| \\ & \leq \varepsilon/4 +
    \varepsilon/4 +\varepsilon < 2\varepsilon.
  \end{aligned}
\end{equation*}
\smallbreak

Now we consider the tangent cones that arise when $(u,v)$ is not a vertex of
$F$.  Then by transformation using the same mappings, one can assume that the
tangent cone $C$ is one of the following three cases:

\emph{Case 4} (upper halfplane, $y \geq v$): The obtuse-angle corner of Case~2
above is a subcone of the tangent cone.  Thus, \eqref{eq:case123-U} and
\eqref{eq:V-from-U} hold for
$U := (u-\delta, u+\delta)$ and $V := (v, v+\delta)$ by the same proofs.  
Also the 2nd quadrant of Case~3b appears as a 
subcone of~$C$.  
Thus, \eqref{eq:W-from-U} holds for 
$W := (u+v-\delta/2, u+v+\delta/2)$ by the same proof.  
Note that $U$, $V$, and $W$ are
sets of the form $N_i\cap \intr(p_i(F))$ for $i=1,2,3$, respectively.

\emph{Case 5} (upper-right halfplane, $x+y \geq u+v$):
Again the obtuse-angle corner of Case~2 above is a subset of the tangent cone.
Thus, \eqref{eq:case123-U} and \eqref{eq:W-from-U} hold for
$U := (u-\delta, u+\delta)$ and $W := (u+v, u+v+\delta)$ by the same proofs. 
By using the transformation $(x,y) \mapsto (y,x)$ and applying
the proof of~\eqref{eq:case123-U} for Case~2 again, we also get
\eqref{eq:V-from-U} for $V := (v-\delta, v+\delta)$.
Again $U$, $V$, and $W$ are sets of the form $N_i\cap \intr(p_i(F))$ for $i=1,2,3$, respectively.

\emph{Case 6} (entire plane). 
Applying Case~4, we get \eqref{eq:case123-U}
and \eqref{eq:V-from-U} for $U := (u-\delta, u+\delta)$ and
$W := (u+v-\delta/2, u+v+\delta/2)$.  Applying Case~4 using the transformation
$(x,y) \mapsto (y,x)$, we get \eqref{eq:V-from-U} for $V := (v-\delta,
v+\delta)$.  Again $U$, $V$, and $W$ are sets of the form $N_i\cap
\intr(p_i(F))$ for $i=1,2,3$, respectively.
\end{proof}

Now we prove the theorem.
\begin{proof}[Proof of \autoref{rk:perturbation-lim-exist-from-2d-face}]
  Let $F$ be a two-dimensional face of $\Delta\mathcal{P}$. Let $(u,v)\in F$ such that $\Delta\pi_F(u,v)=0$. 
  By \autoref{lemma:tight-implies-tight}, $\Delta\pi_F(u,v)=0$ implies
  that \[\Delta\tilde\pi_F(u,v) = \lim_{\substack{(x,y) \to (u,v)\\ (x,y) \in
        \intr(F)}} {\tilde\pi(x)+\tilde\pi(y)-\tilde\pi(x+y)} =0.\]  
  Because minimal valid functions take values in $[0,1]$, 
  the effective perturbation $\tilde\pi$ is bounded. 
  Thus, we can apply \autoref{prop:near-pexider} to $\tilde\pi_1 = \tilde\pi_2 =
  \tilde\pi_3 = \tilde\pi$, which gives the existence of the limits. 
\end{proof}

As a consequence we obtain the following theorem regarding additive edges.  It
is a common generalization of \cite[Lemma
4.5]{basu-hildebrand-koeppe:equivariant} by removing the assumption that all
the breakpoints of $\pi$ are rational numbers; and of \cite[Theorem
6.3]{hong-koeppe-zhou:software-paper} 
by removing the assumption of one-sided continuity.

\begin{theorem}
\label{thm:indirectly_covered}
Let $\pi$ be a minimal function that is piecewise linear over $\P$. 
Let $F$ be a one-dimensional additive face (edge) of $\Delta\P$.
Let $\{i,j\} \subset \{1,2,3\}$ such that $p_i(F)$ and $p_j(F)$ are proper intervals. 
Let $E \subseteq F$ be a sub-interval. For $\theta=\pi$ or $\theta
=\tilde{\pi} \in \tilde{\Pi}^{\pi}(\R,\Z)$, if $\theta$ is affine in $I =
\intr(p_i(E))$, then $\theta$ is affine in $I' = \intr(p_j(E))$ as well with
the same slope. 
\end{theorem}
\begin{proof}
  The proof 
  is identical to the one of \cite[Theorem
  6.3]{hong-koeppe-zhou:software-paper}, 
  replacing the use of \cite[Corollary 6.5]{hong-koeppe-zhou:software-paper}
  regarding the existence of limits by our
  \autoref{rk:perturbation-lim-exist-from-2d-face}.
\end{proof}

In the situation of the theorem, the two proper intervals $p_i(E)$ and
$p_j(E)$ are said to be \emph{connected} through a translation 
(when $F$ is a vertical or horizontal edge) or through a reflection (when $F$
is a diagonal edge). An interval $I'$ that is connected to a covered interval
$I$ is said to be \textit{(indirectly) covered} and in the same connected covered
component as $I$.

\section{An algorithm for a restricted class of locally quasimicroperiodic perturbations}
\label{s:algorithm_for_restricted_class_of_crazy_perturbtions}

In the following, consider a finitely generated additive subgroup~$T$ of the
real numbers, $T = \langle t_1, t_2, \dots, t_n\rangle_{\Z}$, that is dense
in~$\R$. 

\begin{definition}
  We define a restricted class\footnote{These functions are represented by
    instances of the class \sage{PiecewiseCrazyFunction}.} of \emph{locally quasimicroperiodic}
  functions as follows.  
  Let $\bar B$ be a finite set of breakpoints in $[0,1]$.
  Consider a (perturbation) function written 
  as $\bar\pi=\bar\pi^{\mathrm{pwl}}+\bar\pi^{\mathrm{micro}}$ that is
  periodic modulo $\Z$, 
  where $\bar\pi^{\mathrm{pwl}}$ is (possibly discontinuous) piecewise linear over $\mathcal{P}_{\bar{B}}$, 
  and $\bar\pi^{\mathrm{micro}}$ is a locally microperiodic function satisfying that
  \begin{enumerate}[\rm(1)]
  \item $\bar\pi^{\mathrm{micro}}(x)=0$ on any breakpoint 
    $x \in \bar{B}$.
  \item For a closed interval 
    $I$ in $\mathcal{P}_{\bar B}$, $\bar\pi^{\mathrm{micro}}$ restricted to $\intr(I)$ is defined as
\[
\bar\pi^{\mathrm{micro}}_I(x)= 
  \begin{cases} 
   c_1^I & \text{if } x \in  b_1^I + T;\\
   &\vdots \\
   c_{m_I}^I & \text{if } x \in  b_{m_I}^I + T;\\
   0      & \text{otherwise},
  \end{cases}
\]
where $m_I\in\Z_+$.
\end{enumerate}
(Thus, the function deviates from the piecewise linear function
$\bar\pi^{\mathrm{pwl}}$ only on \emph{finitely many} of the infinitely many additive
cosets.  This restriction is in place to enable computations with this class of functions.)
We assume that $c^I_i \neq 0$ and $b^I_i - b^I_j \not\in T$ for $i,j \in \{1, \dots, m_I\}, i\neq j$.
If $m_I = 0$, then $\bar\pi^{\mathrm{micro}}\equiv 0$ on~$I$, and so 
$\bar\pi$ is linear on~$\intr(I)$. 
If $m_I >0$, we
say that $\bar\pi^{\mathrm{micro}}$ is \emph{locally microperiodic}, or, more
precisely, \emph{microperiodic on the open
interval $\intr(I)$}. 
Since $\bar\pi$, restricted to $\intr(I)$, differs from $\bar\pi^{\mathrm{micro}}$ by a linear function, we
say that it is \emph{quasimicroperiodic on $\intr(I)$}. 
\end{definition}
The perturbation function $\bar\pi$
that we will discuss in 
\autoref{s:proof-th:kzh_minimal_has_only_crazy_perturbation_1}
belongs to this restricted class.

We now consider the following algorithmic problem.\footnote{This is
  implemented in \sage{find\_epsilon\_for\_crazy\_perturbation($\pi$, $\bar\pi$)}.
}
\begin{problem}
  \label{prob:crazy_perturbation_effective}
  Given
  (i) a 
    minimal valid function $\pi$ that is
    piecewise linear over $\mathcal{P}_B$,
  (ii) a restricted locally quasimicroperiodic function $\bar\pi$ 
    over $\mathcal{P}_{\bar B}$,
  determine whether $\bar\pi$ is an effective perturbation function for
  $\pi$, i.e., $\bar\pi \in \tilde\Pi^{\pi}(\R,\Z)$, and if yes, 
  find an $\epsilon>0$ such that $\pi\pm\epsilon\bar{\pi}$ are minimal valid functions.
\end{problem}
(By taking the union of the breakpoints and thus defining a common refinement
of the complexes $\mathcal{P}_B$ and $\mathcal{P}_{\bar{B}}$, we can assume
that $\mathcal{P}_B=\mathcal{P}_{\bar{B}}$.)

Consider the two-dimensional polyhedral complex $\Delta\mathcal{P}_B$ and its faces $F$ introduced in \autoref{s:preliminaries}.
\autoref{thm:crazy} below 
solves
the decision problem
of~\autoref{prob:crazy_perturbation_effective}.  Its proof explains how to
find~$\epsilon$. 

\begin{theorem}
\label{thm:crazy}
Let $\pi$ and $\bar\pi$ be as above, with $\bar\pi(0)=\bar\pi(f)=0$. The perturbation $\bar\pi \in \tilde\Pi^{\pi}(\R,\Z)$ if and only if
\begin{enumerate}[\rm(1)]
\item for any face $F$ of $\Delta\mathcal{P}_B$ and $(u, v) \in F$ satisfying $\Delta\pi_F(u,v)=0$, we have $\Delta\bar\pi_F(u,v)=0$; and
\item  for any face $F$ of $\Delta\mathcal{P}_B$ of positive dimension such that there exists $(u,v)\in F$ satisfying $\Delta\pi_F(u,v)=0$, we have $\Delta\bar\pi^{\mathrm{micro}}_F(x,y)=0$ for all $(x,y) \in F$.
\end{enumerate}
\end{theorem}

\begin{proof}[Proof for the $\Rightarrow$ direction:]
Assume $\bar\pi \in \tilde\Pi^{\pi}(\R,\Z)$. Let $\epsilon>0$ such that $\pi+\epsilon \bar\pi$ and $\pi-\epsilon \bar\pi$ are minimal valid functions. Let $F$ be a face of $\Delta\mathcal{P}_B$ and let $(u, v) \in F$ satisfying $\Delta\pi_F(u,v)=0$.  
By \autoref{lemma:tight-implies-tight}, we have $\Delta\bar\pi_F(u,v)=0$. It remains to prove the second necessary condition when $F$ is not a vertex (zero-dimensional face) of $\Delta\mathcal{P}_B$. 

Case 1: Assume that $F$ is a two-dimensional face of $\Delta\mathcal{P}_B$. The projections $p_i(F)$ ($i \in \{1,2,3\}$) are proper intervals. Assume that $p_1(F) \subseteq I$, $p_2(F) \subseteq J$, and $p_3(F) \subseteq K$, where $I, J, K \in \mathcal{P}_B$. By \autoref{rk:perturbation-lim-exist-from-2d-face}, 
the limits $\lim_{t\to p_i(u,v),\; t\in \intr(p_i(F))} \bar\pi(t)$ exist.
Hence, from the definitions of $\bar\pi^{\mathrm{micro}}_I, \bar\pi^{\mathrm{micro}}_J$, and $\bar\pi^{\mathrm{micro}}_K$, we know that $m_I = m_J = m_K =0$. Then $\bar\pi^{\mathrm{micro}}(x) = 0$ for $x \in I$, $x\in J$, or $x \in K$, which implies that $\Delta\bar\pi^{\mathrm{micro}}_F(x,y)=0$ for any $(x,y) \in F$.

Case 2: Assume that $F$ is a one-dimensional face of $\Delta\mathcal{P}_B$. Without loss of generality, we may assume that $F$ is a horizontal edge, i.e., $p_1(F)$ and $p_3(F)$ are closed intervals and $p_2(F)$ is a singleton.  Let $I, J, K \in \mathcal{P}_B$ such that $p_1(F) \subseteq I$, $p_3(F) \subseteq K$ and $J = p_2(F) = \{v\}$. By hypothesis, $(u,v)\in F$, so $u \in p_1(F) \subseteq I$ and $u+v \in p_3(F) \subseteq K$. 

Since the function $x \mapsto \Delta\pi_F(x, v)$ is affine linear over $p_1(F)$, there exists a constant $\alpha >0$ such that $\Delta\pi_F(x, v) \leq \Delta\pi_F(u,v) + \alpha\left| x-u\right|$ for all $x \in p_1(F)$. Let $x \in p_1(F)$. By the hypothesis $\Delta\pi_F(u,v) = 0$, we have  $\Delta\pi_F(x, v) \leq \alpha\left| x-u\right|$. It follows from the subadditivity of $\pi\pm\epsilon\bar\pi$ that 
\[\left|\Delta\bar\pi_F(x,v)\right| \leq \frac{1}{\epsilon} \Delta\pi_F(x,v) \leq \frac{\alpha}{\epsilon} \left| x-u\right|.\] 
Therefore, the function $\Delta\bar\pi_F(\cdot,v)\colon p_1(F) \to \R$ is continuous at $u$, with $\Delta\bar\pi_F(u, v)=0$.

Let $x \in\intr(p_1(F)) \subseteq \intr(I)$, then $x+v \in \intr(p_3(F)) \subseteq \intr(K)$. We can write $\bar\pi^{\mathrm{micro}}(x)$ and $\bar\pi^{\mathrm{micro}}(x+v)$ explicitly.
\[
\bar\pi^{\mathrm{micro}}(x) = \bar\pi^{\mathrm{micro}}_I(x) = 
\begin{cases} 
   c_i^I & \text{if } x \in  b_i^I + T,\; i \in \{1, \dots, m_I\};\\
   0      & \text{otherwise},
  \end{cases}
 \] 
 \[
\bar\pi^{\mathrm{micro}}(x+v) = \bar\pi^{\mathrm{micro}}_K(x+v) = 
\begin{cases} 
   c_k^K & \text{if } x+v \in  b_k^K + T, \; k \in \{1, \dots, m_K\};\\
   0      & \text{otherwise}.
  \end{cases}
 \] 
We also know that $\bar\pi^{\mathrm{micro}}(v) =0$ since $v$ is a breakpoint. 
Thus for $x \in \intr(p_1(F))$, we have that
\begin{equation}
\label{eq:subadditivity_slack_pimicro}
\Delta\bar\pi^{\mathrm{micro}}_F(x,v) = \Delta\bar\pi^{\mathrm{micro}}(x,v) = \bar\pi^{\mathrm{micro}}(x)- \bar\pi^{\mathrm{micro}}(x+v).
\end{equation}
Pick a constant $b$ such that 
\(\forall i \in \{1, \dots, m_I\}\),  \(b \not\equiv b^I_i \pmod T\) and 
\(\forall k \in \{1, \dots, m_K\}\), \(b +v \not\equiv b^K_k \pmod T.\) 
Consider a sequence $\{u_j\}_{j \in\N}$ that converges to $u$ with $u_j \in
\intr(p_1(F))$, $u_j \equiv b \pmod T$ for each $j$. The sequence $\{u_j\}$
exists, since $T$ is dense in $\R$. We have \(\bar\pi^{\mathrm{micro}}(u_j)=
\bar\pi^{\mathrm{micro}}_I(u_j) =0\) and 
\(\bar\pi^{\mathrm{micro}}(u_j+v) = \bar\pi^{\mathrm{micro}}_K(u_j+v)= 0.\)
We know that \(\Delta\bar\pi^{\mathrm{pwl}}_F(u_j,v) =\Delta\bar\pi_F(u_j,v) - \Delta\bar\pi^{\mathrm{micro}}_F(u_j,v)\) and by \eqref{eq:subadditivity_slack_pimicro}, \(\Delta\bar\pi^{\mathrm{micro}}_F(u_j,v)=\bar\pi^{\mathrm{micro}}(u_j)- \bar\pi^{\mathrm{micro}}\allowbreak(u_j+v)=0.\)
Therefore, 
\(\Delta\bar\pi^{\mathrm{pwl}}_F(u_j,v) =\Delta\bar\pi_F(u_j,v). \)
Since 
$\{u_j\}_{j \in\N}$ converges to $u$, and the functions $\Delta\bar\pi^{\mathrm{pwl}}_F(\cdot, v)\colon p_1(F) \to \R$ and $\Delta\bar\pi_F(\cdot, v)\colon p_1(F) \to \R$ are continuous at $u$, by letting $j \to \infty$ in the equation above, we obtain that 
\(\Delta\bar\pi^{\mathrm{pwl}}_F(u,v) = \Delta\bar\pi_F(u, v)=0.\)

Suppose there exists $x \in \intr(p_1(F))$ such that $\Delta\bar\pi^{\mathrm{micro}}_F(x,v)\neq 0$. Since $T$ is dense in $\R$, we can find a sequence $\{u_j\}_{j \in\N}$ converging to $u$, with 
$u_j \in \intr(p_1(F))$, $u_j \equiv x \pmod T$. 
It follows from 
\(\Delta\bar\pi^{\mathrm{micro}}_F(u_j,v) =\Delta\bar\pi_F(u_j,v) - \Delta\bar\pi^{\mathrm{pwl}}_F(u_j,v)\)
and
\(\Delta\bar\pi^{\mathrm{micro}}_F(u_j,v)=\bar\pi^{\mathrm{micro}}(u_j)- \bar\pi^{\mathrm{micro}}(u_j+v)=\bar\pi^{\mathrm{micro}}(x)- \bar\pi^{\mathrm{micro}}(x+v)=\Delta\bar\pi^{\mathrm{micro}}_F(x,v)\)
that
\(\Delta\bar\pi_F(u_j,v) - \Delta\bar\pi^{\mathrm{pwl}}_F(u_j,v) = \Delta\bar\pi^{\mathrm{micro}}_F(x,v).\)
By letting $j\to \infty$ in the above equation, we have a contradiction:
\[0 =  \Delta\bar\pi_F(u, v) - \Delta\bar\pi^{\mathrm{pwl}}_F(u,v) = \Delta\bar\pi^{\mathrm{micro}}_F(x,v) \neq 0.\]
 Therefore, $\Delta\bar\pi^{\mathrm{micro}}_F(x,v)=0$ for all $x \in \intr(p_1(F))$. The constant $\Delta\bar\pi^{\mathrm{micro}}_F(x,v)=0$ extends to the endpoints of $p_1(F)$. We obtain that the statement holds for all $x \in p_1(F)$. This concludes the proof of the $\Rightarrow$ direction.
\end{proof}

\begin{proof}[Proof for the $\Leftarrow$ direction:]
Let $\pi$ and $\bar\pi = \bar\pi^{\mathrm{pwl}} + \bar\pi^{\mathrm{micro}}$ be
given. Assume that 
conditions (1) and (2) are satisfied. We want to find $\epsilon > 0$ such that $\pi^+ = \pi+\epsilon \bar\pi$ and $\pi^- = \pi-\epsilon \bar\pi$ are both minimal valid functions. 
Define
\begin{align*}
m := \min\{\,\Delta\pi_F(x,y) \mid {}& (x,y)\in \verts(\Delta \mathcal{P}_B),\, F
  \text{ is a face of } \Delta \mathcal{P}_B \\ &\text { such that } (x,y)\in F \text{ and } \Delta\pi_F(x,y)\neq 0\,\};
\end{align*}
\[M := \sup_{(x,y)\in \R^2} \left|\Delta\bar \pi(x,y)\right|.\]
Note that $M$ is well defined since $\bar\pi$ is bounded. If $M=0$, for any $\epsilon>0$, $\pi^+$ and $\pi^-$ are subadditive. In the following, we assume $M>0$.
Define \(\epsilon = \frac{m}{M}.\) We also have $m>0$, since $\pi$ is subadditive and $\Delta\pi$ is non-zero somewhere. Thus, $\epsilon>0$.

We claim that $\pi^+$ and $\pi^-$ are subadditive. Let $F$ be a face of $\Delta \mathcal{P}_B$ and let $(x, y) \in F$. We need to show that $\Delta\pi_F^{\pm}(x, y) \geq 0$.  Let $S=\{\,(u,v) \in F \mid \Delta\pi_F(u,v)=0\,\}$.
If $S=\emptyset$, then $\Delta\pi_F(u,v) \geq m$ for any $(u,v) \in \verts(F)$. Since $\Delta\pi_F$ is affine over $F$, we have that $\Delta\pi_F(x,y)\geq m$. Hence,
\begin{equation}
\label{eq:delta_pi_non_negative}
\begin{aligned}
  \Delta\pi_F^{\pm}(x, y) &= \Delta\pi_F(x, y) \pm \epsilon\Delta\bar\pi_F(x,y)
  \\&\geq \Delta\pi_F(x, y) -\epsilon\left| \Delta\bar\pi_F(x,y)\right| \geq m -
  \frac{m}{M}M = 0.
\end{aligned}
\end{equation}
If $S\neq \emptyset$ and $F$ is a zero-dimensional face. Then $\Delta\pi_F(x, y)=0$, and $\Delta\bar\pi_F(x, y)=0$ by condition (1). We have that $\Delta\pi_F^{\pm}(x, y) = 0$.
Now assume that $S\neq \emptyset$ and that the face $F$ has positive dimension. 
By condition (2), we have that $\Delta\bar\pi_F\equiv\Delta\bar\pi^\mathrm{pwl}_F$ on $F$. Therefore,  $\Delta\pi_F^{\pm}$ is affine on $F$. For any vertex $(u,v)$ of $F$, if $(u,v) \not\in S$, then $\Delta\pi_F(u,v)\geq m$ and thus $\Delta\pi_F^{\pm}(u, v) \geq 0$ by \eqref{eq:delta_pi_non_negative};
if $(u,v) \in S$, then $\Delta\bar \pi_F(u,v) = \Delta\pi_F(u,v)=0$ by condition (1), hence $\Delta\pi_F^{\pm}(u, v) =0$. Since $\Delta\pi_F^{\pm}$ is affine on $F$ and $(x,y)\in F$, we obtain that $\Delta\pi_F^{\pm}(x,y) \geq 0$.

We showed that $\pi^{\pm}$ are subadditive. It is clear that $\pi^{\pm}(0)=0$
and $\pi^{\pm}(f)=1$. Let $x\in \R$. Since $\Delta\pi(x, f-x)=0$, condition
(1) implies that $\Delta\bar\pi(x, f-x)=0$. We have
$\bar\pi(x)+\bar\pi(f-x)=\bar\pi(f)=0$, and thus the symmetry condition
$\pi^{\pm}(x) +\pi^{\pm}(f-x) = \pi(x)+\pi(f-x)=1$ is satisfied.  
We know that $\pi$ and $\bar\pi$ are bounded functions, thus $\pi^\pm$ are also bounded.
Suppose that $\pi^+(x)<0$ for some $x\in\R$. Then it follows from the subadditivity that $\pi^+(nx)\leq n\pi^+(x)$ for any $n \in \Z_+$, which contradicts the fact that $\pi^+$ is bounded. We obtain that the functions $\pi^\pm$ are non-negative. 

Therefore, $\pi^\pm$ are minimal valid functions. 
Thus $\bar\pi \in \tilde\Pi^{\pi}(\R,\Z)$.
\end{proof}

\begin{remark}[Implementation details]
\label{rk:implementation_details}
Give functions $\pi$ and $\bar\pi$, face $F \in \Delta\mathcal{P}_B$, and $(u,
v) \in F$ satisfying $\Delta\pi_F(u,v)=0$. Depending on the dimension of the
face $F$, the conditions (1) and (2) of
\autoref{thm:crazy} for having
$\epsilon>0$ are equivalent to the following finitely checkable conditions.

Case 0: $F$ is a vertex (0-dimensional face) of $\Delta\mathcal{P}_B$. We just check if $\bar\pi(u)+\bar\pi(v) = \bar\pi(u+v)$.

Case 1: $F$ is a one-dimensional face of $\Delta\mathcal{P}_B$. Here we consider the case that $F$ is a horizontal edge, $p_1(F) \subseteq I$ and $p_3(F) \subseteq K$ are closed intervals in $\mathcal{P}_B$ and $J = p_2(F) =\{v\}$ is a singleton in $\mathcal{P}_B$. The other cases are similar. We need to check
\begin{enumerate}[\rm(1)]
\item
  $\bar\pi^{\mathrm{pwl}}_I(u)+\bar\pi^{\mathrm{pwl}}(v) =
  \bar\pi^{\mathrm{pwl}}_K(u+v)$
  and 
\item $\bar\pi^{\mathrm{micro}}_I(x) = \bar\pi^{\mathrm{micro}}_K(x+v)$
  for $x \in \intr(p_1(F))$.
\end{enumerate}
The latter condition is equivalent to
\begin{enumerate}[(2a)]
\item[(2a)] for $i \in \{1, \dots, m_I\}$, there is $k \in \{1, \dots, m_K\}$ such that $b^I_i + v - b^K_k \in T$ and $c^I_i = c^K_k$; and
\item[(2b)] for $k \in \{1, \dots, m_K\}$, there is $i \in \{1, \dots, m_I\}$ such that $b^I_i + v - b^K_k \in T$ and $c^I_i = c^K_k$.
\end{enumerate}
Note that it suffices to consider the vertices $(u,v)$ of $F$, since $\Delta\pi_F(u,v)=0$ for some $(u,v) \in\intr(F)$ implies that the same holds for vertices. 

Case 2: $F$ is a two-dimensional face of $\Delta\mathcal{P}_B$. Assume that $p_1(F) \subseteq I, p_2(F) \subseteq J$ and $p_3(F) \subseteq K$, where $I, J, K \in \mathcal{P}_B$. The program checks if $\bar\pi^{\mathrm{pwl}}_I(u)+\bar\pi^{\mathrm{pwl}}_J(v) = \bar\pi^{\mathrm{pwl}}_K(u+v)$, and $\bar\pi^{\mathrm{micro}}_I = \bar\pi^{\mathrm{micro}}_J  = \bar\pi^{\mathrm{micro}}_K \equiv 0$, i.e., $m_I = m_J = m_K =0$. In this case again, it suffices to consider the vertices $(u,v)$ of $F$
satisfying $\Delta\pi_F(u,v)=0$.
\end{remark}

\begin{remark}
\autoref{thm:crazy} generalizes to a more general class of 
functions, in which $\bar\pi^{\mathrm{pwl}}$ is only required to be Lipschitz continuous (but not necessarily affine linear) over $\relint(I)$ for each $I\in \P_{\bar B}$.
\end{remark}

\begin{remark}\label{rem:several-groups}
\autoref{thm:crazy} also generalizes immediately to a more general class of locally quasimicroperiodic functions, in which different dense subgroups $T_I$ of $\R$ are taken in different pieces $\bar\pi^{\mathrm{micro}}_I(x)$, such that the intersection $T_I \cap T_J$ is dense in $\R$ for any intervals $I, J \in \P_{\bar B}$.
\end{remark}
\begin{openquestion}
It is an open question whether the conditions (1) and (2) of
\autoref{thm:crazy} can be checked finitely for the generalized class of \autoref{rem:several-groups}.
\end{openquestion}

\section{The example} 
\label{s:proof-th:kzh_minimal_has_only_crazy_perturbation_1}

Consider the piecewise linear function $\pi$ defined by its values and limits
at its breakpoints $0 = x_0 <  x_1 < \dots < x_{17} = l = \frac{219}{800} < x_{18}
= u = \frac{269}{800} <  x_{19} = f - u = \frac{371}{800} < x_{20} = f - l =
\frac{421}{800} < \dots <  x_{37} = f = \frac{4}{5} < \dots < x_{40} = 1$ in
\autoref{tab:kzh_minimal_has_only_crazy_perturbation_1}.  We have made it
available 
as 
\sage{kzh\_minimal\_has\_only\_crazy\_perturbation\_1}.%

\begin{theorem}\label{th:kzh_minimal_has_only_crazy_perturbation_1}
  The function $\pi$ defined in
  \autoref{tab:kzh_minimal_has_only_crazy_perturbation_1} 
  has the following properties.
  \begin{enumerate}[(i)]
  \item[(i)] 
    $\pi$ is a minimal valid function.  
  \item[(ii)] 
    It cannot be written as a convex combination of piecewise continuous minimal
    valid functions. In particular, it cannot be written as a convex combination of
    piecewise linear minimal valid functions.
  \item[(iii)] 
    It is not extreme because it admits effective locally microperiodic
    perturbations.  
    In particular, define a perturbation
    as follows.
    \begin{equation}\label{eq:crazy_perturbation_for_kzh_minimal_has_only_crazy_perturbation_1}
      \bar\pi(x)= 
      \begin{cases} 
        1& \text{if } x \in (l, u) \text{ such that } x - l \in \langle t_1, t_2 \rangle_{\Z}  \text{, or} \\
        & \text{if } x \in (f-u, f-l) \text{ such that } x - f +u \in \langle t_1, t_2\rangle_{\Z}; \\
        -1 & \text{if } x \in (l, u) \text{ such that } x - u \in \langle t_1, t_2 \rangle_{\Z}  \text{, or} \\
        & \text{if } x \in (f-u, f-l) \text{ such that } x - f+l \in \langle t_1, t_2\rangle_{\Z}; \\
        0 & \text{otherwise},
      \end{cases}
    \end{equation}
    where $f = \frac{4}{5}$, $l=\frac{219}{800}$, $u=\frac{269}{800}$, $t_1 =
    \frac{77}{7752}\sqrt{2}$, $t_2 = \frac{77}{2584}$; see
    \autoref{fig:has_crazy_perturbation}, magenta.
    Let $\epsilon=0.0003$. Then 
    $\pi\pm \epsilon \bar\pi$ are minimal valid functions
    . 
  \end{enumerate}
\end{theorem}

\begin{table}[tp]
  \caption{The piecewise linear function
    \sage{kzh\_minimal\_has\_only\_crazy\_perturbation\_1}, defined by its
    values and limits at the breakpoints.  If a limit is omitted, it equals
    the value.}
  \label{tab:kzh_minimal_has_only_crazy_perturbation_1}
  \centering\small
  \def\arraystretch{1.17}
  \hspace*{-2cm}$
\begin{array}{c@{\quad}*5c}
  \toprule
  i & x_i & \pi(x_i^-) & \pi(x_i) & \pi(x_i^+) & \text{slope}\\
  \midrule
  0 & 0 & \frac{101}{650} & 0 & \frac{101}{650} & \smash{\raisebox{-1.5ex}{$c_3 = -5$}}\\
  1 & \frac{101}{5000} & \frac{707}{13000} & \frac{2727}{13000} & \frac{707}{13000} & \smash{\raisebox{-1.5ex}{$c_1 = \frac{35}{13}$}}\\
  2 & \frac{60153}{369200} &  & \frac{421071}{959920} &  & \smash{\raisebox{-1.5ex}{$c_3 = -5$}}\\
  3 & \frac{849}{5000} & \frac{4851099}{11999000} & -\frac{1925}{71994} \sqrt{2} + \frac{4851099}{11999000} & \frac{4851099}{11999000} & \smash{\raisebox{-1.5ex}{$c_1 = \frac{35}{13}$}}\\
  4 & \frac{1925}{298129} \sqrt{2} + \frac{849}{5000} &  & \frac{67375}{3875677} \sqrt{2} + \frac{4851099}{11999000} &  & \smash{\raisebox{-1.5ex}{$c_3 = -5$}}\\
  5 & \frac{77}{7752} \sqrt{2} + \frac{849}{5000} & \frac{385}{93016248} \sqrt{2} + \frac{4851099}{11999000} & \frac{2695}{100776} \sqrt{2} + \frac{4851099}{11999000} & \frac{385}{93016248} \sqrt{2} + \frac{4851099}{11999000} & \smash{\raisebox{-1.5ex}{$c_1 = \frac{35}{13}$}}\\
  6 & \llap{$a_0 = {}$}\frac{19}{100} & -\frac{1925}{71994} \sqrt{2} + \frac{275183}{599950} & \frac{18196}{59995} & -\frac{1925}{71994} \sqrt{2} + \frac{275183}{599950} & \smash{\raisebox{-1.5ex}{$c_1 = \frac{35}{13}$}}\\
  7 & \frac{77}{22152} \sqrt{2} + \frac{281986521}{1490645000} &  & -\frac{385}{22152} \sqrt{2} + \frac{10467633}{22933000} &  & \smash{\raisebox{-1.5ex}{$c_3 = -5$}}\\
  8 & \frac{40294}{201875} & \frac{848837}{2099500} & \frac{795836841}{1937838500} & \frac{848837}{2099500} & \smash{\raisebox{-1.5ex}{$c_1 = \frac{35}{13}$}}\\
  9 & \frac{36999}{184600} &  & \frac{975607}{2399800} &  & \smash{\raisebox{-1.5ex}{$c_3 = -5$}}\\
  10 & \llap{$a_1 = {}$}\frac{77}{7752} \sqrt{2} + \frac{19}{100} & -\frac{385}{7752} \sqrt{2} + \frac{275183}{599950} & \frac{385}{93016248} \sqrt{2} + \frac{18196}{59995} & -\frac{385}{7752} \sqrt{2} + \frac{275183}{599950} & \smash{\raisebox{-1.5ex}{$c_3 = -5$}}\\
  11 & \frac{1051}{5000} & \frac{4291761}{11999000} & -\frac{1925}{71994} \sqrt{2} + \frac{4291761}{11999000} & \frac{4291761}{11999000} & \smash{\raisebox{-1.5ex}{$c_1 = \frac{35}{13}$}}\\
  12 & \frac{1925}{298129} \sqrt{2} + \frac{1051}{5000} &  & \frac{67375}{3875677} \sqrt{2} + \frac{4291761}{11999000} &  & \smash{\raisebox{-1.5ex}{$c_3 = -5$}}\\
  13 & \llap{$a_2 = {}$}\frac{14199}{64600} & \frac{192500}{3875677} \sqrt{2} + \frac{240046061}{775135400} & \frac{50943}{167960} & \frac{192500}{3875677} \sqrt{2} + \frac{240046061}{775135400} & \smash{\raisebox{-1.5ex}{$c_3 = -5$}}\\
  14 & \frac{77}{7752} \sqrt{2} + \frac{1051}{5000} & \frac{385}{93016248} \sqrt{2} + \frac{4291761}{11999000} & \frac{2695}{100776} \sqrt{2} + \frac{4291761}{11999000} & \frac{385}{93016248} \sqrt{2} + \frac{4291761}{11999000} & \smash{\raisebox{-1.5ex}{$c_1 = \frac{35}{13}$}}\\
  15 & \frac{77}{22152} \sqrt{2} + \frac{342208579}{1490645000} &  & -\frac{385}{22152} \sqrt{2} + \frac{122181831}{298129000} &  & \smash{\raisebox{-1.5ex}{$c_3 = -5$}}\\
  16 & \frac{193799}{807500} &  & \frac{187742}{524875} &  & \smash{\raisebox{-1.5ex}{$c_1 = \frac{35}{13}$}}\\
  17 & \llap{$l = A = {}$}\frac{219}{800} &  & \frac{933}{2080} & \frac{51443}{147680} & \smash{\raisebox{-1.5ex}{$c_2 = \frac{5}{11999}$}}\\
  18 & \llap{$u = A_0 = {}$}\frac{269}{800} & \frac{668809}{1919840} & \frac{683}{2080} &  & \smash{\raisebox{-1.5ex}{$c_1 = \frac{35}{13}$}}\\
  19 & \llap{$f - u = {}$}\frac{371}{800} &  & \frac{1397}{2080} & \frac{1251031}{1919840} & \smash{\raisebox{-1.5ex}{$c_2 = \frac{5}{11999}$}}\\
  20 & \llap{$f - l = {}$}\frac{421}{800} & \frac{96237}{147680} & \frac{1147}{2080} &  & \smash{\raisebox{-1.5ex}{$c_1 = \frac{35}{13}$}}\\
  21 & \frac{452201}{807500} &  & \frac{337133}{524875} &  & \smash{\raisebox{-1.5ex}{$c_3 = -5$}}\\
  22 & -\frac{77}{22152} \sqrt{2} + \frac{850307421}{1490645000} &  & \frac{385}{22152} \sqrt{2} + \frac{175947169}{298129000} &  & \smash{\raisebox{-1.5ex}{$c_1 = \frac{35}{13}$}}\\
  23 & -\frac{77}{7752} \sqrt{2} + \frac{2949}{5000} & -\frac{385}{93016248} \sqrt{2} + \frac{7707239}{11999000} & -\frac{2695}{100776} \sqrt{2} + \frac{7707239}{11999000} & -\frac{385}{93016248} \sqrt{2} + \frac{7707239}{11999000} & \smash{\raisebox{-1.5ex}{$c_3 = -5$}}\\
  24 & \frac{37481}{64600} & -\frac{192500}{3875677} \sqrt{2} + \frac{535089339}{775135400} & \frac{117017}{167960} & -\frac{192500}{3875677} \sqrt{2} + \frac{535089339}{775135400} & \smash{\raisebox{-1.5ex}{$c_3 = -5$}}\\
  25 & -\frac{1925}{298129} \sqrt{2} + \frac{2949}{5000} &  & -\frac{67375}{3875677} \sqrt{2} + \frac{7707239}{11999000} &  & \smash{\raisebox{-1.5ex}{$c_1 = \frac{35}{13}$}}\\
  26 & \frac{2949}{5000} & \frac{7707239}{11999000} & \frac{1925}{71994} \sqrt{2} + \frac{7707239}{11999000} & \frac{7707239}{11999000} & \smash{\raisebox{-1.5ex}{$c_3 = -5$}}\\
  27 & -\frac{77}{7752} \sqrt{2} + \frac{61}{100} & \frac{385}{7752} \sqrt{2} + \frac{324767}{599950} & -\frac{385}{93016248} \sqrt{2} + \frac{41799}{59995} & \frac{385}{7752} \sqrt{2} + \frac{324767}{599950} & \smash{\raisebox{-1.5ex}{$c_3 = -5$}}\\
  28 & \frac{110681}{184600} &  & \frac{1424193}{2399800} &  & \smash{\raisebox{-1.5ex}{$c_1 = \frac{35}{13}$}}\\
  29 & \frac{121206}{201875} & \frac{1250663}{2099500} & \frac{1142001659}{1937838500} & \frac{1250663}{2099500} & \smash{\raisebox{-1.5ex}{$c_3 = -5$}}\\
  30 & -\frac{77}{22152} \sqrt{2} + \frac{910529479}{1490645000} &  & \frac{385}{22152} \sqrt{2} + \frac{12465367}{22933000} &  & \smash{\raisebox{-1.5ex}{$c_1 = \frac{35}{13}$}}\\
  31 & \frac{61}{100} & \frac{1925}{71994} \sqrt{2} + \frac{324767}{599950} & \frac{41799}{59995} & \frac{1925}{71994} \sqrt{2} + \frac{324767}{599950} & \smash{\raisebox{-1.5ex}{$c_1 = \frac{35}{13}$}}\\
  32 & -\frac{77}{7752} \sqrt{2} + \frac{3151}{5000} & -\frac{385}{93016248} \sqrt{2} + \frac{7147901}{11999000} & -\frac{2695}{100776} \sqrt{2} + \frac{7147901}{11999000} & -\frac{385}{93016248} \sqrt{2} + \frac{7147901}{11999000} & \smash{\raisebox{-1.5ex}{$c_3 = -5$}}\\
  33 & -\frac{1925}{298129} \sqrt{2} + \frac{3151}{5000} &  & -\frac{67375}{3875677} \sqrt{2} + \frac{7147901}{11999000} &  & \smash{\raisebox{-1.5ex}{$c_1 = \frac{35}{13}$}}\\
  34 & \frac{3151}{5000} & \frac{7147901}{11999000} & \frac{1925}{71994} \sqrt{2} + \frac{7147901}{11999000} & \frac{7147901}{11999000} & \smash{\raisebox{-1.5ex}{$c_3 = -5$}}\\
  35 & \frac{235207}{369200} &  & \frac{538849}{959920} &  & \smash{\raisebox{-1.5ex}{$c_1 = \frac{35}{13}$}}\\
  36 & \frac{3899}{5000} & \frac{12293}{13000} & \frac{10273}{13000} & \frac{12293}{13000} & \smash{\raisebox{-1.5ex}{$c_3 = -5$}}\\
  37 & \llap{$f = {}$}\frac{4}{5} & \frac{549}{650} & 1 & \frac{549}{650} & \smash{\raisebox{-1.5ex}{$c_1 = \frac{35}{13}$}}\\
  38 & \frac{4101}{5000} & \frac{899}{1000} & \frac{9667}{13000} & \frac{899}{1000} & \smash{\raisebox{-1.5ex}{$c_3 = -5$}}\\
  39 & \frac{4899}{5000} & \frac{101}{1000} & \frac{3333}{13000} & \frac{101}{1000} & \smash{\raisebox{-1.5ex}{$c_1 = \frac{35}{13}$}}\\
  40 & 1 & \frac{101}{650} & 0 & \frac{101}{650} & \smash{\raisebox{-1.5ex}{$$}}\\
  \bottomrule
\end{array}
$
\end{table}

\begin{proof}
Our proof is computer-assisted.  
The reader may 
verify it independently. 
\smallskip

\noindent\emph{Part (i)}. 
Verifying minimality is a routine task, following the algorithm of
\cite[Theorem~2.5]{basu-hildebrand-koeppe:equivariant}; see also
\cite[section~5]{hong-koeppe-zhou:software-paper}.  (The algorithm is
equivalent to 
the one described, in the setting of discontinuous pseudo-periodic
superadditive functions, in Richard, Li, and Miller \cite[Theorem
22]{Richard-Li-Miller-2009:Approximate-Liftings}.)
It is implemented in \cite{infinite-group-relaxation-code} as
\sage{minimality\_test}.  
\begin{verbatim}
    sage: h = kzh_minimal_has_only_crazy_perturbation_1()
    sage: minimality_test(h)
    True
\end{verbatim}
We remark that the software uses exact
computations only. For our example, these take place in the field
$\Q(\sqrt2)$; see Appendix~\ref{s:sage-number-fields}.
The minimality test amounts to verifying subadditivity and symmetry on
vertices of the complex $\Delta\P$.  The reader is invited to inspect the
complex $\Delta\P$ by using the optional argument \sage{show\_plots=True} in
the call to \sage{minimality\_test}.

\smallbreak

\noindent\emph{Part (ii)}. 
Suppose $\tilde\pi$ is a piecewise continuous perturbation such that $\pi\pm \tilde\pi$ are both minimal valid functions. We will prove that $\tilde\pi \equiv 0$.
 
The first step is to compute the directly and indirectly covered intervals, by applying
\autoref{thm:directly_covered} and \autoref{thm:indirectly_covered} a total of
38 times to various additive faces of~$\Delta\P$.  
This computation is implemented in \sage{generate\_covered\_components\_strategically} 
as a
part of \sage{extremality\_test}.  See 
Appendix~\ref{s:debug-proof} for a protocol of this computation.
Again the reader is invited to use the optional argument
\sage{show\_plots=True} to follow the steps of the proof visually. 
In steps $1$, $2$, $3$, and $6$, two-dimensional additive faces of $\Delta\mathcal{P}$ are
considered via \autoref{thm:directly_covered}, so their projections are
directly covered intervals. In the other 
steps, $4, 5, 7, 8, \dots, 38$, one-dimensional additive faces are considered
via \autoref{thm:indirectly_covered}, and the indirectly covered intervals are
found. 
As a result, all the intervals in $\mathcal{P}$, except for $(l, u)$ and $(f-u, f-l)$ are covered intervals belonging to two connected components. Thus the perturbation $\tilde\pi$ is affine linear with 
two independent slope parameters $\tilde c_1$ and $\tilde c_3$ on these intervals. 

Next, we show that $\tilde\pi$ must be affine linear with some slope~$\tilde
c_2$ on the two remaining
intervals $(l, u)$ and $(f-u, f-l)$ as well, where $l =  x_{17} 
$, $u = x_{18} 
$ and $f = x_{37} 
$.  
We reuse a lemma regarding ``reachability'' that was used in \cite[section~5]{basu-hildebrand-koeppe:equivariant} to
establish the extremality of the \sagefuncwithgraph{bhk_irrational} function.
To this end, we introduce the following notation.
Let $a_0 = x_6 = \frac{19}{100}$, $a_1 = x_{10}= \frac{77}{7752}\sqrt{2} + \frac{19}{100}$ and $a_2 = x_{13} = \frac{14199}{64600}$. Define $A = l$,  $A_0 = u$, $A_1 = A_0 + a_0 - a_1$ and $A_2 = A_0 + a_0 - a_2$. Let $t_1 = a_1 - a_0 = \frac{77}{7752}\sqrt{2}$, $t_2 = a_2 - a_0 = \frac{101}{5000}$. Then the numbers $a_0, a_1, a_2, t_1, t_2, f,  A, A_0, A_1, A_2 \in (0,1)$ satisfy the condition (i) $t_1, t_2$ are linearly independent over~$\Q$, and also the conditions (ii) and (iii) of \cite[Assumption 5.1]{basu-hildebrand-koeppe:equivariant}.  
Condition (i) implies that the group $\langle t_1, t_2 \rangle_\Z$ is dense in
$\R$. One can verify that for all $x \in (A, A_i)$,  $\pi(a_i)+\pi(x) =
\pi(a_i+x)$ for $i=0,1,2$. Thus, the same additive equations hold for the
perturbation $\tilde\pi$, i.e., for all $x \in (A, A_i)$,
$\tilde\pi(a_i)+\tilde\pi(x) =\tilde\pi(a_i+x)$ for $i=0,1,2$. 
Let 
$\hat x$ be an arbitrary point in the interval 
$[l + t_2, u - t_2]$.
Define $k_1 = \tilde\pi(a_1)-\tilde\pi(a_0)$ and $k_2 =\tilde\pi(a_2)-\tilde\pi(a_0)$. 
By a generalization of \cite[Lemma~5.2]{basu-hildebrand-koeppe:equivariant} (\autoref{lemma:equiv1-lemma52-general}), if $x \in (l, u)$
satisfies that $x - \hat{x} = \lambda_1 t_1+\lambda_2 t_2$ with $\lambda_1,
\lambda_2 \in \Z$, then $\tilde\pi(x) -\tilde\pi(\hat{x}) = \lambda_1 k_1 +
\lambda_2  k_2$. The piecewise continuous perturbation function $\tilde\pi$ is
bounded. An arithmetic argument using continued fractions (\autoref{lemma:theta-affine}) then implies
that $\frac{k_1}{t_1} = \frac{k_2}{t_2}$. 
Denote $s:= \frac{k_1}{t_1} =\frac{k_2}{t_2}$.
Consider $x \in (l, u)$ such that $x - \hat{x} \in \langle t_1, t_2 \rangle_\Z$, i.e., $x - \hat{x} = \lambda_1 t_1+\lambda_2 t_2$ with $\lambda_1, \lambda_2 \in \Z$; then we have
\[
\frac{\tilde{\pi}(x)-\tilde{\pi}(\hat{x})}{x-\hat{x}}=\frac{\lambda_1 k_1 +\lambda_2 k_2}{\lambda_1 t_1 + \lambda_2 t_2} = s.
\]
Therefore, $\tilde{\pi}$ is affine linear with constant slope $s$ over each
coset $\hat x +  \langle t_1, t_2 \rangle_\Z$ within the interval $(l, u)$,
for $\hat x \in [l+t_2, u-t_2]$. 
Since $\langle t_1, t_2 \rangle_\Z$ is dense in $\R$ and the function $\tilde\pi$ is piecewise continuous on $(l,u)$, we conclude that $\tilde\pi$ is affine linear over the interval $(l, u)$  with slope $s$.
The perturbation $\tilde\pi$ is also affine linear on the interval $(f-u, f-l)$ by the symmetry condition.

Now we can set up a ``symbolic'' piecewise linear function $\tilde\pi$ with 43 parameters, representing the
3~slopes of~$\tilde\pi$ and 40 possible jump values at the breakpoints, taking
the symmetry condition into consideration, as in \cite[section
7]{hong-koeppe-zhou:software-paper} and similar to \cite[Theorem~3.2,
Remark~3.6]{basu-hildebrand-koeppe:equivariant}. 
The 43-dimensional homogeneous linear system implied by the additivity constraints has
a full-rank subsystem.  Hence the solution space has dimension 0. 
See again Appendix~\ref{s:debug-proof} for a protocol of this computation, which
shows the 43 equations that give the full-rank system. 


\smallbreak

\noindent\emph{Part (iii)}. 
We use the algorithm of
\autoref{s:algorithm_for_restricted_class_of_crazy_perturbtions}.
It is implemented as \sage{find\_epsilon\_for\_crazy\_perturbation}. 
The function $\bar\pi = \sage{cp}$ is defined in the doctests of
\sage{kzh\_minimal\_has\_only\_crazy\_perturbation\_1}.

\begin{verbatim}
    sage: find_epsilon_for_crazy_perturbation(h, cp)
    0.0003958663221935161?
\end{verbatim}

This concludes the proof of the theorem.
\end{proof}

\appendix

\section{Exact computations with algebraic field extensions}
\label{s:sage-number-fields}
  The software \cite{infinite-group-relaxation-code} is written in SageMath \cite{sage}, a comprehensive
  Python-based open source computer algebra system. 
  By default it works with (arbitrary-precision) rational numbers; but when
  parameters of a function are irrational 
  algebraic numbers, it constructs a suitable number field, embedded into the
  real numbers, and makes exact computations with the elements of this number
  field. 

  These number fields are algebraic field extensions (in the case of the example function discussed in \autoref{s:proof-th:kzh_minimal_has_only_crazy_perturbation_1}, of
  degree $d=2$) of the field~$\Q$ of rational numbers, in much the same way that the
  field~$\mathbb C$ of complex numbers is an algebraic field extension (of degree $d=2$) of
  the field~$\R$ of real numbers.  Elements of the field are represented as a
  rational coordinate vector of dimension~$d$ over the base field~$\Q$, and all arithmetic computations
  are done by manipulating these vectors.  
  The number fields can be considered either abstractly or as embedded subfields 
  of an enclosing field.  When we say that the number fields are embedded into 
  the enclosing field of real numbers, this means in particular that they
  inherit the linear order from the real numbers.  To decide whether $a < b$, one computes
  sufficiently many digits of both numbers using a rigorous version of
  Newton's method; this is guaranteed to 
  terminate because $a = b$  can be decided by just comparing the
  coordinate vectors. 

  The program \cite{infinite-group-relaxation-code} includes a
  function \sagefunc{nice_field_values} that provides convenient access to 
  the standard facilities of SageMath that construct such an
  embedded number field.


\clearpage
\section{Arithmetic argument in the proof of
  \autoref{th:kzh_minimal_has_only_crazy_perturbation_1}\,(ii)}
\label{s:arithmetic-argument}

\begin{lemma}[{Generalization of \cite[Lemma 5.2]{basu-hildebrand-koeppe:equivariant}}]
  \label{lemma:equiv1-lemma52-general}
  Using the notations and under the conditions of \cite[Assumption
  5.1]{basu-hildebrand-koeppe:equivariant}, suppose that for all
  $x \in [A, A_i]$, $\tilde \pi(a_i) + \tilde \pi(x) = \tilde \pi(a_i + x)$
  for $i=0,1,2$.  
  Let $\hat x \in [A, A_0]$ such that $\hat x\pm t_i \in [A,A_0]$ for $i=1,2$.
  If
  $x = \hat x + \lambda_1 t_1 + \lambda_2 t_2 \in 
  [A, A_0]$ with $\lambda_1, \lambda_2 \in \Z$, then
  $$
  \tilde \pi(x) - \tilde \pi(\hat x) = \lambda_1 \left(\tilde \pi({a_1}) - \tilde
    \pi({a_0})\right) + \lambda_2 \left(\tilde \pi({a_2}) - \tilde
    \pi({a_0})\right).
  $$
\end{lemma}
\begin{proof}
  The proof appears in \cite{basu-hildebrand-koeppe:equivariant} for the case
  $\hat x = (A+A_0)/2$; this is called $x_0$ in
  \cite{basu-hildebrand-koeppe:equivariant}. 
  The proof extends verbatim to general $\hat x$ as in our hypothesis.
\end{proof}

\begin{lemma}
\label{lemma:theta-affine}
Let $\theta \colon \R\to \R$ be a bounded function that is piecewise continuous on the interval $(l, u)$.  Denote $\bar{x}:=\frac{l+u}{2}$. Let $t_1, t_2$ be positive numbers that are linearly independent over $\Q$, and let $k_1, k_2 \in \R$. 
Assume that for any $x\in (l,u)$ such that $x - \bar{x} = \lambda_1 t_1+\lambda_2 t_2$ with $\lambda_1, \lambda_2 \in \Z$, we have $\theta(x)-\theta(\bar{x}) = \lambda_1 k_1+\lambda_2 k_2$.
Then, $\frac{k_1}{t_1} = \frac{k_2}{t_2}$. 
\end{lemma}
\begin{proof}
Suppose for the sake of contradiction that $k_1 t_2 = k_2 t_1 +\sigma$ where $\sigma \neq 0$. Let $U \in \R$ such that $|\theta(x)|\leq U$ for any $x$. Let $N$ be an integer such that $N > (|k_2| (u-l)/2+2 U t_2)/|\sigma|$. 
Since $t_1$ and $t_2$ are linearly independent over~$\Q$, $t_1/t_2$ is irrational. 
The continued fraction approximations 
for $t_1/t_2$ form an infinite sequence $\{p_n/q_n\}_{n\in\N}$ of successive convergents  with the property that $|t_1/t_2 - p_n/q_n| \leq 1/(q_n q_{n+1})$.
Let $(\lambda_1,\lambda_2) = (q_n, -p_n)$ for some large enough index $n$, then $\lambda_1, \lambda_2 \in \Z$ satisfy that $\lambda_1 > N$ and $|\lambda_1 t_1 +\lambda_2 t_2| < (u - l)/2$.
Let $x = \bar{x} + \lambda_1 t_1+\lambda_2 t_2$. Then $x \in (l, u)$, and hence $\theta(x)-\theta(\bar{x}) = \lambda_1 k_1+\lambda_2 k_2$.
We have on the one hand, $|\lambda_1 k_1 + \lambda_2 k_2| \leq 2U$. On the other hand,
  \begin{align*}
    |\lambda_1 k_1 t_2 + \lambda_2 k_2 t_2| &= |\lambda_1 (k_2 t_1 +\sigma) +
    \lambda_2 k_2 t_2| \\ &= |k_2(\lambda_1 t_1 + \lambda_2 t_2) +\lambda_1
    \sigma| \geq \bigl\lvert\lvert\lambda_1\sigma\rvert- \lvert k_2(\lambda_1
    t_1 + \lambda_2 t_2)\rvert\bigr\rvert.
  \end{align*}
Since $|\lambda_1\sigma| > N|\sigma| > |k_2|(u-l)/2+2U t_2$ and $|\lambda_1 t_1 +\lambda_2 t_2| < (u-l)/2$,  we have $|\lambda_1 k_1 t_2 + \lambda_2 k_2 t_2| > 2U t_2$. By dividing both sides by $t_2 >0$, we obtain $|\lambda_1 k_1 + \lambda_2 k_2| > 2U$, a contradiction. Therefore, $k_1 t_2 = k_2 t_1$. 
\end{proof}

\clearpage
\section{Protocol of the automatic proof}
\label{s:debug-proof}

The following is a protocol of the automatic extremality test implemented in
\cite{infinite-group-relaxation-code}.  The protocol provides the details for
the proof of \autoref{th:kzh_minimal_has_only_crazy_perturbation_1} (ii).  We
remark that by invoking \sage{extremality\_test} with the optional argument
\sage{crazy\_perturbations=False}, the code is asked to test extremality
relative to the space of piecewise continuous functions, which is why it
computes \sage{True}.
\begin{lstlisting}[breaklines]
sage: import igp; from igp import *
INFO: Welcome to the infinite-group-relaxation-code. DON'T PANIC. See demo.sage for instructions.
sage: igp.show_values_of_fastpiecewise = False; igp.show_RNFElement_by_embedding = False
sage: igp.strategical_covered_components = True; logging.getLogger().setLevel(logging.DEBUG)
sage: h = kzh_minimal_has_only_crazy_perturbation_1()
INFO: Coerced into real number field: Real Number Field in `a` as the root of the defining polynomial y^2 - 2 near 1.414213562373095?
sage: extremality_test(h, crazy_perturbations=False)
INFO: pi(0) = 0. pi is subadditive. pi is symmetric. Thus pi is minimal.
INFO: Computing maximal additive faces... done
DEBUG: Step 1: Consider the 2d additive <Face ([101/5000, 1317379/9230000], [235207/369200, 1899/2500], [6066621/9230000, 3899/5000])>. [<Int(101/5000, 1317379/9230000)>, <Int(235207/369200, 3899/5000)>] is directly covered.
DEBUG: Step 2: Consider the 2d additive <Face ([2101/2500, 4899/5000], [2101/2500, 4899/5000], [9101/5000, 4899/2500])>. [<Int(4101/5000, 4899/5000)>] is directly covered.
DEBUG: Step 3: Consider the 2d additive <Face ([101/5000, 51/400], [269/800, 8871/20000], [7129/20000, 371/800])>. [<Int(101/5000, 51/400)>, <Int(269/800, 371/800)>] is directly covered. We obtain a new covered component [<Int(101/5000, 1317379/9230000)>, <Int(269/800, 371/800)>, <Int(235207/369200, 3899/5000)>], with overlapping components merged in.
DEBUG: Step 4: Consider the 1d additive <Face ([101/5000, 13941/258400], [14199/64600], [193799/807500, 219/800])>. [<Int(193799/807500, 219/800)>] is indirectly covered. We obtain a new covered component [<Int(101/5000, 1317379/9230000)>, <Int(193799/807500, 219/800)>, <Int(269/800, 371/800)>, <Int(235207/369200, 3899/5000)>], with overlapping components merged in.
DEBUG: Step 5: Consider the 1d additive <Face ([421/800, 452201/807500], [14199/64600], [192779/258400, 3899/5000])>. [<Int(421/800, 452201/807500)>] is indirectly covered. We obtain a new covered component [<Int(101/5000, 1317379/9230000)>, <Int(193799/807500, 219/800)>, <Int(269/800, 371/800)>, <Int(421/800, 452201/807500)>, <Int(235207/369200, 3899/5000)>], with overlapping components merged in.
DEBUG: Step 6: Consider the 2d additive <Face ([101/5000, 1317379/9230000], [101/5000, 1317379/9230000], [101/2500, 60153/369200])>. [<Int(101/5000, 60153/369200)>] is directly covered. We obtain a new covered component [<Int(101/5000, 60153/369200)>, <Int(193799/807500, 219/800)>, <Int(269/800, 371/800)>, <Int(421/800, 452201/807500)>, <Int(235207/369200, 3899/5000)>], with overlapping components merged in.
DEBUG: Step 7: Consider the 1d additive <Face ([101/5000, 101/2500], [3899/5000], [4/5, 4101/5000])>. [<Int(4/5, 4101/5000)>] is indirectly covered. We obtain a new covered component [<Int(101/5000, 60153/369200)>, <Int(193799/807500, 219/800)>, <Int(269/800, 371/800)>, <Int(421/800, 452201/807500)>, <Int(235207/369200, 3899/5000)>, <Int(4/5, 4101/5000)>], with overlapping components merged in.
DEBUG: Step 8: Consider the 1d additive <Face ([4899/5000, 1], [3899/5000], [4399/2500, 8899/5000])>. [<Int(4899/5000, 1)>] is indirectly covered. We obtain a new covered component [<Int(101/5000, 60153/369200)>, <Int(193799/807500, 219/800)>, <Int(269/800, 371/800)>, <Int(421/800, 452201/807500)>, <Int(235207/369200, 3899/5000)>, <Int(4/5, 4101/5000)>, <Int(4899/5000, 1)>], with overlapping components merged in.
DEBUG: Step 9: Consider the 1d additive <Face ([0, 101/5000], [4101/5000], [4101/5000, 2101/2500])>. [<Int(0, 101/5000)>] is indirectly covered. We obtain a new covered component [<Int(0, 101/5000)>, <Int(4101/5000, 4899/5000)>], with overlapping components merged in.
DEBUG: Step 10: Consider the 1d additive <Face ([0, 101/5000], [3899/5000], [3899/5000, 4/5])>. [<Int(3899/5000, 4/5)>] is indirectly covered. We obtain a new covered component [<Int(0, 101/5000)>, <Int(3899/5000, 4/5)>, <Int(4101/5000, 4899/5000)>], with overlapping components merged in.
DEBUG: Step 11: Consider the 1d additive <Face ([101/5000, -1925/298129*a + 58986029/1490645000], [77/7752*a + 19/100], [77/7752*a + 1051/5000, 77/22152*a + 342208579/1490645000])>. [<Int(77/7752*a + 1051/5000, 77/22152*a + 342208579/1490645000)>] is indirectly covered. We obtain a new covered component [<Int(101/5000, 60153/369200)>, <Int(77/7752*a + 1051/5000, 77/22152*a + 342208579/1490645000)>, <Int(193799/807500, 219/800)>, <Int(269/800, 371/800)>, <Int(421/800, 452201/807500)>, <Int(235207/369200, 3899/5000)>, <Int(4/5, 4101/5000)>, <Int(4899/5000, 1)>], with overlapping components merged in.
DEBUG: Step 12: Consider the 1d additive <Face ([-77/22152*a + 850307421/1490645000, -77/7752*a + 2949/5000], [77/7752*a + 19/100], [1925/298129*a + 1133529971/1490645000, 3899/5000])>. [<Int(-77/22152*a + 850307421/1490645000, -77/7752*a + 2949/5000)>] is indirectly covered. We obtain a new covered component [<Int(101/5000, 60153/369200)>, <Int(77/7752*a + 1051/5000, 77/22152*a + 342208579/1490645000)>, <Int(193799/807500, 219/800)>, <Int(269/800, 371/800)>, <Int(421/800, 452201/807500)>, <Int(-77/22152*a + 850307421/1490645000, -77/7752*a + 2949/5000)>, <Int(235207/369200, 3899/5000)>, <Int(4/5, 4101/5000)>, <Int(4899/5000, 1)>], with overlapping components merged in.
DEBUG: Step 13: Consider the 1d additive <Face ([101/5000, 1925/298129*a + 101/5000], [19/100], [1051/5000, 1925/298129*a + 1051/5000])>. [<Int(1051/5000, 1925/298129*a + 1051/5000)>] is indirectly covered. We obtain a new covered component [<Int(101/5000, 60153/369200)>, <Int(1051/5000, 1925/298129*a + 1051/5000)>, <Int(77/7752*a + 1051/5000, 77/22152*a + 342208579/1490645000)>, <Int(193799/807500, 219/800)>, <Int(269/800, 371/800)>, <Int(421/800, 452201/807500)>, <Int(-77/22152*a + 850307421/1490645000, -77/7752*a + 2949/5000)>, <Int(235207/369200, 3899/5000)>, <Int(4/5, 4101/5000)>, <Int(4899/5000, 1)>], with overlapping components merged in.
DEBUG: Step 14: Consider the 1d additive <Face ([-1925/298129*a + 2949/5000, 2949/5000], [19/100], [-1925/298129*a + 3899/5000, 3899/5000])>. [<Int(-1925/298129*a + 2949/5000, 2949/5000)>] is indirectly covered. We obtain a new covered component [<Int(101/5000, 60153/369200)>, <Int(1051/5000, 1925/298129*a + 1051/5000)>, <Int(77/7752*a + 1051/5000, 77/22152*a + 342208579/1490645000)>, <Int(193799/807500, 219/800)>, <Int(269/800, 371/800)>, <Int(421/800, 452201/807500)>, <Int(-77/22152*a + 850307421/1490645000, -77/7752*a + 2949/5000)>, <Int(-1925/298129*a + 2949/5000, 2949/5000)>, <Int(235207/369200, 3899/5000)>, <Int(4/5, 4101/5000)>, <Int(4899/5000, 1)>], with overlapping components merged in.
DEBUG: Step 15: Consider the 1d additive <Face ([-1925/298129*a + 3151/5000, 3151/5000], [19/100], [-1925/298129*a + 4101/5000, 4101/5000])>. [<Int(-1925/298129*a + 3151/5000, 3151/5000)>] is indirectly covered. We obtain a new covered component [<Int(101/5000, 60153/369200)>, <Int(1051/5000, 1925/298129*a + 1051/5000)>, <Int(77/7752*a + 1051/5000, 77/22152*a + 342208579/1490645000)>, <Int(193799/807500, 219/800)>, <Int(269/800, 371/800)>, <Int(421/800, 452201/807500)>, <Int(-77/22152*a + 850307421/1490645000, -77/7752*a + 2949/5000)>, <Int(-1925/298129*a + 2949/5000, 2949/5000)>, <Int(-1925/298129*a + 3151/5000, 3151/5000)>, <Int(235207/369200, 3899/5000)>, <Int(4/5, 4101/5000)>, <Int(4899/5000, 1)>], with overlapping components merged in.
DEBUG: Step 16: Consider the 1d additive <Face ([4899/5000, 1925/298129*a + 4899/5000], [19/100], [5849/5000, 1925/298129*a + 5849/5000])>. [<Int(849/5000, 1925/298129*a + 849/5000)>] is indirectly covered. We obtain a new covered component [<Int(101/5000, 60153/369200)>, <Int(849/5000, 1925/298129*a + 849/5000)>, <Int(1051/5000, 1925/298129*a + 1051/5000)>, <Int(77/7752*a + 1051/5000, 77/22152*a + 342208579/1490645000)>, <Int(193799/807500, 219/800)>, <Int(269/800, 371/800)>, <Int(421/800, 452201/807500)>, <Int(-77/22152*a + 850307421/1490645000, -77/7752*a + 2949/5000)>, <Int(-1925/298129*a + 2949/5000, 2949/5000)>, <Int(-1925/298129*a + 3151/5000, 3151/5000)>, <Int(235207/369200, 3899/5000)>, <Int(4/5, 4101/5000)>, <Int(4899/5000, 1)>], with overlapping components merged in.
DEBUG: Step 17: Consider the 1d additive <Face ([3151/5000, 235207/369200], [19/100], [4101/5000, 61071/73840])>. [<Int(3151/5000, 235207/369200)>] is indirectly covered. We obtain a new covered component [<Int(0, 101/5000)>, <Int(3151/5000, 235207/369200)>, <Int(3899/5000, 4/5)>, <Int(4101/5000, 4899/5000)>], with overlapping components merged in.
DEBUG: Step 18: Consider the 1d additive <Face ([71841/73840, 4899/5000], [19/100], [429353/369200, 5849/5000])>. [<Int(60153/369200, 849/5000)>] is indirectly covered. We obtain a new covered component [<Int(0, 101/5000)>, <Int(60153/369200, 849/5000)>, <Int(3151/5000, 235207/369200)>, <Int(3899/5000, 4/5)>, <Int(4101/5000, 4899/5000)>], with overlapping components merged in.
DEBUG: Step 19: Consider the 1d additive <Face ([77/7752*a, 101/5000], [19/100], [77/7752*a + 19/100, 1051/5000])>. [<Int(77/7752*a + 19/100, 1051/5000)>] is indirectly covered. We obtain a new covered component [<Int(0, 101/5000)>, <Int(60153/369200, 849/5000)>, <Int(77/7752*a + 19/100, 1051/5000)>, <Int(3151/5000, 235207/369200)>, <Int(3899/5000, 4/5)>, <Int(4101/5000, 4899/5000)>], with overlapping components merged in.
DEBUG: Step 20: Consider the 1d additive <Face ([2949/5000, -77/7752*a + 61/100], [19/100], [3899/5000, -77/7752*a + 4/5])>. [<Int(2949/5000, -77/7752*a + 61/100)>] is indirectly covered. We obtain a new covered component [<Int(0, 101/5000)>, <Int(60153/369200, 849/5000)>, <Int(77/7752*a + 19/100, 1051/5000)>, <Int(2949/5000, -77/7752*a + 61/100)>, <Int(3151/5000, 235207/369200)>, <Int(3899/5000, 4/5)>, <Int(4101/5000, 4899/5000)>], with overlapping components merged in.
DEBUG: Step 21: Consider the 1d additive <Face ([61/100, -77/7752*a + 3151/5000], [77/7752*a + 19/100], [77/7752*a + 4/5, 4101/5000])>. [<Int(61/100, -77/7752*a + 3151/5000)>] is indirectly covered. We obtain a new covered component [<Int(101/5000, 60153/369200)>, <Int(849/5000, 1925/298129*a + 849/5000)>, <Int(1051/5000, 1925/298129*a + 1051/5000)>, <Int(77/7752*a + 1051/5000, 77/22152*a + 342208579/1490645000)>, <Int(193799/807500, 219/800)>, <Int(269/800, 371/800)>, <Int(421/800, 452201/807500)>, <Int(-77/22152*a + 850307421/1490645000, -77/7752*a + 2949/5000)>, <Int(-1925/298129*a + 2949/5000, 2949/5000)>, <Int(61/100, -77/7752*a + 3151/5000)>, <Int(-1925/298129*a + 3151/5000, 3151/5000)>, <Int(235207/369200, 3899/5000)>, <Int(4/5, 4101/5000)>, <Int(4899/5000, 1)>], with overlapping components merged in.
DEBUG: Step 22: Consider the 1d additive <Face ([4899/5000, -77/7752*a + 1], [77/7752*a + 19/100], [77/7752*a + 5849/5000, 119/100])>. [<Int(77/7752*a + 849/5000, 19/100)>] is indirectly covered. We obtain a new covered component [<Int(101/5000, 60153/369200)>, <Int(849/5000, 1925/298129*a + 849/5000)>, <Int(77/7752*a + 849/5000, 19/100)>, <Int(1051/5000, 1925/298129*a + 1051/5000)>, <Int(77/7752*a + 1051/5000, 77/22152*a + 342208579/1490645000)>, <Int(193799/807500, 219/800)>, <Int(269/800, 371/800)>, <Int(421/800, 452201/807500)>, <Int(-77/22152*a + 850307421/1490645000, -77/7752*a + 2949/5000)>, <Int(-1925/298129*a + 2949/5000, 2949/5000)>, <Int(61/100, -77/7752*a + 3151/5000)>, <Int(-1925/298129*a + 3151/5000, 3151/5000)>, <Int(235207/369200, 3899/5000)>, <Int(4/5, 4101/5000)>, <Int(4899/5000, 1)>], with overlapping components merged in.
DEBUG: Step 23: Consider the 1d additive <Face ([77/22152*a + 22549/2307500, 101/5000], [14199/64600], [77/22152*a + 342208579/1490645000, 193799/807500])>. [<Int(77/22152*a + 342208579/1490645000, 193799/807500)>] is indirectly covered. We obtain a new covered component [<Int(0, 101/5000)>, <Int(60153/369200, 849/5000)>, <Int(77/7752*a + 19/100, 1051/5000)>, <Int(77/22152*a + 342208579/1490645000, 193799/807500)>, <Int(2949/5000, -77/7752*a + 61/100)>, <Int(3151/5000, 235207/369200)>, <Int(3899/5000, 4/5)>, <Int(4101/5000, 4899/5000)>], with overlapping components merged in.
DEBUG: Step 24: Consider the 1d additive <Face ([452201/807500, -77/22152*a + 850307421/1490645000], [14199/64600], [3899/5000, -77/22152*a + 1823451/2307500])>. [<Int(452201/807500, -77/22152*a + 850307421/1490645000)>] is indirectly covered. We obtain a new covered component [<Int(0, 101/5000)>, <Int(60153/369200, 849/5000)>, <Int(77/7752*a + 19/100, 1051/5000)>, <Int(77/22152*a + 342208579/1490645000, 193799/807500)>, <Int(452201/807500, -77/22152*a + 850307421/1490645000)>, <Int(2949/5000, -77/7752*a + 61/100)>, <Int(3151/5000, 235207/369200)>, <Int(3899/5000, 4/5)>, <Int(4101/5000, 4899/5000)>], with overlapping components merged in.
DEBUG: Step 25: Consider the 1d additive <Face ([121206/201875, -77/22152*a + 910529479/1490645000], [14199/64600], [4101/5000, -77/22152*a + 958337/1153750])>. [<Int(121206/201875, -77/22152*a + 910529479/1490645000)>] is indirectly covered. We obtain a new covered component [<Int(0, 101/5000)>, <Int(60153/369200, 849/5000)>, <Int(77/7752*a + 19/100, 1051/5000)>, <Int(77/22152*a + 342208579/1490645000, 193799/807500)>, <Int(452201/807500, -77/22152*a + 850307421/1490645000)>, <Int(2949/5000, -77/7752*a + 61/100)>, <Int(121206/201875, -77/22152*a + 910529479/1490645000)>, <Int(3151/5000, 235207/369200)>, <Int(3899/5000, 4/5)>, <Int(4101/5000, 4899/5000)>], with overlapping components merged in.
DEBUG: Step 26: Consider the 1d additive <Face ([77/22152*a + 1118413/1153750, 4899/5000], [14199/64600], [77/22152*a + 1772631521/1490645000, 242169/201875])>. [<Int(77/22152*a + 281986521/1490645000, 40294/201875)>] is indirectly covered. We obtain a new covered component [<Int(0, 101/5000)>, <Int(60153/369200, 849/5000)>, <Int(77/22152*a + 281986521/1490645000, 40294/201875)>, <Int(77/7752*a + 19/100, 1051/5000)>, <Int(77/22152*a + 342208579/1490645000, 193799/807500)>, <Int(452201/807500, -77/22152*a + 850307421/1490645000)>, <Int(2949/5000, -77/7752*a + 61/100)>, <Int(121206/201875, -77/22152*a + 910529479/1490645000)>, <Int(3151/5000, 235207/369200)>, <Int(3899/5000, 4/5)>, <Int(4101/5000, 4899/5000)>], with overlapping components merged in.
DEBUG: Step 27: Consider the 1d additive <Face ([-77/7752*a + 3151/5000, -1925/298129*a + 3151/5000], [77/7752*a + 19/100], [4101/5000, 77/22152*a + 4101/5000])>. [<Int(-77/7752*a + 3151/5000, -1925/298129*a + 3151/5000)>] is indirectly covered. We obtain a new covered component [<Int(0, 101/5000)>, <Int(60153/369200, 849/5000)>, <Int(77/22152*a + 281986521/1490645000, 40294/201875)>, <Int(77/7752*a + 19/100, 1051/5000)>, <Int(77/22152*a + 342208579/1490645000, 193799/807500)>, <Int(452201/807500, -77/22152*a + 850307421/1490645000)>, <Int(2949/5000, -77/7752*a + 61/100)>, <Int(121206/201875, -77/22152*a + 910529479/1490645000)>, <Int(-77/7752*a + 3151/5000, -1925/298129*a + 3151/5000)>, <Int(3151/5000, 235207/369200)>, <Int(3899/5000, 4/5)>, <Int(4101/5000, 4899/5000)>], with overlapping components merged in.
DEBUG: Step 28: Consider the 1d additive <Face ([-77/22152*a + 4899/5000, 4899/5000], [77/7752*a + 19/100], [1925/298129*a + 5849/5000, 77/7752*a + 5849/5000])>. [<Int(1925/298129*a + 849/5000, 77/7752*a + 849/5000)>] is indirectly covered. We obtain a new covered component [<Int(0, 101/5000)>, <Int(60153/369200, 849/5000)>, <Int(1925/298129*a + 849/5000, 77/7752*a + 849/5000)>, <Int(77/22152*a + 281986521/1490645000, 40294/201875)>, <Int(77/7752*a + 19/100, 1051/5000)>, <Int(77/22152*a + 342208579/1490645000, 193799/807500)>, <Int(452201/807500, -77/22152*a + 850307421/1490645000)>, <Int(2949/5000, -77/7752*a + 61/100)>, <Int(121206/201875, -77/22152*a + 910529479/1490645000)>, <Int(-77/7752*a + 3151/5000, -1925/298129*a + 3151/5000)>, <Int(3151/5000, 235207/369200)>, <Int(3899/5000, 4/5)>, <Int(4101/5000, 4899/5000)>], with overlapping components merged in.
DEBUG: Step 29: Consider the 1d additive <Face ([-77/7752*a + 77/2584, 101/5000], [77/7752*a + 19/100], [14199/64600, 77/7752*a + 1051/5000])>. [<Int(14199/64600, 77/7752*a + 1051/5000)>] is indirectly covered. We obtain a new covered component [<Int(0, 101/5000)>, <Int(60153/369200, 849/5000)>, <Int(1925/298129*a + 849/5000, 77/7752*a + 849/5000)>, <Int(77/22152*a + 281986521/1490645000, 40294/201875)>, <Int(77/7752*a + 19/100, 1051/5000)>, <Int(14199/64600, 77/7752*a + 1051/5000)>, <Int(77/22152*a + 342208579/1490645000, 193799/807500)>, <Int(452201/807500, -77/22152*a + 850307421/1490645000)>, <Int(2949/5000, -77/7752*a + 61/100)>, <Int(121206/201875, -77/22152*a + 910529479/1490645000)>, <Int(-77/7752*a + 3151/5000, -1925/298129*a + 3151/5000)>, <Int(3151/5000, 235207/369200)>, <Int(3899/5000, 4/5)>, <Int(4101/5000, 4899/5000)>], with overlapping components merged in.
DEBUG: Step 30: Consider the 1d additive <Face ([-77/7752*a + 2949/5000, 37481/64600], [77/7752*a + 19/100], [3899/5000, 77/7752*a + 9951/12920])>. [<Int(-77/7752*a + 2949/5000, 37481/64600)>] is indirectly covered. We obtain a new covered component [<Int(0, 101/5000)>, <Int(60153/369200, 849/5000)>, <Int(1925/298129*a + 849/5000, 77/7752*a + 849/5000)>, <Int(77/22152*a + 281986521/1490645000, 40294/201875)>, <Int(77/7752*a + 19/100, 1051/5000)>, <Int(14199/64600, 77/7752*a + 1051/5000)>, <Int(77/22152*a + 342208579/1490645000, 193799/807500)>, <Int(452201/807500, -77/22152*a + 850307421/1490645000)>, <Int(-77/7752*a + 2949/5000, 37481/64600)>, <Int(2949/5000, -77/7752*a + 61/100)>, <Int(121206/201875, -77/22152*a + 910529479/1490645000)>, <Int(-77/7752*a + 3151/5000, -1925/298129*a + 3151/5000)>, <Int(3151/5000, 235207/369200)>, <Int(3899/5000, 4/5)>, <Int(4101/5000, 4899/5000)>], with overlapping components merged in.
DEBUG: Step 31: Consider the 1d additive <Face ([101/5000, 77/22152*a + 5775/298129], [849/5000], [19/100, 77/22152*a + 281986521/1490645000])>. [<Int(19/100, 77/22152*a + 281986521/1490645000)>] is indirectly covered. We obtain a new covered component [<Int(101/5000, 60153/369200)>, <Int(849/5000, 1925/298129*a + 849/5000)>, <Int(77/7752*a + 849/5000, 19/100)>, <Int(19/100, 77/22152*a + 281986521/1490645000)>, <Int(1051/5000, 1925/298129*a + 1051/5000)>, <Int(77/7752*a + 1051/5000, 77/22152*a + 342208579/1490645000)>, <Int(193799/807500, 219/800)>, <Int(269/800, 371/800)>, <Int(421/800, 452201/807500)>, <Int(-77/22152*a + 850307421/1490645000, -77/7752*a + 2949/5000)>, <Int(-1925/298129*a + 2949/5000, 2949/5000)>, <Int(61/100, -77/7752*a + 3151/5000)>, <Int(-1925/298129*a + 3151/5000, 3151/5000)>, <Int(235207/369200, 3899/5000)>, <Int(4/5, 4101/5000)>, <Int(4899/5000, 1)>], with overlapping components merged in.
DEBUG: Step 32: Consider the 1d additive <Face ([-77/22152*a + 910529479/1490645000, 61/100], [849/5000], [-77/22152*a + 1163641/1490645, 3899/5000])>. [<Int(-77/22152*a + 910529479/1490645000, 61/100)>] is indirectly covered. We obtain a new covered component [<Int(101/5000, 60153/369200)>, <Int(849/5000, 1925/298129*a + 849/5000)>, <Int(77/7752*a + 849/5000, 19/100)>, <Int(19/100, 77/22152*a + 281986521/1490645000)>, <Int(1051/5000, 1925/298129*a + 1051/5000)>, <Int(77/7752*a + 1051/5000, 77/22152*a + 342208579/1490645000)>, <Int(193799/807500, 219/800)>, <Int(269/800, 371/800)>, <Int(421/800, 452201/807500)>, <Int(-77/22152*a + 850307421/1490645000, -77/7752*a + 2949/5000)>, <Int(-1925/298129*a + 2949/5000, 2949/5000)>, <Int(-77/22152*a + 910529479/1490645000, 61/100)>, <Int(61/100, -77/7752*a + 3151/5000)>, <Int(-1925/298129*a + 3151/5000, 3151/5000)>, <Int(235207/369200, 3899/5000)>, <Int(4/5, 4101/5000)>, <Int(4899/5000, 1)>], with overlapping components merged in.
DEBUG: Step 33: Consider the 1d additive <Face ([77/7384, 77/7752*a], [19/100], [36999/184600, 77/7752*a + 19/100])>. [<Int(36999/184600, 77/7752*a + 19/100)>] is indirectly covered. We obtain a new covered component [<Int(0, 101/5000)>, <Int(60153/369200, 849/5000)>, <Int(1925/298129*a + 849/5000, 77/7752*a + 849/5000)>, <Int(77/22152*a + 281986521/1490645000, 40294/201875)>, <Int(36999/184600, 77/7752*a + 19/100)>, <Int(77/7752*a + 19/100, 1051/5000)>, <Int(14199/64600, 77/7752*a + 1051/5000)>, <Int(77/22152*a + 342208579/1490645000, 193799/807500)>, <Int(452201/807500, -77/22152*a + 850307421/1490645000)>, <Int(-77/7752*a + 2949/5000, 37481/64600)>, <Int(2949/5000, -77/7752*a + 61/100)>, <Int(121206/201875, -77/22152*a + 910529479/1490645000)>, <Int(-77/7752*a + 3151/5000, -1925/298129*a + 3151/5000)>, <Int(3151/5000, 235207/369200)>, <Int(3899/5000, 4/5)>, <Int(4101/5000, 4899/5000)>], with overlapping components merged in.
DEBUG: Step 34: Consider the 1d additive <Face ([-77/7752*a + 61/100, 110681/184600], [19/100], [-77/7752*a + 4/5, 29151/36920])>. [<Int(-77/7752*a + 61/100, 110681/184600)>] is indirectly covered. We obtain a new covered component [<Int(0, 101/5000)>, <Int(60153/369200, 849/5000)>, <Int(1925/298129*a + 849/5000, 77/7752*a + 849/5000)>, <Int(77/22152*a + 281986521/1490645000, 40294/201875)>, <Int(36999/184600, 77/7752*a + 19/100)>, <Int(77/7752*a + 19/100, 1051/5000)>, <Int(14199/64600, 77/7752*a + 1051/5000)>, <Int(77/22152*a + 342208579/1490645000, 193799/807500)>, <Int(452201/807500, -77/22152*a + 850307421/1490645000)>, <Int(-77/7752*a + 2949/5000, 37481/64600)>, <Int(2949/5000, -77/7752*a + 61/100)>, <Int(-77/7752*a + 61/100, 110681/184600)>, <Int(121206/201875, -77/22152*a + 910529479/1490645000)>, <Int(-77/7752*a + 3151/5000, -1925/298129*a + 3151/5000)>, <Int(3151/5000, 235207/369200)>, <Int(3899/5000, 4/5)>, <Int(4101/5000, 4899/5000)>], with overlapping components merged in.
DEBUG: Step 35: Consider the 1d additive <Face ([110681/184600, 121206/201875], [14199/64600], [1221391/1490645, 4101/5000])>. [<Int(110681/184600, 121206/201875)>] is indirectly covered. We obtain a new covered component [<Int(101/5000, 60153/369200)>, <Int(849/5000, 1925/298129*a + 849/5000)>, <Int(77/7752*a + 849/5000, 19/100)>, <Int(19/100, 77/22152*a + 281986521/1490645000)>, <Int(1051/5000, 1925/298129*a + 1051/5000)>, <Int(77/7752*a + 1051/5000, 77/22152*a + 342208579/1490645000)>, <Int(193799/807500, 219/800)>, <Int(269/800, 371/800)>, <Int(421/800, 452201/807500)>, <Int(-77/22152*a + 850307421/1490645000, -77/7752*a + 2949/5000)>, <Int(-1925/298129*a + 2949/5000, 2949/5000)>, <Int(110681/184600, 121206/201875)>, <Int(-77/22152*a + 910529479/1490645000, 61/100)>, <Int(61/100, -77/7752*a + 3151/5000)>, <Int(-1925/298129*a + 3151/5000, 3151/5000)>, <Int(235207/369200, 3899/5000)>, <Int(4/5, 4101/5000)>, <Int(4899/5000, 1)>], with overlapping components merged in.
DEBUG: Step 36: Consider the 1d additive <Face ([4899/5000, 292354/298129], [14199/64600], [242169/201875, 221599/184600])>. [<Int(40294/201875, 36999/184600)>] is indirectly covered. We obtain a new covered component [<Int(101/5000, 60153/369200)>, <Int(849/5000, 1925/298129*a + 849/5000)>, <Int(77/7752*a + 849/5000, 19/100)>, <Int(19/100, 77/22152*a + 281986521/1490645000)>, <Int(40294/201875, 36999/184600)>, <Int(1051/5000, 1925/298129*a + 1051/5000)>, <Int(77/7752*a + 1051/5000, 77/22152*a + 342208579/1490645000)>, <Int(193799/807500, 219/800)>, <Int(269/800, 371/800)>, <Int(421/800, 452201/807500)>, <Int(-77/22152*a + 850307421/1490645000, -77/7752*a + 2949/5000)>, <Int(-1925/298129*a + 2949/5000, 2949/5000)>, <Int(110681/184600, 121206/201875)>, <Int(-77/22152*a + 910529479/1490645000, 61/100)>, <Int(61/100, -77/7752*a + 3151/5000)>, <Int(-1925/298129*a + 3151/5000, 3151/5000)>, <Int(235207/369200, 3899/5000)>, <Int(4/5, 4101/5000)>, <Int(4899/5000, 1)>], with overlapping components merged in.
DEBUG: Step 37: Consider the 1d additive <Face ([-77/22152*a + 101/5000, -77/7752*a + 77/2584], [77/7752*a + 19/100], [1925/298129*a + 1051/5000, 14199/64600])>. [<Int(1925/298129*a + 1051/5000, 14199/64600)>] is indirectly covered. We obtain a new covered component [<Int(0, 101/5000)>, <Int(60153/369200, 849/5000)>, <Int(1925/298129*a + 849/5000, 77/7752*a + 849/5000)>, <Int(77/22152*a + 281986521/1490645000, 40294/201875)>, <Int(36999/184600, 77/7752*a + 19/100)>, <Int(77/7752*a + 19/100, 1051/5000)>, <Int(1925/298129*a + 1051/5000, 14199/64600)>, <Int(14199/64600, 77/7752*a + 1051/5000)>, <Int(77/22152*a + 342208579/1490645000, 193799/807500)>, <Int(452201/807500, -77/22152*a + 850307421/1490645000)>, <Int(-77/7752*a + 2949/5000, 37481/64600)>, <Int(2949/5000, -77/7752*a + 61/100)>, <Int(-77/7752*a + 61/100, 110681/184600)>, <Int(121206/201875, -77/22152*a + 910529479/1490645000)>, <Int(-77/7752*a + 3151/5000, -1925/298129*a + 3151/5000)>, <Int(3151/5000, 235207/369200)>, <Int(3899/5000, 4/5)>, <Int(4101/5000, 4899/5000)>], with overlapping components merged in.
DEBUG: Step 38: Consider the 1d additive <Face ([37481/64600, -1925/298129*a + 2949/5000], [77/7752*a + 19/100], [77/7752*a + 9951/12920, 77/22152*a + 3899/5000])>. [<Int(37481/64600, -1925/298129*a + 2949/5000)>] is indirectly covered. We obtain a new covered component [<Int(0, 101/5000)>, <Int(60153/369200, 849/5000)>, <Int(1925/298129*a + 849/5000, 77/7752*a + 849/5000)>, <Int(77/22152*a + 281986521/1490645000, 40294/201875)>, <Int(36999/184600, 77/7752*a + 19/100)>, <Int(77/7752*a + 19/100, 1051/5000)>, <Int(1925/298129*a + 1051/5000, 14199/64600)>, <Int(14199/64600, 77/7752*a + 1051/5000)>, <Int(77/22152*a + 342208579/1490645000, 193799/807500)>, <Int(452201/807500, -77/22152*a + 850307421/1490645000)>, <Int(-77/7752*a + 2949/5000, 37481/64600)>, <Int(37481/64600, -1925/298129*a + 2949/5000)>, <Int(2949/5000, -77/7752*a + 61/100)>, <Int(-77/7752*a + 61/100, 110681/184600)>, <Int(121206/201875, -77/22152*a + 910529479/1490645000)>, <Int(-77/7752*a + 3151/5000, -1925/298129*a + 3151/5000)>, <Int(3151/5000, 235207/369200)>, <Int(3899/5000, 4/5)>, <Int(4101/5000, 4899/5000)>], with overlapping components merged in.
INFO: Completing 11 functional directed moves and 2 covered components...
INFO: Completing 26 functional directed moves and 2 covered components...
INFO: New dense move from strip lemma: [<Int(219/800, 269/800)>, <Int(371/800, 421/800)>]
INFO: Completing 0 functional directed moves and 3 covered components...
INFO: Completion finished.  Found 0 directed moves and 3 covered components.
INFO: All intervals are covered (or connected-to-covered). 3 components.
DEBUG: The covered components are [[<Int(101/5000, 60153/369200)>, <Int(849/5000, 1925/298129*a + 849/5000)>, <Int(77/7752*a + 849/5000, 19/100)>, <Int(19/100, 77/22152*a + 281986521/1490645000)>, <Int(40294/201875, 36999/184600)>, <Int(1051/5000, 1925/298129*a + 1051/5000)>, <Int(77/7752*a + 1051/5000, 77/22152*a + 342208579/1490645000)>, <Int(193799/807500, 219/800)>, <Int(269/800, 371/800)>, <Int(421/800, 452201/807500)>, <Int(-77/22152*a + 850307421/1490645000, -77/7752*a + 2949/5000)>, <Int(-1925/298129*a + 2949/5000, 2949/5000)>, <Int(110681/184600, 121206/201875)>, <Int(-77/22152*a + 910529479/1490645000, 61/100)>, <Int(61/100, -77/7752*a + 3151/5000)>, <Int(-1925/298129*a + 3151/5000, 3151/5000)>, <Int(235207/369200, 3899/5000)>, <Int(4/5, 4101/5000)>, <Int(4899/5000, 1)>], [<Int(0, 101/5000)>, <Int(60153/369200, 849/5000)>, <Int(1925/298129*a + 849/5000, 77/7752*a + 849/5000)>, <Int(77/22152*a + 281986521/1490645000, 40294/201875)>, <Int(36999/184600, 77/7752*a + 19/100)>, <Int(77/7752*a + 19/100, 1051/5000)>, <Int(1925/298129*a + 1051/5000, 14199/64600)>, <Int(14199/64600, 77/7752*a + 1051/5000)>, <Int(77/22152*a + 342208579/1490645000, 193799/807500)>, <Int(452201/807500, -77/22152*a + 850307421/1490645000)>, <Int(-77/7752*a + 2949/5000, 37481/64600)>, <Int(37481/64600, -1925/298129*a + 2949/5000)>, <Int(2949/5000, -77/7752*a + 61/100)>, <Int(-77/7752*a + 61/100, 110681/184600)>, <Int(121206/201875, -77/22152*a + 910529479/1490645000)>, <Int(-77/7752*a + 3151/5000, -1925/298129*a + 3151/5000)>, <Int(3151/5000, 235207/369200)>, <Int(3899/5000, 4/5)>, <Int(4101/5000, 4899/5000)>], [<Int(219/800, 269/800)>, <Int(371/800, 421/800)>]].
DEBUG: Let v in R^43. The i-th entry of v represents the slope parameter on the i-th component of [[<Int(101/5000, 60153/369200)>, <Int(849/5000, 1925/298129*a + 849/5000)>, <Int(77/7752*a + 849/5000, 19/100)>, <Int(19/100, 77/22152*a + 281986521/1490645000)>, <Int(40294/201875, 36999/184600)>, <Int(1051/5000, 1925/298129*a + 1051/5000)>, <Int(77/7752*a + 1051/5000, 77/22152*a + 342208579/1490645000)>, <Int(193799/807500, 219/800)>, <Int(269/800, 371/800)>, <Int(421/800, 452201/807500)>, <Int(-77/22152*a + 850307421/1490645000, -77/7752*a + 2949/5000)>, <Int(-1925/298129*a + 2949/5000, 2949/5000)>, <Int(110681/184600, 121206/201875)>, <Int(-77/22152*a + 910529479/1490645000, 61/100)>, <Int(61/100, -77/7752*a + 3151/5000)>, <Int(-1925/298129*a + 3151/5000, 3151/5000)>, <Int(235207/369200, 3899/5000)>, <Int(4/5, 4101/5000)>, <Int(4899/5000, 1)>], [<Int(0, 101/5000)>, <Int(60153/369200, 849/5000)>, <Int(1925/298129*a + 849/5000, 77/7752*a + 849/5000)>, <Int(77/22152*a + 281986521/1490645000, 40294/201875)>, <Int(36999/184600, 77/7752*a + 19/100)>, <Int(77/7752*a + 19/100, 1051/5000)>, <Int(1925/298129*a + 1051/5000, 14199/64600)>, <Int(14199/64600, 77/7752*a + 1051/5000)>, <Int(77/22152*a + 342208579/1490645000, 193799/807500)>, <Int(452201/807500, -77/22152*a + 850307421/1490645000)>, <Int(-77/7752*a + 2949/5000, 37481/64600)>, <Int(37481/64600, -1925/298129*a + 2949/5000)>, <Int(2949/5000, -77/7752*a + 61/100)>, <Int(-77/7752*a + 61/100, 110681/184600)>, <Int(121206/201875, -77/22152*a + 910529479/1490645000)>, <Int(-77/7752*a + 3151/5000, -1925/298129*a + 3151/5000)>, <Int(3151/5000, 235207/369200)>, <Int(3899/5000, 4/5)>, <Int(4101/5000, 4899/5000)>], [<Int(219/800, 269/800)>, <Int(371/800, 421/800)>]] if i<=3, or the function value jump parameter at breakpoint if i>3. (The symmetry condition is considered so as to reduce the number of jump parameters). Set up the symbolic function sym: [0,1] -> R^43, so that pert(x) = sym(x) * v. The symbolic function sym is <FastPiecewise with 81 parts, 
 <Int{0}> <FastLinearFunction 0>
 <Int(0, 101/5000)> <FastLinearFunction ((0,1,0,0,0,0,0,0,0,0,0,0,0,0,0,0,0,0,0,0,0,0,0,0,0,0,0,0,0,0,0,0,0,0,0,0,0,0,0,0,0,0,0))*x + ((0,0,0,1,0,0,0,0,0,0,0,0,0,0,0,0,0,0,0,0,0,0,0,0,0,0,0,0,0,0,0,0,0,0,0,0,0,0,0,0,0,0,0))>
 <Int{101/5000}> <FastLinearFunction ((0,101/5000,0,1,1,0,0,0,0,0,0,0,0,0,0,0,0,0,0,0,0,0,0,0,0,0,0,0,0,0,0,0,0,0,0,0,0,0,0,0,0,0,0))> ......
 <Int{1}> <FastLinearFunction ((3/5,11/40,1/8,2,2,2,2,2,2,2,2,2,2,2,2,2,2,2,2,2,2,2,2,2,2,2,2,2,2,2,2,2,2,2,2,2,2,2,2,2,2,2,2))>>.
DEBUG: Condition pert(f) = 0 gives the equation (1399/2500, 577/5000, 1/8, 2, 2, 2, 2, 2, 2, 2, 2, 2, 2, 2, 2, 2, 2, 2, 2, 2, 2, 2, 2, 2, 2, 2, 2, 2, 2, 2, 2, 2, 2, 2, 2, 2, 2, 2, 2, 2, 0, 0, 0) * v = 0.
DEBUG: Condition pert(1) = 0 gives the equation (3/5, 11/40, 1/8, 2, 2, 2, 2, 2, 2, 2, 2, 2, 2, 2, 2, 2, 2, 2, 2, 2, 2, 2, 2, 2, 2, 2, 2, 2, 2, 2, 2, 2, 2, 2, 2, 2, 2, 2, 2, 2, 2, 2, 2) * v = 0.
DEBUG: Condition pert(421/800+) + pert(13941/258400-) = pert(37481/64600) gives the equation (1925/298129*a - 5775/298129, -1925/298129*a + 5775/298129, 0, 1, 1, 1, 0, 0, 0, 0, 0, 0, 0, 0, 0, 0, 0, 0, 0, 0, 0, 0, 0, 0, 0, 0, 0, 0, 0, -1, -1, -1, -1, -1, -1, -1, 0, 0, 0, 0, 0, 0, 0) * v = 0.
DEBUG: Condition pert(421/800) + pert(5071/20000-) = pert(3899/5000-) gives the equation (-101/5000, 101/5000, 0, 1, 1, 1, 0, 0, 0, 0, 0, 0, 0, 0, 0, 0, 0, 0, 0, 0, 0, 0, 0, 0, 0, 0, 0, 0, 0, 0, 0, 0, 0, 0, 0, 0, -1, 0, 0, 0, 0, 0, 0) * v = 0.
DEBUG: Condition pert(19/100) + pert(1925/298129*a + 101/5000+) = pert(1925/298129*a + 1051/5000+) gives the equation (-77/22152*a, 77/22152*a, 0, 1, 1, 1, 0, 0, 0, 0, 0, 0, 0, 0, 0, -1, -1, -1, -1, -1, -1, -1, -1, -1, -1, -1, -1, -1, 0, 0, 0, 0, 0, 0, 0, 0, 0, 0, 0, 0, 0, 0, 0) * v = 0.
DEBUG: Condition pert(101/5000+) + pert(849/5000) = pert(19/100+) gives the equation (77/22152*a - 101/5000, -77/22152*a + 101/5000, 0, 1, 1, 1, 0, 0, 0, -1, -1, -1, -1, -1, -1, -1, 0, 0, 0, 0, 0, 0, 0, 0, 0, 0, 0, 0, 0, 0, 0, 0, 0, 0, 0, 0, 0, 0, 0, 0, 0, 0, 0) * v = 0.
DEBUG: Condition pert(101/5000-) + pert(4899/5000+) = pert(1+) gives the equation (2899/5000, 369/1250, 1/8, 2, 2, 2, 2, 2, 2, 2, 2, 2, 2, 2, 2, 2, 2, 2, 2, 2, 2, 2, 2, 2, 2, 2, 2, 2, 2, 2, 2, 2, 2, 2, 2, 2, 2, 2, 2, 2, 1, 2, 2) * v = 0.
DEBUG: Condition pert(101/5000-) + pert(1+) = pert(5101/5000) gives the equation (0, 0, 0, 1, -1, 0, 0, 0, 0, 0, 0, 0, 0, 0, 0, 0, 0, 0, 0, 0, 0, 0, 0, 0, 0, 0, 0, 0, 0, 0, 0, 0, 0, 0, 0, 0, 0, 0, 0, 0, 0, 0, 0) * v = 0.
DEBUG: Condition pert(269/800) + pert(63037/258400) = pert(37481/64600) gives the equation (1925/298129*a + 67129/11925160, -1925/298129*a + 1356387/23850320, -1/16, 1, 1, 1, 1, 1, 1, 1, 1, 1, 1, 1, 1, 1, 1, 1, 1, 1, 1, 1, 1, 1, 1, 1, 1, 1, 1, 0, 0, 0, 0, 0, 0, 0, -1, -1, -1, -2, 0, 0, 0) * v = 0.
DEBUG: Condition pert(14199/64600) + pert(-77/22152*a + 910529479/1490645000-) = pert(-77/22152*a + 958337/1153750-) gives the equation (1925/298129*a - 5775/298129, -1925/298129*a + 5775/298129, 0, 0, 0, 0, 0, 0, 0, 0, 0, 0, 0, 0, 0, 0, 0, 0, 1, 1, 1, 1, 1, 1, 1, 1, 1, 1, 1, 0, 0, 0, 0, 0, 0, 0, 0, 0, 0, 0, -1, -1, -1) * v = 0.
DEBUG: Condition pert(61/100+) + pert(4899/5000-) = pert(7949/5000) gives the equation (77/22152*a + 2899/5000, -77/22152*a + 369/1250, 1/8, 2, 2, 2, 2, 2, 2, 2, 2, 2, 2, 2, 3, 3, 3, 3, 3, 3, 3, 3, 3, 3, 3, 2, 2, 2, 2, 2, 2, 2, 2, 2, 2, 2, 2, 2, 2, 2, 1, 1, 1) * v = 0.
DEBUG: Condition pert(19/100) + pert(2949/5000-) = pert(3899/5000-) gives the equation (-77/22152*a, 77/22152*a, 0, 1, 1, 1, 0, 0, 0, 0, 0, 0, 0, 0, 0, -1, -1, -1, -1, -1, -1, -1, -1, -1, -1, -1, 0, 0, 0, 0, 0, 0, 0, 0, 0, 0, 0, 0, 0, 0, 0, 0, 0) * v = 0.
DEBUG: Condition pert(14199/64600) + pert(-77/22152*a + 850307421/1490645000-) = pert(-77/22152*a + 1823451/2307500-) gives the equation (1925/298129*a - 5775/298129, -1925/298129*a + 5775/298129, 0, 1, 0, 0, 0, 0, 0, 0, 0, 0, 0, 0, 0, 0, 0, 0, 0, 0, 0, 0, 0, 0, 0, 0, 0, 0, 0, -1, -1, -1, -1, -1, 0, 0, 0, 0, 0, 0, 0, 0, 0) * v = 0.
DEBUG: Condition pert(-77/22152*a + 850307421/1490645000+) + pert(77/22152*a + 22549/2307500-) = pert(37481/64600) gives the equation (1925/298129*a - 5775/298129, -1925/298129*a + 5775/298129, 0, 1, 0, 0, 0, 0, 0, 0, 0, 0, 0, 0, 0, 0, 0, 0, 0, 0, 0, 0, 0, 0, 0, 0, 0, 0, 0, -1, -1, -1, 0, 0, 0, 0, 0, 0, 0, 0, 0, 0, 0) * v = 0.
DEBUG: Condition pert(-77/22152*a + 910529479/1490645000) + pert(77/22152*a + 1118413/1153750) = pert(102081/64600) gives the equation (1925/298129*a + 865512/1490645, -1925/298129*a + 3510419/11925160, 1/8, 2, 2, 2, 2, 2, 2, 2, 2, 2, 2, 2, 2, 2, 2, 3, 3, 3, 3, 3, 3, 3, 3, 3, 3, 3, 3, 2, 2, 2, 2, 2, 2, 2, 2, 2, 2, 2, 1, 1, 1) * v = 0.
DEBUG: Condition pert(14199/64600) + pert(292354/298129) = pert(221599/184600) gives the equation (1925/298129*a + 865512/1490645, -1925/298129*a + 3510419/11925160, 1/8, 2, 2, 2, 2, 2, 2, 2, 2, 2, 2, 2, 2, 2, 2, 2, 2, 2, 2, 3, 3, 3, 3, 3, 3, 3, 3, 2, 2, 2, 2, 2, 2, 2, 2, 2, 2, 2, 1, 2, 2) * v = 0.
DEBUG: Condition pert(269/800) + pert(8871/20000-) = pert(3899/5000-) gives the equation (-101/5000, 101/5000, 0, 1, 1, 1, 0, 0, 0, 0, 0, 0, 0, 0, 0, 0, 0, 0, 0, 0, 0, 0, 0, 0, 0, 0, 0, 0, 0, 0, 0, 0, 0, 0, 0, 0, 0, 0, 0, -1, 0, 0, 0) * v = 0.
DEBUG: Condition pert(101/5000+) + pert(101/5000+) = pert(101/2500+) gives the equation (-101/5000, 101/5000, 0, 1, 1, 1, 0, 0, 0, 0, 0, 0, 0, 0, 0, 0, 0, 0, 0, 0, 0, 0, 0, 0, 0, 0, 0, 0, 0, 0, 0, 0, 0, 0, 0, 0, 0, 0, 0, 0, 0, 0, 0) * v = 0.
DEBUG: Condition pert(14199/64600) + pert(549427/1615000+) = pert(452201/807500+) gives the equation (1925/298129*a + 67129/11925160, -1925/298129*a + 1356387/23850320, -1/16, 1, 1, 1, 1, 1, 1, 1, 1, 1, 1, 1, 1, 1, 1, 1, 1, 1, 1, 1, 1, 1, 1, 1, 1, 1, 1, 0, 0, 0, 0, 0, -1, -1, -1, -1, -1, -1, 0, 0, 0) * v = 0.
DEBUG: Condition pert(19/100) + pert(-1925/298129*a + 3151/5000+) = pert(-1925/298129*a + 4101/5000+) gives the equation (-77/22152*a, 77/22152*a, 0, 0, 0, 0, 0, 0, 0, 0, 1, 1, 1, 1, 1, 0, 0, 0, 0, 0, 0, 0, 0, 0, 0, 0, 0, 0, 0, 0, 0, 0, 0, 0, 0, 0, 0, 0, 0, 0, -1, 0, 0) * v = 0.
DEBUG: Condition pert(-77/22152*a + 910529479/1490645000+) + pert(77/22152*a + 1118413/1153750-) = pert(102081/64600) gives the equation (1925/298129*a + 865512/1490645, -1925/298129*a + 3510419/11925160, 1/8, 2, 2, 2, 2, 2, 2, 2, 2, 2, 2, 2, 2, 2, 3, 3, 3, 3, 3, 3, 3, 3, 3, 3, 3, 3, 3, 2, 2, 2, 2, 2, 2, 2, 2, 2, 2, 2, 1, 1, 1) * v = 0.
DEBUG: Condition pert(77/7752*a + 19/100) + pert(-77/22152*a + 4899/5000+) = pert(1925/298129*a + 5849/5000+) gives the equation (-1925/298129*a + 3/5, 1925/298129*a + 11/40, 1/8, 2, 2, 2, 2, 2, 2, 2, 2, 2, 3, 3, 3, 3, 3, 3, 3, 3, 3, 3, 3, 2, 2, 2, 2, 2, 2, 2, 2, 2, 2, 2, 2, 2, 2, 2, 2, 2, 1, 1, 1) * v = 0.
DEBUG: Condition pert(219/800+) + pert(269/800-) = pert(61/100) gives the equation (-77/22152*a + 327/9230, 77/22152*a + 1999/73840, -1/16, 1, 1, 1, 1, 1, 1, 1, 1, 1, 1, 1, 1, 0, 0, 0, 0, 0, 0, 0, 0, 0, 0, 0, 0, 0, 0, 0, 0, 0, 0, 0, 0, 0, 0, 0, -2, -2, 0, 0, 0) * v = 0.
DEBUG: Condition pert(14199/64600) + pert(110681/184600-) = pert(1221391/1490645-) gives the equation (1925/298129*a - 5775/298129, -1925/298129*a + 5775/298129, 0, 0, 0, 0, 0, 0, 0, 0, 0, 0, 0, 0, 0, 0, 0, 0, 0, 0, 0, 0, 1, 1, 1, 1, 1, 1, 1, 0, 0, 0, 0, 0, 0, 0, 0, 0, 0, 0, -1, 0, 0) * v = 0.
DEBUG: Condition pert(101/5000-) + pert(1317379/9230000+) = pert(60153/369200) gives the equation (-101/5000, 101/5000, 0, 1, 0, 0, -1, 0, 0, 0, 0, 0, 0, 0, 0, 0, 0, 0, 0, 0, 0, 0, 0, 0, 0, 0, 0, 0, 0, 0, 0, 0, 0, 0, 0, 0, 0, 0, 0, 0, 0, 0, 0) * v = 0.
DEBUG: Condition pert(77/7752*a + 19/100) + pert(37481/64600-) = pert(77/7752*a + 9951/12920-) gives the equation (-1925/298129*a, 1925/298129*a, 0, 1, 0, 0, 0, 0, 0, 0, 0, 0, 0, 0, 0, 0, 0, 0, 0, 0, 0, 0, 0, -1, -1, -1, -1, -1, -1, -1, 0, 0, 0, 0, 0, 0, 0, 0, 0, 0, 0, 0, 0) * v = 0.
DEBUG: Condition pert(19/100) + pert(2949/5000+) = pert(3899/5000+) gives the equation (-77/22152*a, 77/22152*a, 0, 1, 0, 0, 0, 0, 0, 0, 0, 0, 0, 0, 0, -1, -1, -1, -1, -1, -1, -1, -1, -1, 0, 0, 0, 0, 0, 0, 0, 0, 0, 0, 0, 0, 0, 0, 0, 0, 0, 0, 0) * v = 0.
DEBUG: Condition pert(-1925/298129*a + 3151/5000-) + pert(1925/298129*a + 4899/5000+) = pert(161/100) gives the equation (-77/22152*a + 3/5, 77/22152*a + 11/40, 1/8, 2, 2, 2, 2, 2, 2, 2, 2, 2, 3, 3, 3, 2, 2, 2, 2, 2, 2, 2, 2, 2, 2, 2, 2, 2, 2, 2, 2, 2, 2, 2, 2, 2, 2, 2, 2, 2, 1, 2, 2) * v = 0.
DEBUG: Condition pert(19/100) + pert(110681/184600) = pert(29151/36920) gives the equation (-77/22152*a, 77/22152*a, 0, 1, 0, 0, 0, 0, 0, 0, 0, 0, 0, 0, 0, -1, -1, -1, -1, -1, -1, 0, 0, 0, 0, 0, 0, 0, 0, 0, 0, 0, 0, 0, 0, 0, 0, 0, 0, 0, 0, 0, 0) * v = 0.
DEBUG: Condition pert(4101/5000) + pert(1+) = pert(9101/5000+) gives the equation (0, 0, 0, 1, 0, 0, 0, 0, 0, 0, 0, 0, 0, 0, 0, 0, 0, 0, 0, 0, 0, 0, 0, 0, 0, 0, 0, 0, 0, 0, 0, 0, 0, 0, 0, 0, 0, 0, 0, 0, 0, 0, -1) * v = 0.
DEBUG: Condition pert(19/100) + pert(235207/369200) = pert(61071/73840) gives the equation (-77/22152*a, 77/22152*a, 0, 0, 0, 0, 0, 1, 1, 1, 1, 1, 1, 1, 1, 0, 0, 0, 0, 0, 0, 0, 0, 0, 0, 0, 0, 0, 0, 0, 0, 0, 0, 0, 0, 0, 0, 0, 0, 0, -1, -1, -1) * v = 0.
DEBUG: Condition pert(19/100) + pert(4899/5000-) = pert(5849/5000-) gives the equation (-77/22152*a + 3/5, 77/22152*a + 11/40, 1/8, 2, 2, 2, 2, 2, 3, 3, 3, 3, 3, 3, 3, 2, 2, 2, 2, 2, 2, 2, 2, 2, 2, 2, 2, 2, 2, 2, 2, 2, 2, 2, 2, 2, 2, 2, 2, 2, 1, 1, 1) * v = 0.
DEBUG: Condition pert(61/100+) + pert(-77/7752*a + 1-) = pert(-77/7752*a + 161/100) gives the equation (-1925/298129*a + 3/5, 1925/298129*a + 11/40, 1/8, 2, 2, 2, 2, 2, 2, 2, 2, 2, 2, 2, 3, 3, 3, 3, 3, 3, 3, 3, 3, 2, 2, 2, 2, 2, 2, 2, 2, 2, 2, 2, 2, 2, 2, 2, 2, 2, 1, 2, 2) * v = 0.
DEBUG: Condition pert(421/800) + pert(218033/6460000) = pert(452201/807500) gives the equation (-101/5000, 101/5000, 0, 1, 1, 1, 0, 0, 0, 0, 0, 0, 0, 0, 0, 0, 0, 0, 0, 0, 0, 0, 0, 0, 0, 0, 0, 0, 0, 0, 0, 0, 0, 0, 0, -1, -1, 0, 0, 0, 0, 0, 0) * v = 0.
DEBUG: Condition pert(77/7752*a + 19/100) + pert(-1925/298129*a + 1489408971/1490645000+) = pert(77/22152*a + 1772631521/1490645000+) gives the equation (-1925/298129*a + 3/5, 1925/298129*a + 11/40, 1/8, 2, 2, 2, 2, 2, 2, 2, 2, 2, 2, 2, 2, 2, 2, 2, 3, 3, 3, 3, 3, 2, 2, 2, 2, 2, 2, 2, 2, 2, 2, 2, 2, 2, 2, 2, 2, 2, 1, 2, 2) * v = 0.
DEBUG: Condition pert(19/100) + pert(-77/7752*a + 2949/5000) = pert(-77/7752*a + 3899/5000) gives the equation (0, 0, 0, 1, 1, 1, 0, 0, 0, 0, 0, 0, 0, 0, 0, -1, -1, -1, -1, -1, -1, -1, -1, -1, -1, -1, -1, -1, -1, -1, -1, 0, 0, 0, 0, 0, 0, 0, 0, 0, 0, 0, 0) * v = 0.
DEBUG: Condition pert(-77/22152*a + 850307421/1490645000) + pert(-1925/298129*a + 58986029/1490645000) = pert(-77/7752*a + 61/100) gives the equation (-1925/298129*a, 1925/298129*a, 0, 1, 1, 1, 0, 0, 0, 0, 0, 0, 0, 0, 0, 0, 0, 0, 0, 0, 0, 0, 0, -1, -1, -1, -1, -1, -1, -1, -1, -1, -1, 0, 0, 0, 0, 0, 0, 0, 0, 0, 0) * v = 0.
DEBUG: Condition pert(77/7752*a + 19/100) + pert(-1925/298129*a + 2949/5000) = pert(77/22152*a + 3899/5000) gives the equation (-1925/298129*a, 1925/298129*a, 0, 1, 0, 0, 0, 0, 0, 0, 0, 0, 0, 0, 0, 0, 0, 0, 0, 0, 0, 0, 0, -1, -1, -1, -1, 0, 0, 0, 0, 0, 0, 0, 0, 0, 0, 0, 0, 0, 0, 0, 0) * v = 0.
DEBUG: Condition pert(77/7752*a + 19/100) + pert(-77/7752*a + 269/800-) = pert(421/800-) gives the equation (327/9230, 77/7752*a + 1999/73840, -77/7752*a - 1/16, 1, 1, 1, 1, 1, 1, 1, 1, 1, 1, 1, 1, 1, 1, 1, 1, 1, 1, 1, 1, 0, 0, 0, 0, 0, 0, 0, 0, 0, 0, 0, 0, 0, 0, 0, -2, -2, 0, 0, 0) * v = 0.
DEBUG: Condition pert(110681/184600+) + pert(77/7384-) = pert(61/100) gives the equation (-77/22152*a, 77/22152*a, 0, 1, 0, 0, 0, 0, 0, 0, 0, 0, 0, 0, 0, -1, -1, -1, -1, -1, 0, 0, 0, 0, 0, 0, 0, 0, 0, 0, 0, 0, 0, 0, 0, 0, 0, 0, 0, 0, 0, 0, 0) * v = 0.
DEBUG: Condition pert(19/100) + pert(1925/298129*a + 4899/5000) = pert(1925/298129*a + 5849/5000) gives the equation (-77/22152*a + 3/5, 77/22152*a + 11/40, 1/8, 2, 2, 2, 2, 2, 2, 2, 2, 3, 3, 3, 3, 2, 2, 2, 2, 2, 2, 2, 2, 2, 2, 2, 2, 2, 2, 2, 2, 2, 2, 2, 2, 2, 2, 2, 2, 2, 1, 2, 2) * v = 0.
DEBUG: Condition pert(19/100) + pert(7751/807500) = pert(40294/201875) gives the equation (-77/22152*a + 1236029/1490645000, 77/22152*a - 1236029/1490645000, 0, 1, 0, 0, 0, 0, 0, 0, 0, 0, 0, 0, 0, -1, -1, -1, -1, 0, 0, 0, 0, 0, 0, 0, 0, 0, 0, 0, 0, 0, 0, 0, 0, 0, 0, 0, 0, 0, 0, 0, 0) * v = 0.
DEBUG: Condition pert(19/100) + pert(-77/7752*a + 3151/5000) = pert(-77/7752*a + 4101/5000) gives the equation (0, 0, 0, 0, 0, 0, 0, 0, 0, 0, 0, 0, 0, 1, 1, 0, 0, 0, 0, 0, 0, 0, 0, 0, 0, 0, 0, 0, 0, 0, 0, 0, 0, 0, 0, 0, 0, 0, 0, 0, -1, 0, 0) * v = 0.
DEBUG: Solve the linear system of equations: 43 x 43 dense matrix over Real Number Field in `a` as the root of the defining polynomial y^2 - 2 near 1.414213562373095? * v = 0.
INFO: Finite dimensional test: Solution space has dimension 0. Thus the function is extreme.
True
\end{lstlisting}%

\section*{Acknowledgment} 
The authors wish to express their gratitude to the anonymous referees, whose
comments led to a much improved paper.

{\small
\providecommand\ISBN{ISBN }
\bibliographystyle{../amsabbrvurl}
\bibliography{../bib/MLFCB_bib}
}
\end{document}